\documentclass[11pt]{amsart}
\usepackage{enumitem,amssymb,cmap,stmaryrd}

\newtheorem*{COR}{Corollary}
\newtheorem*{MAIN}{Main Theorem}
\newtheorem{theorem}{Theorem}[section]
\newtheorem{prop}[theorem]{Proposition}
\newtheorem{claim}{Claim}[theorem]
\newtheorem{subclaim}{Subclaim}[claim]
\newtheorem{lemma}[theorem]{Lemma}
\newtheorem{cor}[theorem]{Corollary}
\newtheorem{fact}[theorem]{Fact}
\newtheorem*{prikry}{Prikry Property}
\newtheorem*{cpp}{Complete Prikry Property}
\newtheorem*{zlinked}{Linked${}_0$ Property}

\theoremstyle{definition}
\newtheorem{definition}[theorem]{Definition}
\newtheorem{notation}[theorem]{Notation}
\newtheorem{conv}[theorem]{Convention}
\newtheorem{goal}[theorem]{Goal}
\newtheorem{example}[theorem]{Example}
\newtheorem*{setup}{Setup~\thesection}

\theoremstyle{remark}
\newtheorem{remark}[theorem]{Remark}

\newcounter{condition}

\def\br{\blacktriangleright}
\def\sq{\sqsubseteq}
\def\s{\subseteq}
\def\lh{\ell}
\def\forces{\Vdash}

\newenvironment{blk}[1]{\medskip\noindent\textbf{Building Block~#1.\relax}~}{}
\newcommand\blkref[1]{Building Block~#1}

\providecommand{\myceil}[2]{\left\lceil #1 \right\rceil^{#2} }
\newcommand*\axiomfont[1]{\textsf{\textup{#1}}}
\newcommand{\one}{\mathop{1\hskip-3pt {\rm l}}} 			
\newcommand{\cone}[1]{{\mathbb P\mathrel{\downarrow}#1}}
\newcommand{\conea}[1]{{\mathbb A\mathrel{\downarrow}#1}}
\newcommand{\conez}[2]{{\mathbb P_{#1}\mathrel{\downarrow}#2}}
\newcommand{\fork}[2][]{{\pitchfork_{#1}}(#2)}  			

\def\pI{\textrm{\bf I}}	
\def\pII{\textrm{\bf II}}
\newcommand{\sch}{\axiomfont{SCH}}

\renewcommand{\restriction}{\mathbin\upharpoonright}
\renewcommand{\mid}{\mathrel{|}\allowbreak}

\newcommand\z[1]{\mathring{#1}}
\DeclareMathOperator{\tp}{\mathsf{tp}}
\DeclareMathOperator{\mtp}{mtp}
\DeclareMathOperator{\CONE}{cone}

\DeclareMathOperator{\refl}{Refl}
\DeclareMathOperator{\ord}{Ord}
\DeclareMathOperator{\nacc}{nacc}
\DeclareMathOperator{\dom}{dom}
\DeclareMathOperator{\CPP}{CPP}
\DeclareMathOperator{\otp}{otp}
\DeclareMathOperator{\acc}{acc}
\DeclareMathOperator{\rng}{Im}
\DeclareMathOperator{\tr}{Tr}
\DeclareMathOperator{\cf}{cf}
\DeclareMathOperator{\cl}{cl}

\title[Sigma-Prikry forcing II]{Sigma-Prikry forcing II:\\Iteration Scheme}

\author{Alejandro Poveda}
\address{Einstein Institute of Mathematics, Hebrew University of Jerusalem, Edmond J. Safra Campus, Givat-Ram 91904, Israel. }
\thanks{Poveda was partially supported by the Spanish
Government under grant MTM2017-86777-P, by Generalitat de Catalunya (Catalan Government) under grant SGR 270-2017 and by MECD Grant FPU15/00026.}

\author{Assaf Rinot}
\address{Department of Mathematics, Bar-Ilan University, Ramat-Gan 5290002, Israel.}
\urladdr{http://www.assafrinot.com}
\thanks{Rinot was partially supported by the European Research Council (grant agreement ERC-2018-StG 802756) and by the Israel Science Foundation (grant agreement 2066/18).}

\author{Dima Sinapova}
\address{Department of Mathematics, Statistics, and Computer Science\\
         University of Illinois at Chicago\\
         Chicago, IL 60607-7045\\
         USA}
\urladdr{https://homepages.math.uic.edu/~sinapova/}
\thanks{Sinapova was partially supported by the National Science Foundation, Career-1454945.}

\begin{document}
\date{January 16, 2022}

\begin{abstract}  In Part I of this series \cite{partI}, we introduced a class of notions of forcing which we call $\Sigma$-Prikry,
and showed that many of the known Prikry-type notions of forcing that center around singular cardinals of countable cofinality
are $\Sigma$-Prikry. We showed that given a $\Sigma$-Prikry poset $\mathbb P$ and a $\mathbb{P}$-name for a non-reflecting stationary set $T$,
there exists a corresponding $\Sigma$-Prikry poset that projects to $\mathbb P$ and kills the stationarity of $T$.
In this paper, we develop a general scheme for iterating $\Sigma$-Prikry posets and, 
as an application, we blow up the power of a countable limit of Laver-indestructible supercompact cardinals,
and then iteratively kill all non-reflecting stationary subsets of its successor. 
This yields a model in which the singular cardinal hypothesis fails and
simultaneous reflection of finite families of stationary sets holds.
\end{abstract}
\maketitle
\tableofcontents

\section{Introduction}

In the introduction to Part~I of this series \cite{partI}, we described the need for iteration schemes
and the challenges involved in devising such schemes, especially at the level of successor of singular cardinals.
The main tool available to obtain consistency results at the level of singular cardinals and their successors
is the method of forcing with large cardinals and, in particular, \emph{Prikry-type forcings}.
By Prikry-type forcings one usually means a poset $\mathbb P=(P,\le)$ having the following property.
\begin{prikry} There exists an ordering $\le^*$ on $P$ coarser than $\le$ (typi\-cally, of a better closure degree)
satisfying that for every sentence $\varphi$ in the forcing language
and every $p\in P$ there exists $q\in P$ with $q\le^* p$
deciding $\varphi$.
\end{prikry}

In this paper, we develop an iteration scheme for Prikry-type posets, specifically, for the class of \emph{$\Sigma$-Prikry forcings}  
that we introduced in \cite{partI} (see Definition~\ref{SigmaPrikry} below).
Of course, viable iteration schemes for Prikry-type posets already exists,
namely, the Magidor iteration and the Gitik iteration (see \cite[\S6]{Gitik-handbook}).
In both these cases
the ordering $\le^*$ witnessing the Prikry Property of the iteration
can be roughly described as the finite-support iteration of the $\le^*$-orderings of its components.
As the expectation from the final $\le^*$ is to have an eventually-high closure degree,
the two schemes are typically useful in the context where
one carries an iteration $\langle \mathbb{P}_\alpha; \dot{\mathbb{Q}}_\alpha\mid \alpha<\rho\rangle$
with each $\dot{\mathbb{Q}}_\alpha$ being a $\mathbb P_\alpha$-name for either a trivial forcing,
or a Prikry-type forcing concentrating on the combinatorics of the inaccessible cardinal $\alpha$. This should be compared with the  iteration to
control the power function $\alpha\mapsto2^\alpha$ below some cardinal $\rho$.

In contrast, in this paper, we are interested in carrying out an iteration of length $\kappa^{++}$,
where $\kappa$ is a singular cardinal (or, more generally, forced by the first step of the iteration to become one), and all components of the iteration
are Prikry-type forcings that concentrate on the combinatorics of $\kappa$ or its successor.
For this, we will need to allow a support of arbitrarily large size below $\kappa$.
To be able to lift the Prikry property through an infinite-support iteration,
members of the $\Sigma$-Prikry class are thus required to possess the following stronger property, which is inspired by the concepts coming from the study of topological Ramsey spaces \cite{MR2603812}.
\begin{cpp}
There is a partition of the ordering $\le$ into countably many relations $\langle {\le^n}\mid n<\omega\rangle$ such that,
if we denote $\CONE_n(q):=\{ r\mid r\le^n q\}$,
then, for every $0$-open $U\s P$ (i.e., $q\in U\implies\CONE_0(q)\s U$),
every $p\in P$ and every $n<\omega$, there exists $q\le^0 p$
such that  $\CONE_n(q)$ is either a subset of $U$ or disjoint from $U$.
\end{cpp}
To maintain the above property along the iteration we demand on our posets to satisfy  \emph{property $\mathcal{D}$} 
(Definition~\ref{PropertyD} below). Succinctly, this property is a game-theoretic abstraction of a standard approach for verifying the Prikry property;
it asserts that $\pI$ has a winning strategy in a two-player game in which $\pI$ (the `good' player) works towards \emph{diagonalizing} the sequence of conditions produced by $\pII$ (the `bad' player). 

\smallskip

Another parameter that requires attention when devising an iteration scheme is the chain condition of the components to be used.
In view of the goal of solving a problem concerning the combinatorics of $\kappa$ or its successor
through an iteration of length $\kappa^{++}$,
there is a need to know that all counterexamples to our problem will show up at some intermediate stage of the iteration,
so that we at least have the chance to kill them all.
The standard way to secure the latter is to require that
the whole iteration $\mathbb P_{\kappa^{++}}$ would have the $\kappa^{++}$-chain condition ($\kappa^{++}$-cc).
As the $\kappa$-support iteration of $\kappa^{++}$-cc posets need not have the $\kappa^{++}$-cc (see \cite{roslanowski} for an explicit counterexample),
members of the $\Sigma$-Prikry class are required to satisfy
the following strong form of the $\kappa^{++}$-cc:
\begin{zlinked} There exists a map $c:P\rightarrow\kappa^+$ satisfying that for all $p,q\in P$,
if $c(p)=c(q)$, then $p$ and $q$ are compatible, and, furthermore,
$\CONE_0(p)\cap\CONE_0(q)$ is nonempty.
\end{zlinked}

In particular, our verification of the chain condition of $\mathbb P_{\kappa^{++}}$ will not go through the $\Delta$-system lemma;
rather, we will take advantage of a basic fact concerning the density of box products of topological spaces.

\medskip

Now that we have a way to ensure that all counterexamples show up
at intermediate stages, we fix a bookkeeping list $\langle z_\alpha\mid\alpha<\kappa^{++}\rangle$,
and shall want
that, for any $\alpha<\kappa^{++}$,
$\mathbb P_{\alpha+1}$ will amount to forcing over the model
$V^{\mathbb P_\alpha}$
to solve a problem suggested by $z_\alpha$.
The standard approach to achieve this is to set $\mathbb P_{\alpha+1}:=\mathbb P_\alpha*\dot{\mathbb Q}_\alpha$,
where $\dot{\mathbb Q}_\alpha$ is a $\mathbb P_\alpha$-name for a poset that takes care of $z_\alpha$.
However, the disadvantage of this approach is that
if $\mathbb P_1$ is a notion of forcing that blows up $2^\kappa$,
then any typical poset $\mathbb Q_1$ in $V^{\mathbb P_1}$ which is designed to add a subset of $\kappa^+$ via bounded approximations
will fail to have the $\kappa^{++}$-cc.
To work around this,
in our scheme, we set $\mathbb P_{\alpha+1}:=\mathbb A(\mathbb P_\alpha,z_\alpha)$,
where $\mathbb A({\cdot},{\cdot})$ is a functor that,
to each $\Sigma$-Prikry poset $\mathbb P$ and a problem $z$, produces a $\Sigma$-Prikry poset $\mathbb A(\mathbb P,z)$
that projects onto $\mathbb P$ and solves the problem $z$.
A key feature of this functor is that the projection from $\mathbb A(\mathbb P, z)$ to $\mathbb P$
\emph{splits}, that is, in addition to a projection map $\pi$ from $\mathbb A(\mathbb P, z)$ onto $\mathbb P$,
there is a map $\pitchfork$ that goes in the other direction, and the two maps commute in a very strong sense.
The exact details may be found in our definition of \emph{forking projection} (see Definition~\ref{forking} below).

\medskip

A special case of the main result of this paper may be roughly stated as follows.
\begin{MAIN} Suppose that $\Sigma=\langle\kappa_n\mid n<\omega\rangle$ is a strictly increasing sequence of regular uncountable cardinals,
converging to a cardinal $\kappa$.
For simplicity, let us say that a notion of forcing $\mathbb P$ is \emph{nice}
if it has property $\mathcal D$, $\mathbb{P}\s H_{\kappa^{++}}$ and $\mathbb P$ does not collapse $\kappa^+$.
Now, suppose that:
\begin{itemize}
\item $\mathbb{Q}$ is a nice $\Sigma$-Prikry notion of forcing;
\item $\mathbb A({\cdot},{\cdot})$ is a functor that produces for every
nice $\Sigma$-Prikry notion of forcing $\mathbb P$,
and every $z\in H_{\kappa^{++}}$,
a corresponding nice $\Sigma$-Prikry notion of forcing $\mathbb A(\mathbb P,z)$. 
Moreover, $\mathbb A({\cdot},{\cdot})$ admits a forking projection to $\mathbb P$ with the weak mixing property;
\item  $2^{2^\kappa}=\kappa^{++}$, so that we may fix a bookkeeping list $\langle z_\alpha\mid\alpha<\kappa^{++}\rangle$.
\end{itemize}
Then there exists a sequence $\langle \mathbb P_\alpha\mid\alpha\le\kappa^{++}\rangle$
of forcings
such that $\mathbb P_1$ is isomorphic to $\mathbb Q$,
$\mathbb P_{\alpha+1}$ is isomorphic to $\mathbb A(\mathbb P_\alpha,z_\alpha)$,
and, for every pair $\alpha\le\beta\le\kappa^{++}$, $\mathbb P_\beta$ projects onto $\mathbb P_\alpha$.  Moreover, if for each nonzero limit ordinal $\alpha\leq \kappa^{++}$, a certain canonical subforcing $\z{\mathbb{P}}_\alpha$
of $\mathbb{P}_\alpha$ is dense in $\mathbb P_\alpha$,
then $\langle \mathbb P_\alpha\mid\alpha\le\kappa^{++}\rangle$ consists of nice $\Sigma$-Prikry forcings.
\end{MAIN}

\subsection{Organization of this paper}
In Section~\ref{SPS}, we recall the definitions of the $\Sigma$-Prikry class,
forking projections, and introduce property $\mathcal D$ and the weak mixing property.

In Section~\ref{Iteration}, we present our abstract iteration scheme for $\Sigma$-Prikry posets,
and prove the Main Theorem of this paper (see  Lemmas~\ref{CvIteration} and \ref{CivIteration}).

In Section~\ref{ReflectionAfterIteration}, we present the very first application of our scheme.
We carry out an iteration of length $\kappa^{++}$,
where the first step of the iteration is the \emph{Extender Based Prikry Forcing} due to Gitik and Magidor \cite[\S3]{Git-Mag} for making $2^\kappa=\kappa^{++}$,
and all the later steps are obtained by invoking the functor $\mathbb A(\mathbb P,z)$ from \cite[\S6]{partI}
for killing a nonreflecting stationary set decoded from a $\mathbb{P}$-name $z$.
This functor is due to Sharon \cite[\S2]{AS}, 
and as a corollary, we obtain a streamlined proof of the main result of \cite[\S3]{AS}:

\begin{COR}If $\kappa$ is the limit of a countable increasing sequence of supercompact cardinals,
then there exists a cofinality-preserving forcing extension
in which $\kappa$ remains a strong limit,
every finite collection of stationary subsets of $\kappa^+$ reflects simultaneously, and $2^\kappa=\kappa^{++}$.
\end{COR}

\subsection{Notation and conventions} Our forcing convention is that $p\le q$ means that $p$ extends $q$.
We write $\cone{q}$ for $\{ p\in\mathbb P\mid p\le q\}$.
We will follow the common convention of using
dotted free variables in forcing statement for forcing names and using undotted variables for canonical names for sets from the ground model.
In some instances, to stress that we are dealing with canonical names, we will be using the classical \emph{check name} notation.
Denote $E^\mu_{\theta}:=\{\alpha<\mu\mid \cf(\alpha)=\theta\}$. The sets $E^\mu_{<\theta}$ and $E^\mu_{>\theta}$ are defined in a similar fashion.
For a stationary subset $S$ of a regular uncountable cardinal $\mu$, we write $\tr(S):=\{\gamma\in E^\mu_{>\omega}\mid S\cap\gamma\text{ is stationary in }\gamma\}$.
$H_\nu$ denotes the collection of all sets of hereditary cardinality less than $\nu$.
For every set of ordinals $x$, we denote $\cl(x):=\{ \sup(x\cap\gamma)\mid \gamma\in\ord, x\cap\gamma\neq\emptyset\}$,  $\acc(x):=\{\gamma\in x\mid \sup(x\cap\gamma)=\gamma>0\}$ and $\nacc(x):=x\setminus \acc(x)$.

\section{The $\Sigma$-Prikry class and forking projections}\label{SPS}

In this section, we recall some definitions and facts from \cite[\S2]{partI} and \cite[\S4]{partI},
and then continue developing the theory of \emph{forking projections}.
Familiarity with \cite{partI} is not assumed here.

\subsection{The $\Sigma$-Prikry class and Property $\mathcal D$}
\begin{definition}\label{gradedposet} We say that $(\mathbb P,\lh)$ is a \emph{graded poset}
iff $\mathbb P=(P,\le)$ is a poset, $\lh:P\rightarrow\omega$ is a surjection, and, for all $p\in P$:
\begin{itemize}
\item  For every $q\le p$, $\lh(q)\geq\lh(p)$;
\item  There exists $q\le p$ with $\lh(q)=\lh(p)+1$.
\end{itemize}
\end{definition}
\begin{conv} For a graded poset as above,
we denote $P_n:=\{p\in P\mid \lh(p)=n\}$,
$P_n^p:=\{ q\in P\mid  q\le p, \lh(q)=\lh(p)+n\}$,
and sometimes write $q\le^n p$ (and say the $q$ is \emph{an $n$-step extension} of $p$) rather than writing $q\in P^p_n$.
\end{conv}

\begin{definition}\label{SigmaPrikry}
Suppose that $\mathbb P=(P,\le)$ is a notion of forcing with a greatest element $\one$,
and that $\Sigma=\langle \kappa_n\mid n<\omega\rangle$ is a non-decreasing sequence of regular uncountable cardinals,
converging to some cardinal $\kappa$.
Suppose that $\mu$ is a cardinal such that $\one\forces_{\mathbb P}\check\mu=\kappa^+$.
For functions $\lh:P\rightarrow\omega$ and $c:P\rightarrow \mu$,
we say that $(\mathbb P,\lh,c)$ is \emph{$\Sigma$-Prikry} iff all of the following hold:
\begin{enumerate}
\item\label{c4} $(\mathbb P,\lh)$ is a graded poset;
\item\label{c2} For all $n<\omega$, $\mathbb P_n:=(P_n\cup\{\one\},\le)$ contains a dense subposet $\z{\mathbb P}_n$
which is $\kappa_n$-directed-closed;
\item\label{c1} For all $p,q\in P$, if $c(p)=c(q)$, then $P_0^p\cap P_0^q$ is non-empty;
\item\label{c5} For all $p\in P$, $n,m<\omega$ and $q\le^{n+m}p$, the set $\{r\le^n p\mid  q\le^m r\}$ contains a greatest element which we denote by $m(p,q)$.\footnote{By convention, a greatest element, if exists, is unique.}
In the special case $m=0$, we shall write $w(p,q)$ rather than $0(p,q)$;\footnote{
Note that $w(p,q)$ is the weakest $n$-step extension of $p$  above $q$.}
\item\label{csize} For all $p\in P$,
the set $W(p):=\{w(p,q)\mid q\le p\}$ has size $<\mu$;
\item\label{itsaprojection} For all $p'\le p$ in $P$, $q\mapsto w(p,q)$ forms an order-preserving map from $W(p')$ to $W(p)$;
\item\label{c6}  Suppose that $U\s P$ is a $0$-open set, i.e., $r\in U$ iff $P^r_0\s U$.
Then, for all $p\in P$ and $n<\omega$, there is $q\le^0 p$, such that, either $P^{q}_n\cap U=\emptyset$ or $P^{q}_n\s U$.
\end{enumerate}
\end{definition}

\begin{remark}
\begin{enumerate}[label=(\roman*)]
\item Clause~\eqref{c2} differs from that of \cite[Definition~2.3]{partI}, where we originally required $\mathbb P_n$ itself to be $\kappa_n$-directed-closed.
\item Clause~\eqref{c1} is the Introduction's \emph{Linked$_0$ property}.
Often, we will want to avoid encodings and opt to define the function $c$ as a map from $P$ to some natural set $\mathfrak M$ of size $\leq\mu$,
instead of a map to the cardinal $\mu$ itself. In the special case that $\mu^{<\mu}=\mu$,
we shall simply take $\mathfrak M$ to be $H_{\mu}$.
\item Clause~\eqref{c6} is the \emph{Complete Prikry Property} (CPP).
\end{enumerate}
\end{remark}

\begin{definition}
Let $p\in P$. For each $n<\omega$, we write $W_n(p):=\{ w(p,q)\mid q\in P^p_n\}$,
and $W_{\ge n}(p):=\{ w(p,q)\mid \exists m\in\omega\setminus n\,(q\in P^p_m)\}$.
The object $W(p):=\bigcup_{n<\omega} W_n(p)$ is called \emph{the $p$-tree}.
\end{definition}

\begin{fact}[{\cite[Lemma~2.8]{partI}}]\label{lemma7} Let $p\in P$.
\begin{enumerate}
\item For every $n<\omega$, $W_n(p)$ is a maximal antichain in $\cone{p}$;
\item Every two compatible elements of $W(p)$ are comparable;
\item For any pair $q'\le q$ in $W(p)$,  $q'\in W(q)$;
\item $c\restriction W(p)$ is injective.
\end{enumerate}
\end{fact}

\begin{fact}[{\cite[Lemma~2.10]{partI}}]\label{l14}\hfill
\begin{enumerate}
\item \label{C1l14} $\mathbb P$ does not add bounded subsets of $\kappa$;
\item \label{C2l14} For every regular cardinal $\nu\geq\kappa$, if there exists $p\in P$ for which $p\forces_{\mathbb P}\cf(\nu)<\kappa$,
then there exists $p'\le p$ with $|W(p')|\geq\nu$.\footnote{For future reference, we point out that this fact relies only on Clauses (\ref{c4}), (\ref{c2}), (\ref{c5}) and (\ref{c6}) of Definition~\ref{SigmaPrikry}.
Furthermore, we do not need to know that $\one$ decides a value for $\kappa^+$.}
\end{enumerate}
\end{fact}

\begin{definition} We say that $\vec r=\langle r_\xi \mid \xi<\chi\rangle$ is a \emph{good enumeration} of a set $A$ iff
$\vec r$ is injective, 
$\chi$ is a cardinal,
and $\{ r_\xi\mid \xi<\chi\}=A$.
\end{definition}

\begin{definition}[Diagonalizability]\label{Diagonalizability}
Given $p\in P$, $n<\omega$, and a good enumeration
$\vec r=\langle r_\xi\mid \xi<\chi\rangle$ of $W_n(p)$,
we say that $\vec q=\langle q_\xi\mid \xi<\chi\rangle$ is \emph{diagonalizable} (with respect to $\vec{r}$)  iff the two hold:
\begin{enumerate}[label=(\alph*)]
\item $q_\xi\le^0 r_\xi$ for every $\xi<\chi$;
\item there is $p'\le^0 p$ such that for every $q'\in W_n(p')$, $q'\le^0 q_\xi$, where $\xi$ is the unique index to satisfy $r_\xi=w(p, q')$.
\end{enumerate}
\end{definition}

\begin{definition}[Diagonalizability game]\label{DiagonalizabilityGame}
Given $p\in P$, $n<\omega$, a good enumeration  $\vec r=\langle r_\xi\mid \xi<\chi\rangle$ of $W_n(p)$,
and a dense subset $D$ of $\mathbb{P}_{\ell_\mathbb{P}(p)+n}$, $\Game_\mathbb{P}(p,\vec r,D)$ is a game of length $\chi$ between two players $\pI$ and $\pII$, defined as follows:
\begin{itemize}
\item At stage $\xi<\chi$, $\pI$ plays a condition $p_\xi\leq^0 p$ compatible with $r_\xi$,
and then $\pII$ plays $q_\xi\in D$ such that $q_\xi\leq p_\xi$ and $q_\xi\leq^0 r_\xi$;
\item $\pI$ wins the game iff the resulting sequence $\vec q=\langle q_\xi\mid \xi<\chi\rangle$ is diagonalizable.
\end{itemize}
In the special case that $D$ is all of $\mathbb{P}_{\ell_\mathbb{P}(p)+n}$, we omit it, writing $\Game_\mathbb{P}(p,\vec r)$.
\end{definition}

The following  lemma  will be useful later.

\begin{lemma}\label{propertyDatdense}
Given $p\in P$, $n<\omega$, a good enumeration  $\vec r$ of $W_n(p)$,
and a dense subset $D$ of $\mathbb{P}_{\ell_\mathbb{P}(p)+n}$,
$\pI$ has a winning strategy for $\Game_\mathbb{P}(p,\vec{r},D)$
iff it has a winning strategy for $\Game_\mathbb{P}(p,\vec{r})$.
\end{lemma}
\begin{proof} Only the forward implication requires an argument.
Write $\vec r$ as $\langle r_\xi\mid \xi<\chi\rangle$;
we shall describe a winning strategy for $\pI$ in the game $\Game_\mathbb{P}(p,\vec{r})$
by producing sequences of the form $\langle (p_\eta,q_\eta,q_\eta')\mid \eta<\xi\rangle$,
where $\langle (p_\eta,q_\eta)\mid \eta<\xi\rangle$ is an initial play (consisting of $\xi$ rounds)  in the game
$\Game_\mathbb{P}(p,\vec{r})$,
and $\langle (p_\eta,q_\eta')\mid \eta<\xi\rangle$ is an initial play in the game $\Game_\mathbb{P}(p,\vec{r},D)$.

Assuming that $\pI$ has a winning strategy for $\Game_\mathbb{P}(p,\vec{r},D)$,
here is a description of our winning strategy for $\pI$ in the game $\Game_\mathbb{P}(p,\vec{r})$:

$\br$ For $\xi=0$, we play a condition $p_0$ according to the winning strategy of $\pI$ in the game $\Game_\mathbb{P}(p,\vec{r},D)$.
Then, $\pII$ plays $q_0\leq p_0$ such that $q_0\leq^0 r_0$.
Since $D$ is dense in $\mathbb{P}_{\ell_\mathbb{P}(p)+n}$, we then pick $q_0'\in D$ with $q_0'\le^0 q_0$.

$\br$ Suppose that $\xi<\chi$ is nonzero and that $\langle (p_\eta,q_\eta,q_\eta')\mid \eta<\xi\rangle$ has already been defined.
Let $p_\xi$ be given by the winning strategy of $\pI$ for the game $\Game_\mathbb{P}(p,\vec{r},D)$
with respect to the initial play $\langle (p_\eta,q_\eta')\mid \eta<\xi\rangle$.
Then, $\pII$ plays $q_\xi\leq p_\xi$ such that $q_\xi\leq^0 r_\xi$.
Finally, pick $q_\xi'\in D$ such that $q_\xi'\le^0 q_\xi$.

At the end of the above process, since $\langle (p_\xi,q_\xi')\mid \xi<\chi\rangle$ is a play in the game $\Game_\mathbb{P}(p,\vec{r},D)$
using the winning strategy of $\pI$, we may fix $p'\le^0p$ witnessing that
$\langle q_\xi'\mid \xi<\chi\rangle$ is diagonalizable.
So, for every $q'\in W_n(p')$,
if $\xi$ is the unique index to satisfy $r_\xi=w(p, q')$,
then $q'\le^0 q_\xi'\le^0 q_\xi$. In particular, $p'$ witnesses that
$\langle q_\xi\mid \xi<\chi\rangle$ is diagonalizable, as desired.
\end{proof}

\begin{definition}[Property $\mathcal{D}$]\label{PropertyD}
We say that $(\mathbb{P},\ell_\mathbb{P})$ has \emph{property $\mathcal{D}$}
iff for any $p\in P$, $n<\omega$ and any good enumeration  $\vec r=\langle r_\xi\mid \xi<\chi\rangle$ of $W_n(p)$, $\pI$ has a winning strategy for the game
$\Game_\mathbb{P}(p,\vec r)$.
\end{definition}

\subsection{Forking projections}In this and the next subsection, we continue the work started in  \cite[\S4]{partI} concerning forking projections.
This will play a key role in Section~\ref{Iteration}, where we deal with iterating $\Sigma$-Prikry posets.

\begin{notation} Given two posets $\mathbb P=(P,\le)$ and $\mathbb A=(A,\unlhd)$,
 and a projection $\pi$ from $\mathbb A$ to $\mathbb P$, we denote by $\mathbb A^\pi$ the poset $(A,\unlhd^\pi)$,
 where $a\unlhd^\pi b$ iff $a\unlhd b$ and $\pi(a)=\pi(b)$.

For a subposet $\z{\mathbb A}=(\z{A},\unlhd)$ of $\mathbb A$, we likewise denote $\z{\mathbb A}^\pi:=(\z{A},\unlhd^\pi)$.
\end{notation}

\begin{definition}[{\cite[Definition~4.1]{partI}}]\label{forking} Suppose that $(\mathbb P,\lh_{\mathbb P},c_{\mathbb P})$ is a $\Sigma$-Prikry triple,
$\mathbb A=(A,\unlhd)$ is a notion of forcing,
and $\lh_{\mathbb A}$ and $c_{\mathbb A}$ are functions with $\dom(\lh_{\mathbb A})=\dom(c_{\mathbb A})=A$.

A pair of functions $({\pitchfork},{\pi})$ is said to be a \emph{forking projection} from $(\mathbb A,\lh_{\mathbb A})$ to $(\mathbb P,\lh_{\mathbb P})$
iff all of the following hold:
\begin{enumerate}
\item\label{frk1} $\pi$ is a projection from $\mathbb A$ onto $\mathbb P$, and $\lh_{\mathbb A}=\lh_{\mathbb P}\circ\pi$;
\item\label{frk0} for all $a\in A$, $\fork{a}$ is an order-preserving function from $(\cone{\pi(a)},\le)$ to $(\conea{a},\unlhd)$;
\item\label{frk3} for all $p\in P$, $\{ a \in A\mid \pi(a)=p\}$ admits a greatest element, which we denote by $\myceil{p}{\mathbb A}$;
\item\label{frk4} for all $n,m<\omega$ and $b\unlhd^{n+m} a$, $m(a,b)$ exists and satisfies:
 $$m(a,b)=\fork{a}(m(\pi(a),\pi(b)));$$
\item\label{frk5} for all $a\in A$ and $q\le\pi(a)$,   $\pi(\fork{a}(q))=q$;
\item\label{frk6} for all $a\in A$ and $q\le\pi(a)$, $a=\myceil{\pi(a)}{\mathbb A}$ iff $\fork{a}(q)=\myceil{q}{\mathbb A}$;
\item\label{frk7} for all $a\in A$, $a'\unlhd^0 a$ and $r\le^0 \pi(a')$, $\fork{a'}(r)\unlhd\fork{a}(r)$.
\setcounter{condition}{\value{enumi}}
\end{enumerate}

The pair $({\pitchfork},{\pi})$ is said to be a forking projection from  $(\mathbb A,\lh_{\mathbb A},c_{\mathbb A})$ to $(\mathbb P,\lh_{\mathbb P},c_{\mathbb P})$
iff, in addition to all of the above, the following holds:
\begin{enumerate}
\setcounter{enumi}{\value{condition}}
\item\label{frk2}  for all $a,a'\in A$, if $c_{\mathbb A}(a)=c_{\mathbb A}(a')$, then $c_{\mathbb P}(\pi(a))=c_{\mathbb P}(\pi(a'))$ and, for all $r\in P_0^{\pi(a)}\cap P_0^{\pi(a')}$, $\fork{a}(r)=\fork{a'}(r)$.
\end{enumerate}
\end{definition}
\begin{remark}
Intuitively speaking, $\fork{a}$ is an operator that, for each condition $p\in P\downarrow \pi(a)$, provides the $\unlhd$-greatest condition $b\unlhd a$ with $\pi(b)=p$.
\end{remark}

\begin{example}\label{trivialfunctor}
Suppose that $(\mathbb P,\lh_{\mathbb P},c_{\mathbb P})$ is a $\Sigma$-Prikry triple.
Let $\mu$ denote the cardinal such that $\one\forces_{\mathbb P}\check\mu=\kappa^+$.
We define the following objects:
\begin{itemize}
\item $\mathbb A=(A,\unlhd)$, where $A:=P\times\mu$ and $(p,\alpha)\unlhd (q,\beta)$ iff $p\le q$ and $\alpha\supseteq\beta$;
\item $\lh_{\mathbb A}:A\rightarrow\omega$ via $\lh_{\mathbb A}(p,\alpha):=\lh_{\mathbb P}(p)$;
\item $c_{\mathbb A}:A\rightarrow\mu\times\mu$ via $c_{\mathbb A}(p,\alpha):=(c_{\mathbb P}(p),\alpha)$;
\item  $\pi:A\rightarrow P$ via $\pi(p,\alpha):=p$;
\item for $a=(p,\alpha)\in A$, define
$\fork{a}:\cone{p}\rightarrow A$ via
$\fork{a}(q):=(q,\alpha)$.
\end{itemize}
Then $(\pitchfork,\pi)$ is a forking projection from $(\mathbb A,\lh_{\mathbb A},c_{\mathbb A})$ to $(\mathbb P,\lh_{\mathbb P},c_{\mathbb P})$,
and $\myceil{a}{\mathbb{A}}=(\pi(a),0)$ for all $a\in A$.
\end{example}

\begin{lemma}\label{frkid} Suppose that $({\pitchfork},{\pi})$ is a forking projection from $(\mathbb A,\lh_{\mathbb A})$ to $(\mathbb P,\lh_{\mathbb P})$.
For every $a\in A$, $\fork{a}(\pi(a))=a$.
\end{lemma}
\begin{proof}  By Definition~\ref{forking}\eqref{frk4}, using $(n,m,b):=(0,0,a)$, we infer that
\[\fork{a}(\pi(a))=\fork{a}(w(\pi(a),\pi(a)))=w(a,a)=a.\qedhere\]
\end{proof}

\begin{lemma}[Canonical form]\label{canonical}
Suppose that $(\mathbb P,\lh_{\mathbb P},c_{\mathbb P})$ and $(\mathbb A,\lh_{\mathbb A},c_{\mathbb A})$ are both $\Sigma$-Prikry notions of forcing.
Denote $\mathbb P=(P,\le)$ and $\mathbb A=(A,\unlhd)$.

If $(\mathbb A,\lh_{\mathbb A},c_{\mathbb A})$ \emph{admits a forking
projection} to $(\mathbb P,\lh_{\mathbb P},c_{\mathbb P})$ as witnessed by a pair $({\pitchfork},{\pi})$,
then we may assume that all of the following hold true:
\begin{enumerate}
\item each element of $A$ is a pair $(x,y)$ with $\pi(x,y)=x$;
\item for all $a\in A$, $\myceil{\pi(a)}{\mathbb A}=(\pi(a),\emptyset)$;
\item for all $p,q\in P$, if $c_{\mathbb P}(p)=c_{\mathbb P}(q)$, then $c_{\mathbb A}(\myceil{p}{\mathbb A})=c_{\mathbb A}(\myceil{q}{\mathbb A})$.
\end{enumerate}
\end{lemma}
\begin{proof} By applying a bijection, we may assume that $A=|A|$ with $\one_{\mathbb A}=\emptyset$.
To clarify what we are about to do, we agree to say that ``$a$ is a lift'' iff $a=\myceil{\pi(a)}{\mathbb A}$.
Now, define $f:A\rightarrow P\times A$ via:
$$f(a):=\begin{cases}
(\pi(a),\emptyset),&\text{if }a\text{ is a lift};\\
(\pi(a),a),&\text{otherwise}.
\end{cases}$$

\begin{claim} $f$ is injective.
\end{claim}
\begin{proof} Suppose $a,a'\in A$ with $f(a)=f(a')$.

$\br$ If $a$ is not a lift and $a'$ is not a lift, then from $f(a)=f(a')$ we immediately get
that $a=a'$.

$\br$ If $a$ is a lift and $a'$ is  a lift, then from $f(a)=f(a')$, we infer that $\pi(a)=\pi(a')$, so that $a=\myceil{\pi(a)}{\mathbb A}=\myceil{\pi(a')}{\mathbb A}=a'$.

$\br$ If $a$ is not a lift, but $a'$ is a lift, then from $f(a)=f(a')$, we infer that
$a=\emptyset=\one_{\mathbb A}$, contradicting the fact that
$\one_{\mathbb A}=\myceil{\one_{\mathbb P}}{\mathbb A}=\myceil{\pi(\one_{\mathbb A})}{\mathbb A}$ is a lift.
So this case is void.
\end{proof}
Let $B:=\rng(f)$ and
${\unlhd_B}:=\{ (f(a),f(b))\mid a\unlhd b\}$,
so that $\mathbb B:=(B,\unlhd_B)$ is isomorphic to $\mathbb A$.
Define $\lh_{\mathbb B}:=\lh_{\mathbb A}\circ f^{-1}$ and $\pi_{\mathbb B}:=\pi\circ f^{-1}$.
Also, define $\pitchfork_{\mathbb B}$ via $\fork[\mathbb B]{b}(p):=f(\fork{f^{-1}(b)}(p))$.
It is clear that $b\in B$ is a lift iff $f^{-1}(a)$ is a lift iff $b=(\pi_{\mathbb B}(b),\emptyset)$.

Next, define $c_{\mathbb B}:B\rightarrow\mu\times2$ by letting for all $b\in B$:
$$c_{\mathbb B}(b):=\begin{cases}
(c_{\mathbb P}(\pi_{\mathbb B}(b)),0),&\text{if }b\text{ is a lift};\\
(c_{\mathbb A}(f^{-1}(b)),1),&\text{otherwise}.
\end{cases}$$
\begin{claim} Suppose $b_0,b_1\in B$ with $c_{\mathbb B}(b_0)=c_{\mathbb B}(b_1)$.
Then $c_{\mathbb P}(\pi_{\mathbb B}(b_0))=c_{\mathbb P}(\pi_{\mathbb B}(b_1))$ and,
for all $r\in P_0^{\pi_{\mathbb B}(b_0)}\cap P_0^{\pi_{\mathbb B}(b_1)}$, $\fork[\mathbb B]{b_0}(r)=\fork[\mathbb B]{b_1}(r)$.
\end{claim}
\begin{proof}
We focus on verifying that for all $r\in P_0^{\pi_{\mathbb B}(b_0)}\cap P_0^{\pi_{\mathbb B}(b_1)}$, $\fork[\mathbb B]{b_0}(r)=\fork[\mathbb B]{b_1}(r)$.
For each $i<2$, denote $a_i:=f^{-1}(b_i)$ and $p_i:=\pi_{\mathbb B}(b_i)$,
so that $\pi(a_i)=p_i$. Suppose $r\in P_0^{p_0}\cap P_0^{p_1}$.

$\br$ If $b_0$ is a lift, then so are $b_1,a_0,a_1$.
Therefore, for each $i<2$, Definition~\ref{forking}(\ref{frk6})
implies that $\fork[\mathbb B]{b_i}(r)=f(\fork{a_i}(r))=f(\myceil{r}{\mathbb A})=\myceil{r}{\mathbb B}$.
Consequently, $\fork[\mathbb B]{b_0}(r)=\fork[\mathbb B]{b_1}(r)$, as desired.

$\br$ Otherwise, $c_{\mathbb A}(a_0)=c_{\mathbb A}(a_1)$.
As $r\in P_0^{\pi(a_0)}\cap P_0^{\pi(a_1)}$,
 ${\pitchfork_{\mathbb B}}(b_0)(p)=f(\fork{a_0}(p))=f(\fork{a_1}(p))=\fork[\mathbb B]{b_1}(p)$.
\end{proof}
This completes the proof.
\end{proof}

\begin{setup}
Throughout the rest of this section, suppose that:
\begin{itemize}
\item $\mathbb P=(P,\le)$ is a notion of forcing with a greatest element $\one_{\mathbb P}$;
\item $\mathbb A=(A,\unlhd)$ is a notion of forcing with a greatest element $\one_{\mathbb A}$;
\item $\Sigma=\langle \kappa_n\mid n<\omega\rangle$ is a non-decreasing sequence of regular uncountable cardinals,
converging to some cardinal $\kappa$, and  $\mu$ is a cardinal such that $\one_{\mathbb P}\Vdash_{\mathbb P}\check\mu=\check\kappa^+$;
\item $\lh_{\mathbb P}$ and $c_{\mathbb P}$ are functions witnessing that $(\mathbb P,\lh_{\mathbb P},c_{\mathbb P})$ is $\Sigma$-Prikry;
\item $\lh_{\mathbb A}$ and $c_{\mathbb A}$ are functions with $\dom(\lh_{\mathbb A})=\dom(c_{\mathbb A})=A$;
\item $(\pitchfork,\pi)$ is a forking projection from $(\mathbb A,\lh_{\mathbb A},c_{\mathbb A})$ to $(\mathbb P,\lh_{\mathbb P},c_{\mathbb P})$.
\end{itemize}
\end{setup}

The next two facts will help verifying Clauses \eqref{c4} and  \eqref{c1} of Definition~\ref{SigmaPrikry} for the different stages of the iteration in Section~\ref{Iteration}.

\begin{fact}[{\cite[Lemma~4.3]{partI}}]\label{forkingfacts}
Suppose that $({\pitchfork},{\pi})$ is a forking projection from $(\mathbb A,\lh_{\mathbb A})$ to $(\mathbb P,\lh_{\mathbb P})$,
or, just a pair of maps satisfying Clauses \eqref{frk1}, \eqref{frk0} and \eqref{frk4} of Definition~\ref{forking}.
For each $a\in A$, the following holds:
\begin{enumerate}
\item  $\fork{a}\restriction W(\pi(a))$ forms a bijection from $W(\pi(a))$ to $W(a)$;
\item for all $n<\omega$ and $r\in P_n^{\pi(a)}$,  $\fork{a}(r)\in A_n^a$.
\end{enumerate}

In particular, $(\mathbb A,\lh_{\mathbb A})$ is a graded poset.
\end{fact}

\begin{fact}[{\cite[Lemma~4.7]{partI}}]
Suppose that $({\pitchfork},{\pi})$ is a forking projection from $(\mathbb A,\lh_{\mathbb A},c_{\mathbb A})$ to $(\mathbb P,\lh_{\mathbb P},c_{\mathbb P})$,
or, just a pair of maps satisfying Clauses \eqref{frk1}, \eqref{frk0}, \eqref{frk4}, \eqref{frk7} and \eqref{frk2} of Definition~\ref{forking}.
For all $a,a'\in A$, if $c_{\mathbb A}(a)=c_{\mathbb A}(a')$, then $A_0^a\cap A_0^{a'}$ is non-empty. In particular, if $|\rng(c_{\mathbb A})|\le\mu$, then $(\mathbb{A},\ell_\mathbb{A})$ is $\mu^+$-2-linked$_0$.
\end{fact}

\begin{lemma}\label{propertyDyieldsCPP}
Suppose that $(\mathbb{A},\ell_\mathbb{A})$ has property $\mathcal{D}$. Then it has the $\CPP$.
\end{lemma}
\begin{proof}
Let $U\s A$ be a $0$-open set, $a\in A$ and $n<\omega$;
we shall find $\bar a\unlhd^0 a$ such that either $A^{\bar a}_n\cap U=\emptyset$ or $A^{\bar a}_n\s U$.

Let $\vec{r}=\langle r_\xi\mid \xi<\chi\rangle$ be a good enumeration of $W_n(a)$.
Let $\langle (a_\xi, b_\xi)\mid \xi<\chi \rangle$ list the rounds of the game $\Game_\mathbb{A}(a,\vec r)$
in which, in round $\xi$, $\pI$ plays according to their winning strategy and
$\pII$ plays $b_\xi\unlhd^n a_\xi$ such that:
\begin{itemize}
\item[(i)] $b_\xi\unlhd^0 r_\xi$, and
\item[(ii)] if $A^{r_\xi}_0\cap U\neq \emptyset$, then $b_\xi\in U$.
\end{itemize}

Let $a'\unlhd^0 a$ be a condition witnessing the diagonalizability of $\langle b_\xi\mid \xi<\chi\rangle$.
Set $p:=\pi(a)$ and $p':=\pi(a')$.
By Fact~\ref{forkingfacts}, $W(a)=\fork{a}``W(p)$,
hence, for each $q\leq^n p$, we may let $\xi(q)<\chi$ be such that $\fork{a}(w(p,q))=r_{\xi(q)}$.
Set $\bar{U} := \{q\in P_n^p\mid b_{{\xi(q)}}\in U\}$.
As $\xi(q')= \xi(q)$ whenever $q'\leq^0 q\leq^n p$, the set $\bar{U}$ is $0$-open.
Recalling Setup~\ref{SPS}, $(\mathbb{P},\ell_\mathbb{P},c)$ is $\Sigma$-Prikry,
so applying $\CPP$ to $\bar{U}$,  $p'$, and $n$, we find $\bar p\le^0 p'$ such that either  $P^{\bar p}_n\s\bar{U}$ or $P^{\bar p}_n\cap \bar{U}=\emptyset$.

Set $\bar a:=\fork{a'}(\bar p)$.
Since $\bar p \le^0 p'\le^0 p$,  Clauses \eqref{frk1} and \eqref{frk0} of Definition~\ref{forking}  yield $\bar a \unlhd^0 a'\unlhd^0 a$.

\begin{claim} Let $b\in A^{\bar a}_n$.
Then:
\begin{enumerate}
\item $b\unlhd^0 b_{\xi(\pi(b))}$;
\item If $b\in U$, then $P^{\bar p}_n\s\bar{U}$.
\end{enumerate}
\end{claim}
\begin{proof} Denote $q:=\pi(b)$.

(1) Since $w(a',b)\in W_n(a')$
and $a'$ is a witness to diagonalizability of $\langle b_\xi\mid \xi<\chi\rangle$,
$b\unlhd^0 w(a',b)\unlhd^0 b_{\xi}$,
where $\xi$ is the unique index to satisfy $r_\xi=w(a,b)$.
By Clause~\eqref{frk4} of Definition~\ref{forking}, $$r_\xi= w(a,b)=\fork{a}(w(p,q))=r_{\xi(q)},$$
so that $\xi=\xi(\pi(b))$.

(2) Assuming that $b\in U$, we altogether infer that
$b\in A^{r_{\xi(q)}}_0\cap U$,
and then Clause~(ii) above implies that $b_{\xi(q)}\in U$.
By the definition of $\bar U$, then, $q\in\bar U\cap P^{\bar p}_n$.
So, by the choice of $\bar p$, furthermore $P^{\bar p}_n\s\bar{U}$.
\end{proof}

It thus follows that if $A^{\bar a}_n\cap U\neq\emptyset$,
then for every $b\in A^{\bar a}_n$, $\pi(b)\in P_n^{\bar p}\s\bar U$, so that
$b_{\xi(\pi(b))}\in U$. By the preceding claim, $b\unlhd^0 b_{\xi(\pi(b))}$,
so, since $U$ is $0$-open, $b\in U$.
Thus we have shown that if $A^{\bar a}_n\cap U\neq\emptyset$, then $A^{\bar a}_n\s U$.
\end{proof}

\begin{prop}\label{TheWitnessOfMixingYieldsDiagona}
Let $a\in A$, $n<\omega$ and $\vec{s}=\langle s_\xi\mid \xi<\chi\rangle$ be a good enumeration of $W_n(a)$.  
Let $p'\le^0 \pi(a)$.

Suppose that $\langle b_\xi\mid \xi<\chi\rangle$ is a sequence of conditions in $\conea{a}$ such that:
\begin{enumerate}
\item[$(\alpha)$] $\langle \pi(b_\xi)\mid \xi<\chi\rangle$ is diagonalizable with respect to $\langle \pi(s_\xi)\mid \xi<\chi\rangle$, as witnessed by $p'$;\footnote{By Fact~\ref{forkingfacts},  $\langle \pi(s_\xi)\mid \xi<\chi\rangle$ is a good enumeration of $W_n(\pi(a))$.}
\item[$(\beta)$]  $b$ is a condition in $\mathbb A$ with $\pi(b)=p'$
such that, for all  $q'\in W_n(p')$, $$\fork{b}(q')\unlhd^0 b_{\xi},$$
where $\xi$ is the unique index such that $\pi(s_\xi)=w(\pi(a), q')$.
\end{enumerate}
Then $b$ witnesses that $\langle b_\xi\mid \xi<\chi\rangle$ is diagonalizable with respect to $\vec{s}$.
\end{prop}
\begin{proof} 
We go over the two clauses of Definition~\ref{Diagonalizability}:
\begin{itemize}
\item[(a)]

Let $\xi<\chi$. By Clause~$(\alpha)$ above,  $\pi(b_\xi)\le^0 \pi(s_\xi)$. Together with Definition~\ref{SigmaPrikry}\eqref{itsaprojection}, it follows that $$w(\pi(a), \pi(b_\xi))\le^0 w(\pi(a), \pi(s_\xi))=\pi(s_\xi).$$
Finally, Clauses~\eqref{frk1}, \eqref{frk4} and \eqref{frk5} of Definition~\ref{forking} yield
$$b_\xi\unlhd^0 w(a,b_\xi)=\fork{a}(w(\pi(a), \pi(b_\xi)))\unlhd^0 \fork{a}(\pi(s_\xi))=s_\xi.$$

\item[(b)]

Let $b'\in W_n(b)$, and we shall show that
$b'\unlhd^0 b_\xi$, where $\xi$ is the unique index to satisfy $s_\xi=w(a, b')$.
Set $q':=\pi(b')$. As $\pi(b)=p'$,
we infer from Definition~\ref{forking}\eqref{frk4} that $b'=\fork{b}(q')$ and $q'\in W_n(p')$. 
Thus, by Clause~$(\beta)$ above $b'=\fork{b}(q')\unlhd^0 b_\xi$, where $\xi$ is the unique index such that $\pi(s_\xi)=w(\pi(a),q')$.  
Again by Definition~\ref{forking}\eqref{frk4},  $$s_\xi=\fork{a}(\pi(s_\xi))=\fork{a}(w(\pi(a),q'))=w(a,b'),$$
as desired.\qedhere
\end{itemize}
\end{proof}

\subsection{Types and the Weak Mixing Property} In this subsection, we will provide a sufficient condition for $(\mathbb A,\ell_\mathbb{A})$ to inherit property $\mathcal D$
from $(\mathbb P,\ell_\mathbb{P})$.

While reading the next two definitions, the reader may want to have a simple example in mind;
such an example is given by Lemma~\ref{trivialfunctortp} below.

 \begin{definition}[Types]\label{type}
A \emph{type over $(\pitchfork,\pi)$}
is a map $\tp\colon A\rightarrow{}^{<\mu}\omega$ having the following properties:
\begin{enumerate}
\item\label{type1} for each $a\in A$, either $\dom(\tp(a))=\alpha+1$ for some $\alpha<\mu$, in which case we define $\mtp(a):=\tp(a)(\alpha)$,
or $\tp(a)$ is empty, in which case we define $\mtp(a):=0$;
\item\label{type2} for all $a,b\in A$ with $b\unlhd a$, $\dom(\tp(a))\leq \dom(\tp(b))$ and for each $i\in\dom(\tp(a))$, $\tp(b)(i)\leq \tp(a)(i)$;
\item\label{type3} for all $a\in A$ and $q\le \pi(a)$, $\dom(\tp(\fork{a}(q)))=\dom(\tp(a))$;
\item\label{type6} for all $a\in A$,  $\tp(a)=\emptyset$ iff $a=\myceil{\pi(a)}{\mathbb{A}}$; 
\item\label{type4} for all $a\in A$ and  $\alpha\in \mu\setminus \dom(\tp(a))$, there exists a \emph{stretch of $a$ to $\alpha$}, denoted $a{}^{\curvearrowright\alpha}$,
and satisfying the following:
\begin{enumerate}[label=(\alph*)]
\item $a{}^{\curvearrowright\alpha}\unlhd^\pi a$;
\item  $\dom(\tp(a{}^{\curvearrowright\alpha}))=\alpha+1$;
\item  $\tp(a{}^{\curvearrowright\alpha})(i)\leq \mtp(a)$ whenever $\dom(\tp(a))\le i\le\alpha$;
\end{enumerate}
\item \label{newstretch} for all $a,b\in A$ with $\dom(\tp(a))=\dom(\tp(b))$, for every $\alpha\in \mu\setminus \dom(\tp(a))$,
if $b\unlhd a$, then
$b{}^{\curvearrowright\alpha}\unlhd a{}^{\curvearrowright\alpha}$;
\item \label{type5} For each $n<\omega$, the poset $\z{\mathbb{A}}_n$ is dense in $\mathbb{A}_n$, where $\z{\mathbb{A}}_n:=(\z{A}_n,\unlhd)$ and $\z{A}_n:=\{a\in A_n\mid \pi(a)\in\z{P}_n\ \&\ \mtp(a)=0\}$.
\end{enumerate}
 \end{definition}
 \begin{remark}\label{RemarkType} Note that Clauses \eqref{type2} and \eqref{type3} imply that for all $m,n<\omega$, $a\in\z A_m$
 and $q \le\pi(a)$, if $q\in\z P_n$ then $\fork{a}(q)\in\z A_n$.
 \end{remark}

The next definition is a weakening of \cite[Definition~4.11]{partI}.

 \begin{definition}[Weak Mixing Property]\label{mixingproperty}
The forking projection $(\pitchfork,\pi)$ is said to have the \emph{weak mixing property} iff
it admits a type $\tp$ satisfying that 
for all  $n<\omega$,
$a\in A$, $\vec r$, and $p'\le^0 \pi(a)$,
and for every function $g:W_n(\pi(a))\rightarrow \mathbb{A}\downarrow a$, if there exists an ordinal $\iota$ such that all of the following hold:\footnote{The ordinal $\iota$ would help us keep track of the support when appealing to the weak mixing property in an iteration (see, e.g., Lemma \ref{MixingforLimits} and Claim~\ref{ClaimPrepropertyD}).}
\begin{enumerate}
\item\label{Mixing1} $\vec r=\langle r_\xi \mid \xi<\chi\rangle$ is a good enumeration of $W_n(\pi(a))$;
\item\label{Mixing2}  $\langle \pi(g(r_\xi))\mid \xi<\chi\rangle$ is diagonalizable with respect to $\vec r$, as witnessed by $p'$;\footnote{In particular, $\pi(g(r_\xi))\le^0 r_\xi$ and $\lh_{\mathbb A}(g(r_\xi))=\lh_{\mathbb A}(a)+n$ for every $\xi<\chi$.}
\item\label{Mixing3} for every $\xi<\chi$:
\begin{itemize}
\item if $\xi<\iota$, then $\dom(\tp(g(r_\xi)))=0$;
\item if $\xi=\iota$, then $\dom(\tp(g(r_\xi)))\geq 1$;
\item if $\xi>\iota$, then $\dom(\tp(g(r_\xi)))>(\sup_{\eta<\xi}\dom(\tp(g(r_\eta))))+1$;
\end{itemize}
\item\label{Mixing4} for all $\xi\in(\iota,\chi)$ and $i\in[\dom(\tp(a)),\sup_{\eta<\xi}\dom(\tp(g(r_\eta)))]$,
$$\tp(g(r_\xi))(i)\leq\mtp(a),$$

\item\label{Mixing5} $\sup_{\xi<\chi} \mtp(g(r_\xi))<\omega$,
\end{enumerate}
then there exists $b\unlhd^0 a$ with $\pi(b)=p'$ such that, for all  $q'\in W_n(p')$, $$\fork{b}(q')\unlhd^0 g(w(\pi(a),q')).$$
\end{definition}

\begin{lemma}\label{trivialfunctortp} The forking projection $(\pitchfork,\pi)$ from Example~\ref{trivialfunctor}
has the weak mixing property.
\end{lemma}
\begin{proof} We attach a type $\tp:A\rightarrow{}^{<\mu}\omega$ as follows.
For every $a=(p,\alpha)\in A$, with $\alpha>0$, let $\tp(a)$ be the constant $(\alpha+1)$-sequence whose sole value is $0$. Otherwise, let $\tp(a):=\emptyset$.
We shall verify that $\tp$ witnesses that $(\pitchfork,\pi)$ has the weak mixing property.
To this end, suppose that we are given $n<\omega$, $a\in A$, $\vec r=\langle r_\xi \mid \xi<\chi\rangle$, $p'\le^0 \pi(a)$,
a function $g:W_n(\pi(a))\rightarrow {\mathbb{A}}\downarrow a$ and an ordinal $\iota$ satisfying Clauses (1)--(4) of Definition~\ref{mixingproperty}.
For each $\xi<\chi$, write $(q_\xi,\alpha_\xi):=g(r_\xi)$. Note that by Clause~\eqref{type6} of  Definition~\ref{type} and Example~\ref{trivialfunctor}, $\alpha_\xi=0$ for all $\xi<\iota$.

Set $b:=(p',\alpha')$, for $\alpha':=\sup_{\iota\leq \xi<\chi}\alpha_\xi$. Clearly, $b\unlhd^0 a$. Note that, by regularity of $\mu$, $\alpha'<\mu$.
Now, since $p'$ witnesses that $\langle q_\xi\mid \xi<\chi\rangle$ is diagonalizable,
for every $q'\in W_n(p')$, if we let $\xi$ denote the unique index to satisfy $r_\xi=w(\pi(a), q')$,
then $q'\leq^0 q_\xi$.
As $\alpha'\ge\alpha_\xi$, it altogether follows that
$(q',\alpha')=\fork{b}(q')\unlhd^0 g(w(\pi(a),q'))=(q_\xi,\alpha_\xi)$.
\end{proof}

\begin{lemma}\label{MixingLiftsPropertyD}
Suppose that $(\pitchfork,\pi)$  has the weak mixing property and that $(\mathbb{P},\ell_\mathbb{P})$ has property $\mathcal{D}$. Then $(\mathbb{A},\ell_\mathbb{A})$ has  property $\mathcal{D}$, as well.
\end{lemma}
\begin{proof}
Let $a\in A$ and $n<\omega$. Let $\vec{s}=\langle s_\xi\mid \xi<\chi\rangle$ be a good enumeration of $W_n(a)$. 
By Lemma~\ref{propertyDatdense} and Definition~\ref{type}\eqref{type5}, it suffices to show that $\pI$ has a winning strategy  in $\Game_\mathbb{A}(a,\vec{s},D)$,
where $D:=\z{A}_{\ell_\mathbb{A}(a)+n}$. For each $\xi<\chi$, let $r_\xi:=\pi(s_\xi)$.
By Fact~\ref{forkingfacts}, 
$s_\xi=\fork{a}(r_\xi)$,
and  $\vec{r}:=\langle r_\xi \mid \xi<\chi\rangle$ forms a good enumeration  of $W_n(\pi(a))$.

Fix any type $\tp$ witnessing the weak mixing property of $(\pitchfork,\pi)$. 
We shall describe a winning strategy for $\pI$ in the game $\Game_\mathbb{A}(a,\vec{s},D)$
by producing sequences of the form $\langle (p_\eta,a_\eta,b_\eta,q_\eta)\mid \eta<\xi\rangle$,
where $\langle (a_\eta,b_\eta)\mid \eta<\xi\rangle$ is an initial play (consisting of $\xi$ rounds)  in the game
$\Game_\mathbb{A}(a,\vec{s},D)$,
and $\langle (p_\eta,q_\eta)\mid \eta<\xi\rangle$ is an initial play in the game $\Game_\mathbb{P}(\pi(a),\vec{r})$. Roughly speaking, the idea is to design the moves of $\pI$ (i.e., the $a_\eta$'s) so that they force $\pII$ to play conditions $b_\eta$ in such a way that the map $s_\eta\stackrel{g}{\mapsto} b_\eta$ satisfies the requirements of Definition~\ref{mixingproperty}; most notably, Clauses~\eqref{Mixing3} and \eqref{Mixing4}. To comply with this we shall do suitable stretches when defining the conditions $a_\eta$'s.

$\br$ For $\xi=0$,
we first play a condition $p_0$ according to the winning strategy for $\pI$ in the game $\Game_\mathbb{P}(\pi(a),\vec{r})$.
In particular, $p_0\le^0\pi(a)$. 
As $p_0$ is compatible with $r_0$,
fix a condition $r'\le p_0,r_0$, and note that it follows from Definition~\ref{forking}\eqref{frk0} that $\fork{a}(r')\unlhd \fork{a}(p_0),\fork{a}(r_0)$.  Now set $\bar{\alpha}_0:=\dom(\tp(a))+1$ and  $a_0:=\fork{a}(p_0){}^{\curvearrowright \bar{\alpha}_0}$. By Definition~\ref{type}\eqref{type4},  $a_0\unlhd^0 a$, $\pi(a_0)=p_0$, and it can also be shown that $a_0$ is compatible with $s_0$. Indeed, by Clauses \eqref{type4} and \eqref{newstretch} of Definition~\ref{type},
we infer that:
\begin{itemize}
\item $a_0 \unlhd^\pi \fork{a}(p_0)$;
\item $\dom(\tp(a_0))=\bar\alpha_0+1$;
\item $\tp(a_0)(i)\le\mtp(\fork{a}(p_0))$, whenever $\dom(\tp(\fork{a}(p_0)))\le i\le\bar\alpha_0$;
\item $\fork{a}(r'){}^{\curvearrowright\bar\alpha_0}\mathrel{\unlhd}\fork{a}(p_0){}^{\curvearrowright\bar\alpha_0}=a_0$
and $\fork{a}(r'){}^{\curvearrowright\bar\alpha_0}\mathrel{\unlhd}\fork{a}(r')\allowbreak\mathrel{\unlhd}s_0$.
\end{itemize}
Thus, $a_0$ is compatible with $s_0$. Next, let $\pII$ play $b_0\in D$ at will, subject to ensuring that $b_0\unlhd a_0$ and $b_0\unlhd^0 s_0$. Finally, let $q_0:=\pi(b_0)$.

$\br$ Suppose that $\xi<\chi$ is nonzero and that $\langle (p_\eta,a_\eta,b_\eta,q_\eta)\mid \eta<\xi\rangle$ has already been defined.
Let $p_\xi$ be given by the winning strategy for $\pI$ in the game $\Game_\mathbb{P}(\pi(a),\vec{r})$
with respect to the initial play $\langle (p_\eta,q_\eta)\mid \eta<\xi\rangle$.
As in the previous case, we may fix a condition $r'$ such that
$\fork{a}(r')\unlhd\fork{a}(p_\xi),s_\xi$.

Set $\bar{\alpha}_\xi:=(\sup_{\eta<\xi}\dom(\tp(b_\eta)))+1$. Then, by Clauses \eqref{type4} and \eqref{newstretch} of Definition~\ref{type},
we may let $a_\xi:=\fork{a}(p_\xi){}^{\curvearrowright\bar\alpha_\xi}$, and  argue as before that $a_\xi$ is compatible with $s_\xi$. Also, note that $a_\xi\unlhd^0 a$ and $\pi(a_\xi)=p_\xi$. Next, let $\pII$ play any $b_\xi\in D$ such that $b_\xi\unlhd a_\xi$ and $b_\xi\unlhd^0 s_\xi$. Finally, let $q_\xi:=\pi(b_\xi)$. 

At the end of the game,
we have produced a sequence $\langle (p_\xi,a_\xi,b_\xi,q_\xi)\mid \xi<\chi\rangle$.
Since $\langle (p_\xi,q_\xi)\mid \xi<\chi\rangle$ is the outcome of a $\Game_{\mathbb P}(\pi(a),\vec r)$-game
in which $\pI$ played according to a winning strategy, we may fix $p'\le^0\pi(a)$ witnessing that $\langle q_\xi\mid \xi<\chi\rangle$ is diagonalizable.

It follows that if we define a function $g:W_n(\pi(a))\rightarrow D$ via $g(r_\xi):=b_\xi$, then all the requirements of Definition~\ref{mixingproperty} are fulfilled with respect to $\iota:=0$
(Note that we have secured that $\dom(\tp(a_\xi))>0$ for all $\xi<\chi$).
For instance, to see that Clause~\eqref{Mixing4} of Definition~\ref{mixingproperty}  holds,
notice that by Clauses \eqref{type2} and \eqref{type4} of Definition~\ref{type}, for all $\xi<\chi$ and $i\in[\dom(\tp(a)),\dom(\tp(a_\xi)))$,
$$\tp(b_\xi)(i)\le\tp(a_\xi)(i)\le\mtp(\fork{a}(p_\xi))\le \mtp(a).$$

Consequently, we may pick $b\unlhd^0 a$ with $\pi(b)=p'$ such that for all
 $q'\in W_n(p')$, $$\fork{b}(q')\unlhd^0 g(w(\pi(a),q')).$$
By definition, for each $q'\in W_n(p')$, $g(w(\pi(a),q'))=b_\xi$, where $\xi$ is the unique index such that $\pi(s_\xi)=w(\pi(a),q')$. Therefore, invoking Proposition~\ref{TheWitnessOfMixingYieldsDiagona} we infer that $b$ diagonalizes $\langle b_\xi\mid \xi<\chi\rangle$, as desired.
\end{proof}

\begin{cor}\label{PropertyDplusMixingYieldsCPP}

If $(\mathbb{P},\ell_\mathbb{P})$ has property $\mathcal{D}$,
and $({\pitchfork},{\pi})$ has the weak mixing property,
then $(\mathbb{A},\ell_\mathbb{A})$ has the $\CPP$.
\end{cor}
\begin{proof}
By Lemmas \ref{propertyDyieldsCPP} and \ref{MixingLiftsPropertyD}.
\end{proof}

\begin{lemma}\label{forkinganddirectedclosure}
Suppose that $({\pitchfork},{\pi})$ is as in Setup~\ref{SPS} or,  
just a pair of maps satisfying Clauses \eqref{frk1}, \eqref{frk0}, \eqref{frk5} and \eqref{frk7} of Definition~\ref{forking}.

Let $n<\omega$.
Assuming that $(\pitchfork,\pi)$ admits a type, and $\z{\mathbb{A}}_n$ is defined according to the last clause of Definition~\ref{type},
if $\z{\mathbb A}_n^\pi$ is $\kappa_n$-directed-closed, then so is $\z{\mathbb{A}}_n$.
\end{lemma}
\begin{proof} The proof is very similar to that of \cite[Lemma~4.6]{partI},
bearing Remark~\ref{RemarkType} in mind.
\end{proof}

\section{Iteration Scheme}\label{Iteration}

In this section, we present our iteration scheme for $\Sigma$-Prikry posets.
Throughout the section, assume that $\Sigma=\langle \kappa_n\mid n<\omega\rangle$ is a non-decreasing sequence of regular uncountable cardinals.
Denote $\kappa:=\sup_{n<\omega}\kappa_n$.
Also, assume that $\mu$ is some cardinal satisfying $\mu^{<\mu}=\mu$, so that $|H_\mu|=\mu$.

The following convention will be applied hereafter:
\begin{conv}\label{convention31}  For all ordinals $\alpha\le\delta\le\mu^+$:
\begin{enumerate}
\item  $\emptyset_\delta:=\delta\times\{\emptyset\}$ denotes the $\delta$-sequence with constant value $\emptyset$;
\item  For an $\alpha$-sequence $p$ and a $\delta$-sequence $q$, $p*q$ denotes the unique $\delta$-sequence satisfying that for all $\beta<\delta$:
$$(p*q)(\beta)=\begin{cases}
q(\beta),&\text{if }\alpha\le\beta<\delta;\\
p(\beta),&\text{otherwise}.
\end{cases}$$
\item Let $\mathbb{P}_\delta:=(P_\delta,\le_\delta)$ and $\mathbb{P}_\alpha:=(P_\alpha,\le_\alpha)$ be forcing posets such that $P_\delta\s{}^\delta{H_{\mu^+}}$ and $P_\alpha\s{}^\alpha{H_{\mu^+}}$. Also, assume $p\mapsto p\restriction\alpha$ defines a projection between $\mathbb{P}_\delta$ and $\mathbb{P}_\alpha$. We denote by $i^\delta_\alpha: V^{\mathbb{P}_\alpha}\rightarrow V^{\mathbb{P}_\delta}$ the map defined by recursion over the rank of each $\mathbb{P}_\alpha$-name $\sigma$ as follows:
$$i^\delta_\alpha(\sigma):=\{(i^\delta_\alpha(\tau),p*\emptyset_\delta)\mid (\tau,p)\in \sigma\}.$$
\end{enumerate}
\end{conv}

Our iteration scheme requires three building blocks:

\blk{I}
We are given a $\Sigma$-Prikry triple $(\mathbb{Q},\lh,c)$
such that $\mathbb Q=(Q,\le_Q)$ is a subset of $H_{\mu^+}$, $\one_{\mathbb Q}\forces_{\mathbb Q}\check\mu=\kappa^+$ and $\one_{\mathbb Q}\forces_{\mathbb{Q}}``\kappa\text{ is singular}"$.\footnote{At the behest of the referee, we stress that the last hypothesis plays a rather isolated role; see Footnote~\ref{sole} on page~\pageref{sole}.} {Additionally, we  assume that $(\mathbb{Q},\ell)$ has property $\mathcal{D}$.}
To streamline the matter, we also require that $\one_{\mathbb Q}$ be equal to $\emptyset$.

\blk{II} For every $\Sigma$-Prikry triple $(\mathbb P,\lh_{\mathbb P},c_{\mathbb P})$ {having property $\mathcal{D}$}
such that $\mathbb P=\left(P,\le\right)$ is a subset of $H_{\mu^+}$,
$\one_{\mathbb P}\forces_{\mathbb P}\check\mu=\kappa^+$
and $\one_{\mathbb P}\forces_{\mathbb{P}}``\kappa\text{ is singular}"$,
every $r^\star\in P$, and every $\mathbb P$-name $z\in H_{\mu^+}$,
we are given a corresponding $\Sigma$-Prikry triple $(\mathbb A,\lh_{\mathbb A},c_{\mathbb A})$ such that:
\begin{enumerate}[label=(\alph*)]
\item\label{C1blk2} $(\mathbb A,\lh_{\mathbb A},c_{\mathbb A})$ admits a forking projection $({\pitchfork},{\pi})$ to $(\mathbb P,\lh_{\mathbb P},c_{\mathbb P})$ that has the weak mixing property;\footnote{So, by Lemma~\ref{MixingLiftsPropertyD},  $(\mathbb A,\lh_{\mathbb A})$  has property $\mathcal{D}$, as well.} 
\item \label{Cnewblk2} for each $n<\omega$, $\z{\mathbb{A}}^\pi_n$  is $\kappa_n$-directed-closed;\footnote{\label{FootnoteBBII}$\z{A}_n$ denotes the poset of Definition~\ref{type}\eqref{type5} regarded with respect to the type witnessing Clause~\ref{C1blk2} of \blkref{II}.}
\item\label{C2blk2} $\one_{\mathbb A}\forces_{\mathbb A}\check\mu=\kappa^+$;
\item\label{C3blk2} $\mathbb A=(A,\unlhd)$ is a subset of $H_{\mu^+}$.
\setcounter{condition}{\value{enumi}}
\end{enumerate}
By Lemma~\ref{canonical}, we may streamline the matter, and also require that:
\begin{enumerate}[label=(\alph*)]
\setcounter{enumi}{\value{condition}}
\item\label{C5blk2} each element of $A$ is a pair $(x,y)$ with $\pi(x,y)=x$;
\item\label{C6blk2} for every $a\in A$, $\myceil{\pi(a)}{\mathbb A}=(\pi(a),\emptyset)$;
\item\label{C7blk2} for every $p,q\in P$, if $c_{\mathbb P}(p)=c_{\mathbb P}(q)$, then $c_{\mathbb A}(\myceil{p}{\mathbb A})=c_{\mathbb A}(\myceil{q}{\mathbb A})$.
\end{enumerate}

\blk{III}
We are given a function $\psi:\mu^+\rightarrow H_{\mu^+}$.

\begin{goal}\label{goals}
Our goal is to define a system $\langle  (\mathbb{P}_\delta,\lh_\delta,c_\delta,\langle\pitchfork_{\delta,\gamma}\mid \gamma\le\delta\rangle)\mid \delta\le\mu^+\rangle$ in such a way that
for all $\gamma\le\delta\le\mu^+$:
\begin{enumerate}
\item[(i)] $\mathbb P_\delta$ is a poset $(P_\delta,\le_\delta)$, $P_\delta\s{}^{\delta}H_{\mu^+}$,
and, for all $p\in P_\delta$,  $|B_p|<\mu$, where $B_p:=\{ \beta+1\mid \beta\in\dom(p)\ \&\ p(\beta)\neq\emptyset\}$;
\item[(ii)] The map $\pi_{\delta,\gamma}:P_\delta\rightarrow P_\gamma$ defined by $\pi_{\delta,\gamma}(p):=p\restriction\gamma$ forms a projection from $\mathbb P_\delta$ to $\mathbb P_\gamma$ and $\lh_\delta= \lh_\gamma\circ \pi_{\delta,\gamma}$; 
\item[(iii)] $\mathbb P_0$ is a trivial forcing, $\mathbb P_1$ is isomorphic to $\mathbb Q$ given by \blkref{I},
and $\mathbb P_{\delta+1}$ is isomorphic to $\mathbb A$ given by \blkref{II} when invoked with $(\mathbb P_\delta,\lh_\delta,c_\delta)$ and a pair $(r^\star,z)$ which is decoded from $\psi(\delta$);
\item[(iv)] If $\delta>0$, then $(\mathbb P_\delta,\lh_\delta,c_\delta)$ is a $\Sigma$-Prikry triple {having property $\mathcal{D}$} whose greatest element is $\emptyset_\delta$, $\lh_\delta=\lh_1\circ \pi_{\delta,1}$, and $\emptyset_\delta\forces_{\mathbb P_\delta}\check\mu=\kappa^+$;
\item[(v)] If $0<\gamma< \delta\le\mu^+$, then $({\pitchfork_{\delta,\gamma}},{\pi_{\delta,\gamma}})$ is a forking projection from $(\mathbb P_\delta,\lh_\delta)$ to $(\mathbb P_\gamma,\lh_\gamma)$;
in case $\delta<\mu^+$, $({\pitchfork_{\delta,\gamma}},{\pi_{\delta,\gamma}})$  is furthermore a forking projection from $(\mathbb P_\delta,\lh_\delta,c_\delta)$ to $(\mathbb P_\gamma,\lh_\gamma,c_\gamma)$, and
in case $\delta=\gamma+1$, $(\pitchfork_{\delta,\gamma},\pi_{\delta,\gamma})$ has the weak mixing property;
\item[(vi)] If $0<\alpha\le\beta\le\delta$, then, for all $p\in P_{\delta}$ and $r\le_\alpha p\restriction\alpha$, $\fork[\beta,\alpha]{p\restriction\beta}(r)=(\fork[\delta,\alpha]{p}(r))\restriction \beta$.
\end{enumerate}
\end{goal}

\begin{remark}\label{remark43} Note the asymmetry between the cases $\delta<\mu^+$ and  $\delta=\mu^+$:
\begin{enumerate}
\item By Clause~(i), we will have that $\mathbb P_\delta\s H_{\mu^+}$ for all $\delta<\mu^+$, but $\mathbb P_{\mu^+}\nsubseteq H_{\mu^+}$.
Still, $\mathbb P_{\mu^+}$ will nevertheless be isomorphic to a subset of $H_{\mu^+}$, as
we may identify $P_{\mu^+}$ with $\{ p\restriction(\sup(B_p)+1)\mid p\in P_{\mu^+}\}$.
\item Clause~(v) puts a weaker assertion for $\delta=\mu^+$. In order to avoid trivialities, let us assume that $\mu^+$-many stages in our iteration $\mathbb{P}_{\mu^+}$ are non-trivial.
To see  the restriction in Clause~(v)  is necessary
note that, by the pigeonhole principle, there must exist two conditions $p,q\in  P_{\mu^+}$
and an ordinal $\gamma<\mu^+$ for which
$c_{\mu^+}(p)=c_{\mu^+}(q)$,  $B_p\s\gamma$, but $B_q\nsubseteq\gamma$.
Now, towards a contradiction, assume there is a map $\pitchfork$ such that
$({\pitchfork},{\pi_{{\mu^+},\gamma}})$ forms a forking projection from $(\mathbb P_{\mu^+},\lh_{\mu^+},c_{\mu^+})$ to $(\mathbb P_\gamma,\lh_\gamma,c_\gamma)$.
By Definition~\ref{forking}\eqref{frk2}, then, $c_{\gamma}(p\restriction\gamma)=c_{\gamma}(q\restriction\gamma)$,
so that by Definition~\ref{SigmaPrikry}\eqref{c1}, we should be able to pick $r\in (P_{\gamma})^{p\restriction\gamma}_0\cap(P_{\gamma})^{q\restriction\gamma}_0$,
and then by Definition~\ref{forking}\eqref{frk2}, $\fork{p}(r)=\fork{q}(r)$.
Finally, as $B_p\s\gamma$, $p=\myceil{p\restriction\gamma}{\mathbb P_{\mu^+}}$,\footnote{This is a consequence of the fact that $p=(p\restriction\gamma)*\emptyset_{\mu^+}=\myceil{p\restriction\gamma}{\mathbb{P}_{\mu^+}}$. See the discussion at the beginning of Lemma~\ref{CvIteration}.}
so that, by Definition~\ref{forking}\eqref{frk6}, $\fork{p}(r)=\myceil{r}{\mathbb P_{\mu^+}}$.
But then $\fork{q}(r)=\myceil{r}{\mathbb P_{\mu^+}}$,
so that, by Definition~\ref{forking}\eqref{frk6}, $q=\myceil{q\restriction\gamma}{\mathbb P_{\mu^+}}$,
contradicting the fact that $B_q\nsubseteq\gamma$.
\end{enumerate}
\end{remark}

\subsection{Defining the iteration}\label{DefiningTheIteration}
For every $\delta<\mu^+$, fix an injection $\phi_\delta:\delta\rightarrow\mu$.
As $|H_\mu|=\mu$, by the Engelking-Kar\l owicz theorem, we may also fix a sequence $\langle e^i\mid i<\mu\rangle$ of functions from $\mu^+$ to $H_\mu$ such that
for every function $e:C\rightarrow H_\mu$ with $C\in[\mu^+]^{<\mu}$, there is $i<\mu$ such that $e\s e^i$.

The upcoming definition is by recursion on $\delta\le\mu^+$, and we continue as long as we are successful. We shall later verify that the described process is indeed successful.

$\br$ Let $\mathbb P_0:=(\{\emptyset\},\le_0)$ be the trivial forcing.
Let $\lh_0$ and $c_0$ be the constant function $\{(\emptyset,\emptyset)\}$,
and let $\pitchfork_{0,0}$ be the constant function $\{ (\emptyset,\{(\emptyset,\emptyset)\})\}$,
so that $\fork[0,0]{\emptyset}$ is the identity map.

$\br$ Let $\mathbb P_1:=(P_1,\le_1)$, where $P_1:={}^1Q$ and  $p\le_1 p'$ iff $p(0)\le_Q p'(0)$.
Define $\lh_{1}$ and $c_1$ by stipulating $\lh_1(p):=\lh(p(0))$ and $c_1(p):=c(p(0))$.
For all $p\in P_1$, let $\fork[1,0]{p}:\{\emptyset\}\rightarrow\{p\}$ be the constant function,
and let $\fork[1,1]{p}$ be the identity map.

$\br$ Suppose $\delta<\mu^+$ and that $\langle (\mathbb P_\beta,\lh_\beta,c_\beta,\langle \pitchfork_{\beta,\gamma}\mid\gamma\le\beta\rangle)\mid\beta\le\delta\rangle$ has already been defined.
We now define the triple $(\mathbb{P}_{\delta+1},\lh_{\delta+1}, c_{\delta+1})$ and the sequence of maps $\langle \pitchfork_{\delta+1,\gamma}\mid\gamma\leq \delta+1\rangle$, as follows.

$\br\br$ If $\psi(\delta)$ happens to be a triple $(\beta,r,\sigma)$, where $\beta<\delta$, $r\in P_\beta$ and $\sigma$ is a $\mathbb P_\beta$-name,
then we appeal to \blkref{II} with $(\mathbb P_\delta,\lh_\delta,c_\delta)$,
$r^\star:=r*\emptyset_\delta$ and $z:=i^\delta_\beta(\sigma)$
to get a corresponding $\Sigma$-Prikry poset $(\mathbb A,\lh_{\mathbb A},c_{\mathbb A})$.

$\br\br$ Otherwise, we obtain $(\mathbb A,\lh_{\mathbb A},c_{\mathbb A})$ by appealing to \blkref{II} with $(\mathbb P_\delta,\lh_\delta,c_\delta)$, $r^\star:=\emptyset_\delta$ and $z:=\emptyset$.

In both cases, we also obtain a forking projection $({\pitchfork},{\pi})$ from $(\mathbb A,\lh_{\mathbb A},c_{\mathbb A})$ to $(\mathbb P_\delta,\lh_\delta,c_\delta)$.
Furthermore, each condition in $\mathbb A=(A,\unlhd)$ is a pair $(x,y)$ with $\pi(x,y)=x$, and, for every $p\in P_\delta$, $\myceil{p}{\mathbb A}=(p,\emptyset)$.
Now, define $\mathbb P_{\delta+1}:=(P_{\delta+1},\le_{\delta+1})$ by letting $P_{\delta+1}:=\{ x {}^\smallfrown\langle y\rangle \mid (x,y)\in A\}$,
and then let $p\le_{\delta+1} p'$ iff $(p\restriction\delta,p(\delta))\unlhd (p'\restriction\delta,p'(\delta))$.
Put $\lh_{\delta+1}:=\lh_1\circ \pi_{\delta+1,1}$ and define $c_{\delta+1}:P_{\delta+1}\rightarrow H_\mu$ via $c_{\delta+1}(p):=c_{\mathbb A}(p\restriction\delta,p(\delta))$.

Next, let $p\in P_{\delta+1}$, $\gamma\le\delta+1$ and $r\le_\gamma p\restriction\gamma$ be arbitrary; we need to define $\fork[\delta+1,\gamma]{p}(r)$.
For $\gamma=\delta+1$, let $\fork[\delta+1,\gamma]{p}(r):=r$, and for $\gamma\le\delta$, let
\begin{equation}\label{PitchfrokSuccessor}
\tag{*}\fork[\delta+1,\gamma]{p}(r):=x{}^\smallfrown\langle y\rangle\text{ iff }\fork{p\restriction\delta,p(\delta)}(\fork[\delta,\gamma]{p\restriction\delta}(r))=(x,y).
\end{equation}

$\br$ Suppose $\delta\in\acc(\mu^{+}+1)$,
and that $\langle (\mathbb P_\beta,\lh_\beta,c_\beta,\langle \pitchfork_{\beta,\gamma}\mid\gamma\le\beta\rangle)\mid\beta<\delta\rangle$ has already been defined.
Define $\mathbb P_\delta:=(P_\delta,\le_\delta)$ by letting $P_\delta$ be all $\delta$-sequences $p$
such that $|B_p|<\mu$ and $\forall\beta<\delta(p\restriction\beta\in P_\beta)$.
Let $p\le_{\delta} q$ iff  $\forall{\beta<\delta}(p\restriction \beta\le_{\beta} q\restriction \beta)$. Let $\lh_\delta:=\lh_1\circ\pi_{\delta,1}$.
Next, we define $c_\delta:P_\delta\rightarrow H_\mu$, as follows.

$\br\br$ If $\delta<\mu^+$, then, for every $p\in P_\delta$, let
$$c_\delta(p):=\{ (\phi_\delta(\gamma),c_{\gamma}(p\restriction\gamma))\mid \gamma\in B_p\}.$$

$\br\br$ If $\delta=\mu^+$, then, given $p\in P_\delta$,
first let $C:=\cl(B_p)$, then define a function $e:C\rightarrow H_\mu$ by stipulating:
$$e(\gamma):=(\phi_\gamma[C\cap\gamma],c_{\gamma}(p\restriction\gamma)),$$
and then let $c_\delta(p):=i$ for the least $i<\mu$ such that $e\s e^i$.

Finally, let $p\in P_{\delta}$, $\gamma\le\delta$ and $r\le_\gamma p\restriction\gamma$ be arbitrary; we need to define $\fork[\delta,\gamma]{p}(r)$.
For $\gamma=\delta$, let $\fork[\delta,\gamma]{p}(r):=r$,
and for $\gamma<\delta$, let
$\fork[\delta,\gamma]{p}(r):=\bigcup\{\fork[\beta,\gamma]{p\restriction\beta}(r)\mid \gamma\le\beta<\delta\}$.

\begin{conv} Even though $(\mathbb P_0,\lh_0)$ is not a graded poset,
in order to smooth up inductive claims that come later, we define $\le_0^0$ to be $\le_0$,
and likewise, for every $p\in P_0$, we interpret $(P_0)^{p}_0$ as $\{q\in P_0\mid q\le_0^0 p\}$.
\end{conv}

\subsection{Verification} We now verify that for all $\delta\le\mu^+$,
$(\mathbb{P}_\delta,\lh_\delta,c_\delta,\langle\pitchfork_{\delta,\gamma}\mid \gamma\le\delta\rangle)$ fulfills requirements (i)--(vi) of Goal~\ref{goals}.
By the recursive definition given so far, it is obvious that Clauses (i) and (iii) hold, so we
focus on the rest. We commence with an expanded version of Clause~(vi).

\begin{lemma}\label{CviIteration}
For all $\alpha\le\delta\leq  \mu^+$, $p\in P_\delta$ and $r\in P_\alpha$ with $r\le_\alpha p\restriction\alpha$, if we let $q:=\fork[\delta,\alpha]{p}(r)$, then:
\begin{enumerate}
\item\label{vi-like}  $q\restriction\beta=\fork[\beta,\alpha]{p\restriction\beta}(r)$ for all $\beta\in[\alpha,\delta]$;
\item\label{SupportForking} $B_{q}=B_p\cup B_r$;
\item\label{C5LemmaOfForking} $q\restriction\alpha=r$;
\item If $\alpha=0$, then $q=p$;
\item\label{CViforking} $p=(p\restriction\alpha)*\emptyset_\delta$ iff $q=r*\emptyset_\delta$;
\item\label{Cviiforking} for all $p'\le^0_\delta p$, if $r\le^0_\alpha p'\restriction\alpha$, then $\fork[\delta,\alpha]{p'}(r)\le_\delta\fork[\delta,\alpha]{p}(r)$.
\end{enumerate}
\end{lemma}
\begin{proof} Clause~(3) follows from Clause~(1) and the fact that $\fork[\alpha,\alpha]{p\restriction\alpha}$ is the identity function.
Clause~(5) follows from Clauses (2) and (3).

We now prove Clauses (1), (2), (4) and (6) by induction on $\delta\le\mu^+$:
\begin{itemize}
\item[$\br$] The case $\delta=0$ is trivial, since, in this case, all the conditions under consideration (and their corresponding $B$-sets) are empty,
and all the maps under consideration are the identity.

\item[$\br$] The case $\delta=1$ follows from the fact that, by definition, $\fork[1,0]{p}(r)=p$ and $\fork[1,1]{p}(r)=r$.

\item[$\br$] Suppose $\delta\ge2$ is a successor ordinal,
say $\delta=\delta'+1$, and that the claim holds for $\delta'$.
Fix arbitrary $\alpha\le\delta$, $p\in P_\delta$ and $r\in P_{\alpha}$ with $r\le_\alpha p\restriction\alpha$. Denote $q:=\fork[\delta,\alpha]{p}(r)$.
Recall that $\mathbb P_{\delta}=\mathbb P_{\delta'+1}$ was defined
by feeding $(\mathbb{P}_{\delta'},\lh_{\delta'},c_{\delta'})$ into \blkref{II},
thus obtaining a $\Sigma$-Prikry triple $(\mathbb A,\lh_{\mathbb A},c_{\mathbb A})$
along with a forking projection $({\pitchfork},{\pi})$,
such that each condition in the poset $\mathbb A=(A,\unlhd)$ is a pair $(x,y)$
with $\pi(x,y)=x$.
Furthermore, by the definition of $\pitchfork_{\delta,\alpha}$,
$q=\fork[\delta,\alpha]{p}(r)$ is equal to $x{}^\smallfrown \langle y\rangle$, where
$$(x,y):=\fork{p\restriction\delta', p(\delta')}(\fork[\delta',\alpha]{p\restriction\delta'}(r)).$$
In particular, $q\restriction\delta'=x=\pi(\fork{p\restriction\delta', p(\delta')}(\fork[\delta',\alpha]{p\restriction\delta'}(r)))$,
which, by Definition~\ref{forking}\eqref{frk5}, is equal to $\fork[\delta',\alpha]{p\restriction\delta'}(r)$.

(1) It follows that, for all $\beta\in[\alpha,\delta)$,
$$q\restriction\beta=(q\restriction\delta')\restriction\beta=\fork[\delta',\alpha]{p\restriction\delta'}(r)\restriction\beta=\fork[\beta,\alpha]{p\restriction\beta}(r),$$
where the rightmost equality follows from the induction hypothesis.
In addition, the case $\beta=\delta$ is trivial.

(2) To avoid trivialities, assume $\alpha<\delta$.
By Clause~(1), $q\restriction\delta'=\fork[\delta,\alpha]{p\restriction\delta'}(r)$.
So, by the induction hypothesis, $B_{q\restriction\delta'}=B_{p\restriction\delta'}\cup B_r$,
and we are left with showing that $\delta\in B_{q}$ iff $\delta\in B_p$.
As $q\le_\delta p$, we have $B_{q}\supseteq B_p$, so the forward implication is clear.
Finally, if $\delta\notin B_p$,
then $p(\delta')=\emptyset$, and hence
$$(x,y)=\fork{p\restriction\delta',\emptyset}(\fork[\delta',\alpha]{p\restriction\delta'}(r)).$$
It thus follows from Clause~\ref{C6blk2} of \blkref{II}
together with the fact that $\pitchfork$ satisfies Clause~\eqref{frk6} of Definition~\ref{forking}
that $(x,y)=(\fork[\delta',\alpha]{p\restriction\delta'}(r),\emptyset)$.
Recalling that $q=x {}^\smallfrown \langle y\rangle$, we conclude that $\delta\notin B_{q}$, as desired.

(4) If $\alpha=0$, then,
by the induction hypothesis,
$\fork[\delta',0]{p\restriction\delta'}(r)=p\restriction\delta'$,
so that
\[\begin{aligned}
(x,y)
&=&\fork{p\restriction\delta', p(\delta')}(\fork[\delta',0]{p\restriction\delta'}(r))\\
&=&\fork{p\restriction\delta', p(\delta')}(p\restriction\delta')\\
&=&(p\restriction\delta', p(\delta'))
&=&(x,y),
\end{aligned}\]
where the rightmost equality follows from Lemma~\ref{frkid}.
Altogether, $q=x{}^\smallfrown \langle y\rangle=p$.

(6) To avoid trivialities, assume that $\fork[\delta,\alpha]{p'}(r)\neq\fork[\delta,\alpha]{p}(r)$,
so that $\alpha<\delta$. By Clause~(4), we may also assume that $0<\alpha$.
Fix $p'\le^0_\delta p$ with $r\le^0_\alpha p'\restriction\alpha$.
By the definition of $\le_{\delta'+1}$,
proving $\fork[\delta,\alpha]{p'}(r)\le_\delta\fork[\delta,\alpha]{p}(r)$ amounts to
verifying that $(x',y')\unlhd (x,y)$, where
$$(x',y'):=\fork{p'\restriction\delta', p'(\delta')}(\fork[\delta',\alpha]{p'\restriction\delta'}(r)).$$
Now, by the induction hypothesis, $\fork[\delta',\alpha]{p'\restriction\delta'}(r)\le_{\delta'}\fork[\delta',\alpha]{p\restriction\delta'}(r)$.
So, since $\fork{p\restriction\delta', p(\delta')}$ is order-preserving, it suffices to prove that
$$(x',y')\unlhd \fork{p\restriction\delta', p(\delta')}(\fork[\delta',\alpha]{p'\restriction\delta'}(r)).$$
Denote $a:=(p\restriction\delta', p(\delta'))$ and $a':=(p'\restriction\delta', p'(\delta'))$.
Then, by Clause~(\ref{frk7}) of Definition~\ref{forking}, indeed
$$\fork{a'}(\fork[\delta',\alpha]{p'\restriction\delta'}(r))\unlhd\fork{a}(\fork[\delta',\alpha]{p'\restriction\delta'}(r)).$$

\item[$\br$] Suppose $\delta\in\acc(\mu^++1)$ is an ordinal such that, for all $\delta'<\delta$, $\beta\in[\alpha,\delta']$,
$p\in P_{\delta'}$ and $r\in P_\alpha$ with $r\le_\alpha p\restriction\alpha$,
$$\fork[\beta,\alpha]{p\restriction\beta}(r)=(\fork[\delta',\alpha]{p\restriction\delta'}(r))\restriction\beta.$$

Fix arbitrary $\alpha\le\delta$, $p\in P_\delta$ and $r\in P_{\alpha}$ with $r\le_\alpha p\restriction\alpha$. Denote $q:=\fork[\delta,\alpha]{p}(r)$.
By our definition of $\pitchfork_{\delta,\alpha}$ at the limit stage, we have:
$$q=\bigcup\{\fork[\beta,\alpha]{p\restriction\beta}(r)\mid \alpha\le\beta<\delta\}.$$
By the induction hypothesis, $\langle \fork[\beta,\alpha]{p\restriction\beta}(r)\mid \alpha\le\beta<\delta\rangle$ is a $\s$-increasing sequence,
and $B_{\fork[\beta,\alpha]{p\restriction\beta}(r)}=B_{p\restriction\beta}\cup B_r$ whenever $\alpha\le\beta<\delta$.
It thus follows that $q$ is a legitimate condition,
and Clauses (1), (2), (4) and (6) are satisfied.\qedhere
\end{itemize}
\end{proof}

Our next task is to verify Clauses (ii) and 
(v) of Goal~\ref{goals}:

\begin{lemma}\label{CvIteration}
Suppose that $\delta\le\mu^+$ is such that for all nonzero $\gamma<\delta$, $(\mathbb{P}_\gamma, c_\gamma, \lh_\gamma)$ is $\Sigma$-Prikry.
Then:
\begin{itemize}
\item for all nonzero $\gamma\le\delta$, $({\pitchfork_{\delta,\gamma}},{\pi_{\delta,\gamma}})$ is a forking projection from $(\mathbb{P}_\delta, \lh_\delta)$ to $(\mathbb{P}_\gamma,\lh_\gamma)$, 
where $\pi_{\delta,\gamma}$ is defined as in Goal~\ref{goals}(ii);
\item  if $\delta<\mu^+$, then $({\pitchfork_{\delta,\gamma}},{\pi_{\delta,\gamma}})$ is furthermore a forking projection from $(\mathbb{P}_\delta, \lh_\delta,c_\delta)$ to $(\mathbb{P}_\gamma, \lh_\gamma,c_\gamma)$ 
\item if $\delta=\gamma+1>1$,  then $(\pitchfork_{\delta,\gamma},\pi_{\delta,\gamma})$ has the weak mixing property.
\end{itemize}
\end{lemma}
\begin{proof}
Let us go over the clauses of Definition~\ref{forking}.

Clause~\eqref{frk5} is covered by Lemma~\ref{CviIteration}(\ref{C5LemmaOfForking}), and
Clause~\eqref{frk7} is covered by Lemma~\ref{CviIteration}(\ref{Cviiforking}).
Clause~(\ref{frk3}) is obvious, since
for all nonzero $\gamma<\delta$ and $p\in P_\gamma$, a straight-forward verification makes it clear that $p*\emptyset_\delta$ is the greatest element of
$\{ q \in P_\delta\mid \pi_{\delta,\gamma}(q)=p\}$.
Consequently, Clause~(\ref{frk6}) follows from Lemma~\ref{CviIteration}(\ref{CViforking}).

Thus, we are left with verifying Clauses \eqref{frk1},\eqref{frk0},\eqref{frk4} and \eqref{frk2}.
The next claim takes care of the first three.

\begin{claim}\label{PropertiesForking} For all nonzero $\gamma\le\delta$ and $p\in P_\delta$:
\begin{enumerate}
  \item\label{CiiIteration} $\pi_{\delta,\gamma}$ forms a projection from $\mathbb P_\delta$ to $\mathbb P_\gamma$, and $\lh_\delta=\lh_\gamma\circ\pi_{\delta,\gamma}$;
\item\label{ClaimForkingC1} $\fork[\delta,\gamma]{p}$ is an order-preserving function from $(\conez{\gamma}{(p\restriction\gamma)},{\le_\gamma})$ to $(\conez{\delta}{p},{\le_\delta})$;
\item\label{ClaimForkingC3} for all $n,m<\omega$ and $q\le^{n+m}_\delta  p$, $m(p,q)$ exists and, furthermore,
  $$m(p,q)=\fork[\delta,\gamma]{p}({m(p\restriction\gamma,q\restriction\gamma)}).$$
\end{enumerate}
\end{claim}
\begin{proof}  We commence by proving \eqref{ClaimForkingC1} and \eqref{ClaimForkingC3} by induction on $\delta\le\mu^+$:

\begin{itemize}
\item[$\br$] The case $\delta=1$ is trivial, since, in this case, $\gamma=\delta$.

\item[$\br$] Suppose $\delta=\delta'+1$ is a successor ordinal and that the claim holds for $\delta'$.
Let $\gamma\le\delta$ and $p\in P_\delta$ be arbitrary. To avoid trivialities, assume $\gamma<\delta$.
By the induction hypothesis,
$\fork[\delta',\gamma]{p\restriction\delta'}$ is an order-preserving function from $\conez{\gamma}{(p\restriction\gamma)}$ to $\conez{\delta'}{(p\restriction\delta')}$.

Recall that $\mathbb P_{\delta}=\mathbb P_{\delta'+1}$ was defined
by feeding $(\mathbb{P}_{\delta'},\lh_{\delta'},c_{\delta'})$ into \blkref{II},
thus obtaining a $\Sigma$-Prikry triple $(\mathbb A,\lh_{\mathbb A},c_{\mathbb A})$
along with the pair $({\pitchfork},{\pi})$.
Now, as $\fork{p\restriction\delta',p(\delta')}$
and $\fork[\delta',\gamma]{p\restriction\delta'}$ are both order-preserving,
the very definition of $\fork[\delta,\gamma]{p\restriction\gamma}$ and $\le_{\delta'+1}$ implies that $\fork[\delta,\gamma]{p\restriction\gamma}$ is order-preserving.
In addition, as $(x,y)$ is a condition in $\mathbb A$ iff $x{}^\smallfrown \langle y\rangle\in P_\delta$
and as
$\fork{p\restriction\delta',p(\delta')}$ is an order-preserving function from $\conez{\delta'}{(p\restriction\delta')}$ to $\conea{(p\restriction\delta',p(\delta'))}$,
we infer that, for all $r\le_\gamma p\restriction\gamma$,
$\fork[\delta,\gamma]{p\restriction\gamma}(r)$ is in $\conez{\delta}{p}$.

Let $q\le^{n+m}_\delta  p$ for some $n,m<\omega$.
Let $$(x,y):=m((p\restriction\delta', p(\delta')), (q\restriction\delta', q(\delta'))).$$
Trivially, $m(p,q)$ exists and is equal to $x{}^\smallfrown \langle y\rangle$.
We need to show that $m(p,q)=\fork[\delta,\gamma]{p}({m(p\restriction\gamma,q\restriction\gamma)})$.
By Definition~\ref{forking}\eqref{frk4},
$$(x,y)=\fork{p\restriction\delta', p(\delta')}(m(p\restriction\delta', q\restriction\delta')).$$
By the induction hypothesis, $$m(p\restriction\delta',q\restriction\delta')=\fork[\delta',\gamma]{p\restriction\delta'}({m(p\restriction\gamma, q\restriction\gamma)}),$$
and so it follows that
$$(x,y)=\fork{p\restriction\delta', p(\delta')}(\fork[\delta',\gamma]{p\restriction\delta'}({m(p\restriction\gamma,q\restriction\gamma)})).$$
Thus, by the definition of $\pitchfork_{\delta,\gamma}$ and the above equation,  we have that $\fork[\delta,\gamma]{p}({m(p\restriction\gamma,q\restriction\gamma)})$ is indeed equal to $x{}^\smallfrown \langle y\rangle$.

\item[$\br$] Suppose $\delta\in\acc(\mu^++1)$ is an ordinal for which the claim holds below $\delta$.
Let $\gamma\le\delta$ and $p\in P_\delta$ be arbitrary. To avoid trivialities, assume $\gamma<\delta$.
By Lemma~\ref{CviIteration}(\ref{vi-like}), for every $r\in \conez{\gamma}{(p\restriction\gamma)}$:
$$\fork[\delta,\gamma]{p}(r)=\bigcup_{\gamma\le\delta'<\delta}\fork[\delta',\gamma]{p\restriction\delta'}(r).$$
As for all $q,q'\in P_\delta$, $q\le_\delta q'$ iff $\forall \delta'<\delta(q\restriction\delta'\le_{\delta'} q'\restriction\delta')$,
the induction hypothesis implies that
$\fork[\delta,\gamma]{p}$ is an order-preserving function from $\conez{\gamma}{(p\restriction\gamma)}$ to $\conez{\delta}{p}$;

Finally, let $q\le_\delta  p$; we shall show that $m(p,q)$ exists and is, in fact, equal to the condition $\fork[\delta,\gamma]{p}({m(p\restriction\gamma,q\restriction\gamma)})$.
By Lemma~\ref{CviIteration}(\ref{vi-like}) and the induction hypothesis,
$$\fork[\delta,\gamma]{p}(m(p\restriction\gamma, q\restriction\gamma))=\bigcup_{\gamma\le{\delta'}<\delta} m(p\restriction{\delta'}, q\restriction{\delta'}),$$
call it $r$. We shall show that $r$ plays the role of $m(p,q)$.

By the definition of $\le_\delta$, it is clear that $q\le^m_\delta r\le^n_\delta p$, so it remains to show that it is the greatest condition in $(P_\delta^p)_{n}$ to satisfy this.
Fix an arbitrary $s\in (P_\delta^p)_{n}$ with $q\le^m_\delta s$.
For each ${\delta'}<\delta$, $q\restriction{\delta'}\le^m_{\delta'} s\restriction{\delta'}\le^n_{\delta'} p\restriction{\delta'}$, so that $s\restriction{\delta'}\le_{\delta'} m(p\restriction{\delta'}, q\restriction{\delta'})$, and thus $s\leq_\delta r$. Altogether this shows that  $r=m(p,q)$.
\end{itemize}
After proving Clauses \eqref{ClaimForkingC1} and \eqref{ClaimForkingC3} above,
we are now left with proving \eqref{CiiIteration}.  The case $\gamma=\delta$ is trivial, so assume $\gamma<\delta$.
Clearly, $\pi_{\delta,\gamma}$ is order-preserving and also $\pi_{\delta,\gamma}(\emptyset_\delta)=\emptyset_\gamma$.
Let $p\in {P}_\delta$ and $q\in {P}_\gamma$ be such that $q\le_\gamma \pi_{\delta,\gamma}(p)$.
Set $q^*:=\fork[\delta,\gamma]{p}(q)$. By Lemma~\ref{CviIteration}\eqref{C5LemmaOfForking},  $\pi_{\delta,\gamma}(q^*)=q$ and by Clause~\eqref{ClaimForkingC1} of this claim, $q^*\le_\delta p$.
Altogether, $\pi_{\delta,\gamma}$ is indeed a projection.
For the second part, recall that, for all $\beta\leq  \mu^+$, $\lh_\beta:=\lh_1\circ \pi_{\beta,1}$, hence $\lh_\delta=\lh_1\circ \pi_{\delta,1}=\lh_1\circ (\pi_{\gamma,1}\circ\pi_{\delta,\gamma})=(\lh_1\circ \pi_{\gamma,1})\circ\pi_{\delta,\gamma}=\lh_\gamma\circ \pi_{\delta,\gamma}.$
\end{proof}

We are left with verifying Clause~\eqref{frk2} of Definition~\ref{forking} to show that $(\pitchfork_{\delta,\gamma}, \pi_{\delta,\gamma})$ is a forking projection from $(\mathbb{P}_\delta,\ell_\delta,c_\delta)$ to $(\mathbb{P}_\gamma,\ell_\gamma,c_\gamma)$.

\begin{claim}\label{C2ClaimForking} Suppose $\delta\neq\mu^+$.
For all $p,p'\in P_\delta$ with  $c_\delta(p)=c_\delta(p')$ and all nonzero $\gamma\le\delta$:
\begin{itemize}
\item $c_\gamma(p\restriction\gamma)=c_\gamma(p'\restriction\gamma)$, and
\item $\fork[\delta,\gamma]{p}(r)=\fork[\delta,\gamma]{p'}(r)$ for every $r\in (P_\gamma)^{p\restriction\gamma}_0\cap (P_\gamma)^{p'\restriction\gamma}_0$.
\end{itemize}
\end{claim}
\begin{proof}
By induction on $\delta<\mu^+$:
\begin{itemize}
\item[$\br$] The case $\delta=1$ is trivial, since, in this case, $\gamma=\delta$.
\item[$\br$] Suppose $\delta=\delta'+1$ is a successor ordinal
and that the claim holds for $\delta'$.
Fix an arbitrary pair $p,p'\in P_\delta$ with $c_\delta(p)=c_\delta(p')$.

Recall that $\mathbb P_{\delta}=\mathbb P_{\delta'+1}$ was defined
by feeding $(\mathbb{P}_{\delta'},\lh_{\delta'},c_{\delta'})$ into \blkref{II},
thus obtaining a $\Sigma$-Prikry triple $(\mathbb A,\lh_{\mathbb A},c_{\mathbb A})$
along the pair $({\pitchfork},{\pi})$.
By the definition of $c_{\delta'+1}$, we have
$$c_{\mathbb A}(p\restriction\delta',p(\delta'))=c_{\delta}(p)=c_{\delta}(p')=c_{\mathbb A}(p'\restriction\delta',p'(\delta')).$$
So, as $({\pitchfork},{\pi})$ is a forking projection from $(\mathbb A,\lh_{\mathbb A},c_{\mathbb A})$ to
$(\mathbb{P}_{\delta'},\lh_{\delta'},c_{\delta'})$, we have $c_{\delta'}(p\restriction\delta')=c_{\delta'}(p'\restriction\delta')$,
and, for all $r\in (P_{\delta'})_0^{p\restriction\delta'}\cap(P_{\delta'})_0^{p'\restriction\delta'}$,
$\fork{p\restriction\delta',p(\delta')}(r)=\fork{p'\restriction\delta',p'(\delta')}(r)$.

Now, as $c_{\delta'}(p\restriction\delta')=c_{\delta'}(p'\restriction\delta')$,
the induction hypothesis implies that $c_{\gamma}(p\restriction\gamma)=c_{\gamma}(p'\restriction\gamma)$ for all nonzero $\gamma\le\delta'$.
In addition, the case $\gamma=\delta$ is trivial.

Finally, fix a nonzero $\gamma\le\delta$ and $r\in (P_\gamma)^{p\restriction\gamma}_0\cap (P_\gamma)^{p'\restriction\gamma}_0$,
and let us prove that $\fork[\delta,\gamma]{p}(r)=\fork[\delta,\gamma]{p'}(r)$. To avoid trivialities, assume $\gamma<\delta$.
It follows from the definition of $\pitchfork_{\delta,\gamma}$ that $\fork[\delta,\gamma]{p}(r)=x{}^\smallfrown\langle y\rangle$ and $\fork[\delta,\gamma]{p'}(r)=x'{}^\smallfrown\langle y'\rangle$,
where:
\begin{itemize}
\item $(x,y):=\fork{p\restriction\delta', p(\delta')}(\fork[\delta',\gamma]{p\restriction\delta'}(r))$, and
\item $(x',y'):=\fork{p'\restriction\delta', p'(\delta')}(\fork[\delta',\gamma]{p'\restriction\delta'}(r))$.
\end{itemize}
But we have already pointed out that the induction hypothesis implies that $\fork[\delta',\gamma]{p\restriction\delta'}(r)=\fork[\delta',\gamma]{p'\restriction\delta'}(r)$,
call it, $r'$.   So, we just need to prove that $\fork{p\restriction\delta', p(\delta')}(r')=\fork{p'\restriction\delta', p'(\delta')}(r')$.
But we also have $c_{\mathbb A}(p\restriction\delta,p(\delta'))=c_\delta(p)=c_{\delta}(p')=c_{\mathbb A}(p'\restriction\delta,p'(\delta'))$ 
and by our choice of $r$ and  Clause~ \eqref{ClaimForkingC1} of Claim~\ref{PropertiesForking}, $r'\in (P_{\delta'})_0^{p\restriction \delta'}\cap (P_{\delta'})_0^{p'\restriction \delta'}$. 
So, as $({\pitchfork},{\pi})$ is a forking projection from $(\mathbb A,\lh_{\mathbb A},c_{\mathbb A})$ to $(\mathbb{P}_{\delta'},\lh_{\delta'},c_{\delta'})$,
Clause~\eqref{frk2} of Definition~\ref{forking} implies that
$$\fork{p\restriction\delta', p(\delta')}(r')=\fork{p'\restriction\delta', p'(\delta')}(r'),$$ as desired.
\item[$\br$] Suppose $\delta\in\acc(\mu^+)$ is an ordinal for which the claim holds below $\delta$.
For any condition $q\in\bigcup_{\delta'\le\delta}P_{\delta'}$, define a function $f_q:B_q\rightarrow H_\mu$ via $f_q(\delta'):=c_{\delta'}(q\restriction\delta')$.
Now, fix an arbitrary pair $p,p'\in P_\delta$ with $c_\delta(p)=c_\delta(p')$.
By the definition of $c_\delta$ this means that
$$\hspace{35pt}\{ (\phi_\delta(\gamma),c_{\gamma}(p\restriction\gamma))\mid \gamma\in B_p\}=\{ (\phi_\delta(\gamma),c_{\gamma}(p'\restriction\gamma))\mid \gamma\in B_{p'}\}.$$
As $\phi_\delta$ is injective,  $f_p=f_{p'}$.
Next, let $\gamma\le\delta$ be nonzero; we need to show that $c_\gamma(p\restriction\gamma)=c_\gamma(p'\restriction\gamma)$.
The case $\gamma=\delta$ is trivial, so assume $\gamma<\delta$.

Now, if $\dom(f_p)\setminus\gamma$ is nonempty, then for $\delta':=\min(\dom(f_p)\setminus\gamma)$,
we have $c_{\delta'}(p\restriction\delta')=f_p(\delta')=f_{p'}(\delta')=c_{\delta'}(p'\restriction\delta')$,
and then the induction hypothesis entails $c_\gamma(p\restriction\gamma)=c_\gamma(p'\restriction\gamma)$.
In particular, if $\dom(f_p)$ is unbounded in $\delta$,
then $c_\gamma(p\restriction\gamma)=c_\gamma(p'\restriction\gamma)$ for all $\gamma\le\delta$.

Next, suppose that $\dom(f_p)$ is bounded in $\delta$
and let $\epsilon<\delta$ be the least ordinal to satisfy $\dom(f_p)\s\epsilon$.
We already know that $c_\gamma(p\restriction\gamma)=c_\gamma(p'\restriction\gamma)$ for all $\gamma<\epsilon$,
and now prove by induction that $c_\gamma(p\restriction\gamma)=c_\gamma(p'\restriction\gamma)$ for all $\gamma\in[\epsilon,\delta)$.
For a successor ordinal $\gamma$, this follows from Clauses \ref{C6blk2} and \ref{C7blk2} of \blkref{II},
and for a limit ordinal $\gamma$, this follows from the fact that the injectivity of $\phi_\gamma$ and the equality $f_{p\restriction\gamma}=f_p=f_{p'}=f_{p'\restriction\gamma}$ implies that $c_\gamma(p\restriction\gamma)=c_{\gamma}(p'\restriction\gamma)$.

\smallskip

Finally, fix a nonzero $\gamma\le\delta$ and $r\in (P_\gamma)^{p\restriction\gamma}_0\cap (P_\gamma)^{p'\restriction\gamma}_0$,
and let us prove that $\fork[\delta,\gamma]{p}(r)=\fork[\delta,\gamma]{p'}(r)$. To avoid trivialities, assume $\gamma<\delta$.
We already know that, for all $\delta'\in[\gamma,\delta)$, $c_{\delta'}(p\restriction\delta')=c_{\delta'}(p'\restriction\delta')$,
and so the induction hypothesis implies that $\fork[\delta',\gamma]{p\restriction\delta'}(r)=\fork[\delta',\gamma]{p'\restriction\delta'}(r)$,
and then by Lemma~\ref{CviIteration}(\ref{vi-like}):
\[\begin{aligned}
\fork[\delta,\gamma]{p}(r)&=\bigcup_{\gamma\le\delta'<\delta}\fork[\delta',\gamma]{p\restriction\delta'}(r)=\\
&=\bigcup_{\gamma\le\delta'<\delta}\fork[\delta',\gamma]{p'\restriction\delta'}(r)=\fork[\delta,\gamma]{p'}(r),
\end{aligned}\]
as desired.\qedhere
\end{itemize}
\end{proof}

\begin{claim}\label{IndeedMixingAtSuccessors}
If $\delta=\beta+1>1$, then $(\pitchfork_{\delta,\beta},\pi_{\delta,\beta})$ has the weak mixing property.
\end{claim}
\begin{proof}
Once again, recall that $\mathbb P_{\beta+1}$ was obtained by feeding $(\mathbb{P}_{\beta},\lh_{\beta},c_{\beta})$ into \blkref{II},
thus obtaining a $\Sigma$-Prikry triple $(\mathbb A,\lh_{\mathbb A},c_{\mathbb A})$,
along with a pair $({\pitchfork},{\pi})$  having the weak mixing property. Let  $\tp$ be a type over $({\pitchfork},{\pi})$ 
witnessing  the weak mixing property. 
For each condition $p\in P_{\beta+1}$, set $ \tp_{\beta+1}(p):=\tp(p\restriction\beta,p(\beta))$. 
The canonical isomorphism from $\mathbb{A}$ to $\mathbb{P}_{\beta+1}$ (i.e., $(x,y)\mapsto x{}^\smallfrown \langle y\rangle$)
makes it clear that $\tp_{\beta+1}$ is a type over $(\pitchfork_{\beta+1,\beta},\pi_{\beta+1,\beta})$ witnessing the weak mixing property.
\end{proof}
This completes the proof of Lemma~\ref{CvIteration}.
\end{proof}

\begin{definition}\label{particularorderings}  For each nonzero $\beta<\mu^+$,
we let $\tp_{\beta+1}$ be the witness to the weak mixing property of $(\pitchfork_{\beta+1,\beta},\pi_{\beta+1,\beta})$,
as defined in the proof of Subclaim~\ref{IndeedMixingAtSuccessors}.
\end{definition}

Recalling Definition~\ref{SigmaPrikry}\eqref{c2},
for all nonzero $\delta\le\mu^+$ and $n<\omega$, we need to identify a candidate for a canonical dense subposet $\z{\mathbb P}_{\delta n}=(\z{P}_{\delta n},\le_\delta)$ of $\mathbb P_{\delta n}$.
We do this next.

\begin{definition}\label{DefRingForLimits}
Let $n<\omega$. Set $\z{P}_{1n}:={}^{1}{{(\z{Q}_n})}$.\footnote{Here, $\z{Q}_n$ is obtained from Clause~\eqref{c2} of Definition~\ref{SigmaPrikry} with respect to the triple $(\mathbb{Q},\ell,c)$ given by \blkref{I}.}
Then, for each   $\delta\in[2,\mu^+]$, define $\z P_{\delta n}$ by recursion:
$$\z{P}_{\delta n}:=\begin{cases}
\{p\in P_\delta\mid \pi_{\delta,\beta}(p)\in \z{P}_{\beta n}\ \&\ \mtp_{\beta+1}(p)=0\},&\text{if }\delta=\beta+1;\\
\{p\in P_\delta\mid \pi_{\delta,1}(p)\in\z{P}_{1n}\ \&\ \forall\gamma\in B_p\,\mtp_\gamma(\pi_{\delta,\gamma}(p))=0\},&\text{otherwise}.\\
\end{cases}
$$
\end{definition}

\begin{lemma}\label{ringcoheres}\label{LiftingAndRings} Let $n<\omega$ and $1\leq \beta<\delta\leq \mu^+$. Then:
\begin{enumerate}
\item $\pi_{\delta,\beta}``\z{P}_{\delta n}\s \z{P}_{\beta n}$;
\item For every $p\in \z{P}_{\beta n}$, $p* \emptyset_{\delta}\in \z{P}_{\delta n}$.
\end{enumerate}
\end{lemma}
\begin{proof} By a straight-forward induction, relying on Clause~\eqref{type6} of Definition~\ref{type}.
\end{proof}

We are now left with addressing  Clause~(iv) of Goal~\ref{goals}. 
Prior to that we will provide a sufficient condition securing that for each $\delta\in\acc(\mu^++1)$, the pair $(\mathbb{P}_\delta,\ell_\delta)$ has property $\mathcal{D}$.  
For this, we establish a version of the Weak Mixing Property (Definition~\ref{mixingproperty}) for limit stages.

\begin{lemma}\label{MixingforLimits}
Let $\delta\in \acc(\mu^++1)$.
For all $a\in P_\delta$, $n<\omega$, $\vec{r}=\langle r_\xi\mid \xi<\chi\rangle$, and $p'\le^0\pi_{\delta,1}(a)$, and for every function $g:W_n(\pi_{\delta,1}(a))\rightarrow \mathbb{P}_\delta\downarrow a$,  if all of the following hold:
\begin{enumerate}\setcounter{enumi}{-1}
\item $\langle B_{g(r_\xi)}\mid \xi<\chi\rangle$ is  $\s$-increasing. 
For each $\gamma$ in $B:=\bigcup_{\xi<\chi} B_{g(r_\xi)}$, we
denote $\iota_\gamma:=\min\{\xi<\chi\mid \gamma\in B_{g(r_\xi)}\}$;

\item $\vec{r}$ is a good enumeration of $W_n(\pi_{\delta,1}(a))$;
\item  $\langle \pi_{\delta,1}(g(r_\xi))\mid \xi<\chi\rangle$ is diagonalizable with respect to $\vec r$ as witnessed by $p'$; 
\item   for all  $\gamma\in B$ and $\xi<\chi$, 
$$\begin{array}{lll}
\ \ \ \ \ &\dom(\tp_\gamma(\pi_{\delta,\gamma}(g(r_\xi))))=0,&\xi<\iota_{\gamma};\\
&\dom(\tp_\gamma(\pi_{\delta,\gamma}(g(r_\xi))))\ge 1,&\xi=\iota_{\gamma};\\
&\dom(\tp_\gamma(\pi_{\delta,\gamma}(g(r_\xi))))>(\sup_{\eta<\xi}\dom(\tp_\gamma(\pi_{\delta,\gamma}(g(r_\eta)))))+1,&\xi>\iota_{\gamma}.
\end{array}$$
\item for all $\gamma\in B$, $\xi\in (\iota_\gamma,\chi)$, 
and  $i\le \sup_{\eta<\xi}\dom(\tp_\gamma(\pi_{\delta,\gamma}(g(r_\eta))))$, 
$$\ \ \ \ \ \  i\ge \dom(\tp_\gamma(\pi_{\delta,\gamma}(a)))\implies\tp_\gamma(\pi_{\delta,\gamma}(g(r_\xi)))(i)\leq\mtp_\gamma(\pi_{\delta,\gamma}(a));$$
\item for all $\gamma\in B$, $\sup_{\iota_\gamma\leq \xi<\chi}\mtp_\gamma(\pi_{\delta,\gamma}(g(r_\xi)))<\omega$,
\end{enumerate}
then there exists  $b\in P_\delta$ such that:
\begin{enumerate}[label=(\alph*)]
\item $\pi_{\delta,1}(b)=p'$;
\item $b\le^0_\delta a$;
\item for all $q'\in W_n(p')$, $\fork[\delta,1]{b}(q')\le_\delta^0 g(w(\pi_{\delta,1}(a),q'))$.
\end{enumerate}
\end{lemma}
\begin{proof} Let $a$, $n$, $\vec{r}$, $p'$ and $g:W_n(a\restriction1)\rightarrow \mathbb{P}_\delta\downarrow a$ be as above.
Let $\langle \gamma_\tau\mid \tau<\theta\rangle$ be the increasing enumeration of the set $B$. 
From Goal~\ref{goals}(i) and $\chi<\mu$, we infer that $\theta<\mu$. 
For each $\tau<\theta$:
\begin{itemize}
\item as  $\gamma_\tau$ is a successor ordinal, we let $\beta_\tau$ denote its predecessor;
\item for every $\xi<\chi$, let $r^\tau_\xi:=\fork[\beta_\tau, 1]{a\restriction \beta_\tau}(r_\xi)$.
By Fact~\ref{forkingfacts}, $\vec{r}^\tau:=\langle r^\tau_\xi\mid \xi<\chi\rangle$ is a good enumeration of $W_n(a\restriction \beta_\tau)$;
\item derive a map $g_\tau\colon W_n(a\restriction \beta_\tau)\rightarrow\mathbb{P}_{\gamma_\tau}\downarrow(a\restriction \gamma_\tau)$ via
$$g_\tau(r^\tau_\xi):= g(r_\xi)\restriction \gamma_\tau.$$
\end{itemize}

\begin{claim} Suppose there is a sequence $\langle (b_\tau,p^\tau)\mid \tau<\theta\rangle\in  \prod_{\tau<\theta}(P_{\gamma_\tau}\times P_{\beta_\tau})$
satisfying that for all $\tau<\theta$:
\begin{enumerate}[label=$\rm{(\Roman*)}$]
\item  $b_0\restriction 1=p^0\restriction 1=p'$;
\item $b_{\tau}\restriction \gamma_{\tau'}=b_{\tau'}$ for all $\tau'<\tau$;
\item  $b_\tau$ witnesses the conclusion of Definition~\ref{mixingproperty} with respect to the tuple $(a\restriction \gamma_\tau, \vec{r}^\tau, p^\tau ,g_\tau, \iota_{\gamma_\tau})$.
In particular, $p^\tau\le_{\beta_\tau}^0 a\restriction \beta_\tau$ diagonalizes $\langle g_\tau(r^\tau_\xi)\restriction \beta_\tau\mid \xi<\chi\rangle$.
\end{enumerate}

Then there is $b\in P_\delta$ as in the conclusion of the Lemma.
\end{claim}
\begin{proof} By (II) above,
we may let $b^*:=\bigcup_{\tau<\theta} b_\tau$,
so that  $b^*\in P_\epsilon$ for $\epsilon:=\dom(b^*)$. 
For each $\tau<\theta$, Clauses (II) and (III) yield
$$b^*\restriction \gamma_\tau={b_\tau}\le^0_{\gamma_\tau}{a\restriction \gamma_\tau},$$
and hence $b^*\le^0_\epsilon(a\restriction\epsilon)$.
So we may let $b:=\fork[\delta,\epsilon]{a}(b^*)$,
and infer from (I) that $b\restriction 1=p'$. Also, we have that $b\restriction \gamma\le^0_{\gamma} a\restriction \gamma$, 
for each $\gamma\in B_a$. This shows that Clauses (a) and (b) 
of the lemma hold. 

We are now left with verifying Clause (c).  Let $q'\in W_n(p')$; we want to show that
$\fork[\delta,1]{b}(q')\le^0_\delta g(w(a\restriction 1,q'))$.  
Note that by Lemma~\ref{CviIteration}\eqref{SupportForking}, $B_a\s B_{b^*}=B_b$,
so that $b=b^*\ast \emptyset_\delta$.
Hence,  $\{\gamma_\tau\mid \tau<\theta\}$ is cofinal in $B_b$, and so
it suffices to prove that, for each $\tau<\theta$, 
$$\fork[\delta,1]{b}(q')\restriction\gamma_\tau\le^0_{\gamma_\tau} g(w(a\restriction 1,q'))\restriction\gamma_\tau.$$

For each $\tau<\theta$,
combining Clause (II) with Lemma~\ref{CviIteration}\eqref{vi-like} we have $$\fork[\delta,1]{b}(q')\restriction\gamma_\tau=\fork[\gamma_\tau,1]{b\restriction \gamma_\tau}(q')=\fork[\gamma_\tau,1]{b_\tau}(q'),$$
hence it suffices to check  that
\begin{equation}\label{eq1}
 \tag{$\star$}\fork[\gamma_\tau,1]{b_\tau}(q')\le_{\gamma_\tau}^0 g(w(a\restriction 1,q'))\restriction\gamma_\tau.
 \end{equation}

By the definition of $\pitchfork_{\gamma_\tau,1}$ in Subsection~\ref{DefiningTheIteration} (see equation~\eqref{PitchfrokSuccessor} on page~\pageref{PitchfrokSuccessor}),  it follows that
\begin{equation}\label{eq2}
   \tag{$\star\star$}\fork[\gamma_\tau,1]{b_\tau}(q')=\fork[\gamma_\tau,\beta_\tau]{b_\tau}(\fork[\beta_\tau,1]{b_\tau\restriction \beta_\tau}(q')).
\end{equation}

Since $b_\tau\restriction 1=p'$ and $q'\in W_n(p')$, Lemma~\ref{CvIteration} yields  $r:=\fork[\beta_\tau,1]{b_\tau\restriction \beta_\tau}(q')$ is in $W_n(b_\tau\restriction \beta_\tau)$.
Combining equation \eqref{eq2} with (III), we infer that
$$\fork[\gamma_\tau,1]{b_\tau}(q')=\fork[\gamma_\tau,\beta_\tau]{b_\tau}(r)\le_{\gamma_\tau}^0 g_{\tau}(w(a\restriction \beta_\tau, r))=g(w(a\restriction 1,q'))\restriction\gamma_\tau, $$
where the rightmost equality follows from the definition of $g_\tau$ and  the fact that $r\restriction 1=q'$.
This verifies equation \eqref{eq1} and yields the claim. 
\end{proof}

Let us now argue by induction that such $\langle (b_\tau, p^\tau) \mid \tau<\theta\rangle$ exists.

\begin{claim}\label{Claimdelta=0mixing}
There is a pair $(b_0,p^0)$ for which Clauses \rm{(I)}--\rm{(III)} hold.
\end{claim}
\begin{proof}  Clause~\rm{(II)} is trivial at this stage.
Setting $p^0:=\fork[\beta_0,1]{a\restriction \beta_0}(p')$ takes care of the second part of Clause~\rm{(I)},
and we shall come back to the first part towards the end.
Now, let us examine the tuple $(a\restriction \gamma_0, \vec{r}^0, p^0 ,g_0,\iota_{\gamma_0})$
against the clauses of Definition~\ref{mixingproperty}
with respect to the forking projection $(\pitchfork_{\gamma_0,\allowbreak \beta_0}, \pi_{\gamma_0,\beta_0})$: Clause~\eqref{Mixing1} is obvious and  Clauses~\eqref{Mixing3}--\eqref{Mixing5}  follow combining the corresponding clauses in the lemma with the definition of  $g_0$.

Regarding Clause \eqref{Mixing2}, we claim that  $p^0$ 
diagonalizes $\langle g_0(r^0_\xi)\restriction \beta_0\mid \xi<\chi\rangle$. 
To this end, we will check  $(\alpha)$ and $(\beta)$ of Proposition~\ref{TheWitnessOfMixingYieldsDiagona}, when this is regarded with respect to the forking projection $(\pitchfork_{\beta_0,1}, \pi_{\beta_0,1})$, and the parameters $a\restriction \beta_0$, $\vec{r}^0$, $\langle g_0(r^0_\xi)\restriction \beta_0\mid \xi<\chi\rangle$, $p'$ and $p^0$, respectively. 

\begin{itemize}
\item[$(\alpha)$] Note that $g_0(r^0_\xi)\upharpoonright 1=g(r_\xi)\upharpoonright 1$ for each $\xi<\chi$. Therefore,  Clause~\eqref{Mixing1}    implies that $p'$ diagonalizes $\langle g_0(r^0_\xi)\restriction 1\mid \xi<\chi\rangle$.

\item[$(\beta)$] By Clause \eqref{Mixing1} of the lemma,   $p'\le^0_1 a\restriction 1$, hence $p^0\le_{\beta_0}^0 a\restriction \beta_0$.
Let $q'\in W_n(p')$. 
Again by Clause~\eqref{Mixing1},  $q'\le_1^0 g(r_\xi)\restriction 1$, where $\xi$ is the unique index such that $r_\xi=w(a\restriction 1, q'\restriction 1)$.
Finally, combining Lemma~\ref{CviIteration}\eqref{CViforking} and Lemma~\ref{CvIteration} we have
$$\fork[\beta_0,1]{p^0}(q')\le^{0}_{\beta_0}\fork[\beta_0,1]{a\restriction \beta_0}(g(r_\xi)\restriction 1)=g(r_\xi)\restriction 1\ast \emptyset_{\beta_0}=g_0(r^0_\xi)\restriction\beta_0,$$
where the above equalities follow from $\beta_0< \min (B)$. 
\end{itemize}

  Altogether, $(a\restriction \gamma_0, \vec{r}^0, p^0 ,g_0,\iota_{\gamma_0})$ witnesses Clauses \eqref{Mixing1}--\eqref{Mixing5} of Definition~\ref{mixingproperty}. Thus,  appealing to Lemma~\ref{CvIteration},
we obtain  $b\in P_{\gamma_0}$ such that $b\restriction \beta_0=p^0$ and $b\le_{\gamma_0}^0 a\restriction \gamma_0$  
that witnesses the conclusion of Definition~\ref{mixingproperty}.  
Clearly,  $b_0:=b$ and $p^0$ are as wanted. 
\end{proof}

Suppose now $\tau<\theta$, and that $\langle (b_{\tau'},p^{\tau'})\mid {\tau'}<\tau\rangle$ has been constructed maintaining (I)--(III). 

\begin{claim}\label{ClaimInductivestep}
There is a pair $(b_\tau,p^\tau)$ satisfying Clauses $\rm{(I)}$--$\rm{(III)}$.
\end{claim}
\begin{proof}
Set $b^*:=\bigcup_{{\tau'}<\tau}b_{\tau'}$ and  $\epsilon:=\dom(b^*)$. Note that  $\epsilon\leq \beta_\tau$, as $\gamma_\tau\in\nacc(\mu^+)$. 
Also, using  (I) and (II) of the induction, $b^*\in P_{\epsilon}$ and $\pi_{\epsilon,1}(b^*)=p'$. 

Set $p^\tau:=\fork[\beta_\tau,\epsilon]{a\restriction \beta_\tau}(b^*)$.
As in the previous claim, to obtain a condition $b_\tau$ satisfying (III), it suffices to show that $p^\tau$
diagonalizes $\langle g_\tau(r^\tau_\xi)\restriction \beta_\tau\mid \xi<\chi\rangle$.
For this, we will want to appeal to Proposition~\ref{TheWitnessOfMixingYieldsDiagona} with $(\pitchfork_{\beta_\tau,1} \pi_{\beta_\tau,1})$, $a\restriction \beta_\tau$, $\vec{r}^\tau$, $\langle g_\tau(r^\tau_\xi)\restriction \beta_\tau\mid \xi<\chi\rangle$, $p'$ and $p^\tau$.
The verification of Clause~($\alpha$) is exactly the same as in Claim~\ref{Claimdelta=0mixing},
so we focus on Clause~($\beta$).
By (II) and (III) of the induction hypothesis, $b^*\le^0_\epsilon a\restriction \epsilon$ and $b^*\restriction 1=p'$. Hence, $p^\tau\in P_{\beta_\tau}$, ${p^\tau}\le^0_{\beta_\tau}{a\restriction \beta_\tau}$ and $p^\tau\restriction 1=p'$.

Let $q'\in W_n(p')$. 
Our aim is to show that $$\fork[\beta_\tau, 1]{p^\tau}(q')\le^0_{\beta_\tau} g_\tau(r^\tau_\xi)\restriction \beta_\tau,$$ for the unique index $\xi$ such that $r_\xi=w(a\restriction 1, q')$.

By virtue of Lemma~\ref{CviIteration}\eqref{CViforking},  $B_{\fork[\beta_\tau,1]{p^\tau}(q')}=B_{p^\tau}=B_{b^*}$. Hence, it will be enough to check that
$\fork[\beta_\tau, 1]{p^\tau}(q')\restriction \epsilon\le^0_{\epsilon} g_\tau(r^\tau_\xi)\restriction \epsilon$.

For each ${\tau'}<\tau$, combining (II) of  the induction hypothesis with Clauses~\eqref{vi-like} and \eqref{C5LemmaOfForking} of
Lemma~\ref{CviIteration}, we have
$$\fork[\beta_\tau, 1]{p^\tau}(q')\restriction \gamma_{\tau'}=\fork[\gamma_{\tau'},1]{b_{\tau'}}(q')=\fork[\gamma_{\tau'},\beta_{\tau'}]{b_{\tau'}}(s_{\tau'}),$$
where $s_{\tau'}:=\fork[\beta_{\tau'},1]{b_{\tau'}\restriction \beta_{\tau'}}(q')$. Indeed, for the latter equality, see Equation~\eqref{eq2} above.

Thus, by (III) of our induction hypothesis,
$$\fork[\beta_\tau, 1]{p^\tau}(q')\restriction \gamma_{\tau'}=\fork[\gamma_{\tau'},\beta_{\tau'}]{b_{\tau'}}(s_{\tau'})\le^0_{\gamma_{\tau'}} g_{\tau'}(r^{\tau'}_\xi),$$
where $\xi$ is the unique index such that $r^{\tau'}_\xi=w(a\restriction \beta_{\tau'},s_{\tau'})$. 

Since $g_{\tau}(r^{\tau}_\xi)\restriction \gamma_{\tau'}=g_{\tau'}(r^{\tau'}_\xi)$, the above expression  actually yields $\fork[\beta_\tau, 1]{p^\tau}(q')\restriction \gamma_{\tau'}\le^0_{\gamma_\tau'}g_{\tau}(r^{\tau}_\xi)\restriction \gamma_{\tau'}$. Altogether, we have shown that $$\fork[\beta_\tau, 1]{p^\tau}(q')\restriction \epsilon \le^0_\epsilon g_\tau(r^\tau_\xi)\restriction \epsilon.$$

Finally, note that $$r_\xi=r^{\tau'}_\xi\restriction 1=w(a\restriction \beta_{\tau'},s_{\tau'})\restriction 1=w(a\restriction 1, q'),$$
where the last equality follows from Lemma~\ref{CvIteration} and  $s_{\tau'}\restriction 1=q'$.

The above shows that $(a\restriction \gamma_\tau,\vec{r}^\tau, p^\tau, g_\tau,i_{\gamma_\tau})$ fulfills the assumptions of Definition~\ref{mixingproperty} with respect the  pair $(\pitchfork_{\gamma_\tau,\beta_\tau},\pi_{\gamma_\tau,\beta_\tau})$.
Appealing to Lemma~\ref{CvIteration}
we obtain ${b_\tau}\le_{\gamma_\tau}^0{a\restriction \gamma_\tau}$ with $b_\tau\restriction \beta_\tau=p^\tau$
such that the pair $(b_\tau, p^\tau)$ witnesses (III).

Let us now show that $(b_\tau, p^\tau)$ satisfies (I) and (II). 
By (II) of the induction hypothesis and the definition of $p^\tau$, for each $\tau'<\tau$,  $$b_\tau\restriction \gamma_{\tau'}=(p^\tau\restriction \epsilon)\restriction \gamma_{\tau'}=b^*\restriction\gamma_{\tau'}=b_{\tau'}.$$
Similarly, by (I) of the induction hypothesis,  $b_\tau\restriction 1=b_{\tau'}\restriction 1=p'$.
\end{proof}

The above completes the induction and  yields the lemma.
\end{proof}

The following technical lemma yields a sufficient condition for the pair $(\mathbb{P}_\delta,\ell_\delta)$ to have property $\mathcal{D}$ in case $\delta\in\acc(\mu^++1)$.

\begin{lemma}\label{lemmapropertyDforlimitMoreOver}
Let $\delta\in\acc(\mu^++1)$,
$a\in P_\delta$, $n<\omega$ and $\vec{s}=\langle s_\xi\mid \xi<\chi\rangle$  be a good enumeration of $W_n(a)$. 
Set $l:=\ell_\delta(a)$.
Suppose that $D$ is a set of conditions in $(\z{\mathbb{P}}_\delta)_{l+n}$
which is dense in $(\mathbb{P}_\delta)_{l+n}$.
Then $\pI$ has a winning strategy for the game $\Game_{\mathbb{P}_\delta}(a,\vec{s},D)$.
\end{lemma}
\begin{proof}  

Set $p:=\pi_{\delta,1}(a)$ and $r_\xi:=\pi_{\delta,1}(s_\xi)$ for each $\xi<\chi$.  By Clauses \eqref{frk4} and \eqref{frk5} of Definition~\ref{forking}, $\vec{r}=\langle r_\xi\mid \xi<\chi\rangle$ is a good enumeration of $W_n(p)$. 

We now describe our strategy for $\pI$.
Suppose that $\xi<\chi$ and that $\langle (a_\eta,b_\eta)\mid \eta<\xi\rangle$ is an initial play of the game $\Game_{\mathbb{P}_\delta}(a,\vec{s},D)$;
we need to define $a_\xi$.

$\br$ If $\xi=0$, 
then let $p_0$ be the $0^{th}$-move of $\pI$ according to some winning strategy in $\Game_{\mathbb P_1}(p, \vec{r})$,
which is available by virtue of \blkref{I}.  Recalling Definition~\ref{DiagonalizabilityGame}, $p_0$ is compatible with $r_0$, so we may let $t_0$ be a condition in $\mathbb{P}_1$ such that $t_0\le_1 p_0, r_0$.

If $B_a$ is empty, then let  $a_0:=\fork[\delta,1]{a}(p_0)$ and $z_0:=\fork[\delta,1]{a}(t_0)$. By Lemma~\ref{CvIteration}, $z_0\le_\delta a_0, s_0$,  hence $a_0$ is a legitimate move for $\pI$.

Suppose now that $B_a$ is nonempty, and let $\langle \gamma_\tau\mid \tau\le\theta\rangle$ be the increasing enumeration of the closure of $B_a$. 
For every $\tau\in\nacc(\theta+1)$, $\gamma_\tau$ is a successor ordinal, so we let $\beta_\tau$ denote its predecessor.
By recursion on $\tau\le\theta$, we shall define a coherent sequence $\langle (a_0^\tau, z_0^\tau)\mid \tau\le\theta\rangle\in\prod_{\tau\le \theta}(P_{\gamma_\tau}\times P_{\gamma_\tau})$,\footnote{That is, for any pair $\tau'\leq \tau$, $a^\tau_0\restriction \tau'=a^{\tau'}_0$ and $z^\tau_0\restriction \tau'=z^{\tau'}_0$.}
and then we shall let $a_0:=\fork[\delta,\gamma_\theta]{a}(a^\theta_0)$ and $z_0:=\fork[\delta,\gamma_\theta]{a}(z^\theta_0)$. 

The idea is to craft the sequence $\langle a^\tau_0\mid \tau\leq \theta\rangle$  so that for all $\gamma\in B_a$,  $a_0\restriction \gamma$ satisfies \eqref{Mixing2}--\eqref{Mixing4} of Lemma~\ref{MixingforLimits}.  On the other hand,  $\langle z^\tau_0\mid \tau\leq \theta\rangle$ will provide a sequence of auxiliary conditions witnessing that $z^\tau_0\le_{\gamma_\tau} a^\tau_0, s_0\restriction\gamma_\tau$. This will ensure at the end that $a_0$ is a legitimate move for $\pI$. 

\smallskip

$\br\br$ Set $\varrho^0_0:=\dom(\tp_{\gamma_0}(a\restriction \gamma_0))+\omega+1$,
and then let 
$$a_0^0:=\fork[\gamma_0,1]{a\restriction\gamma_0}(p_0){}^{\curvearrowright\varrho^0_0},$$
$$z_0^0:=\fork[\gamma_0,1]{a\restriction\gamma_0}(t_0){}^{\curvearrowright\varrho^0_0},$$
where ${}^\curvearrowright$ is the \emph{stretch operation} provided by Definition~\ref{type}\eqref{type4}
with respect to the type $\tp_{\gamma_0}$ over the forking projection $(\pitchfork_{\gamma_0,\beta_0}, \pi_{{\gamma_0,\beta_0}})$.\footnote{\label{Footnoteontypes}Note that $\fork[\gamma_{0},1]{a\restriction\gamma_{0}}(p_0)=\fork[\gamma_{0},\beta_{0}]{a\restriction\gamma_{0}}(\fork[\beta_{0},1]{a\restriction\beta_{0}}(p_0))$ (see \eqref{PitchfrokSuccessor} on page~\pageref{PitchfrokSuccessor}).} 

Since $p_0\le^0_{1} a\restriction 1$,  $a^0_0\in P_{\gamma_0}$ and also  $a^0_0\le^0_{\gamma_0} a\restriction \gamma_0$. Similarly, $z^0_0\in P_{\gamma_0}$. 

\begin{claim}
$z^0_0\le_{\gamma_0} a^0_0, s_0\restriction \gamma_0$.
\end{claim}
\begin{proof}
Combining Clause~\eqref{type4}(a) of Definition~\ref{type} with Lemma~\ref{CvIteration},
$$z^0_0\le_{\gamma_0}^0\fork[\gamma_0,1]{a\restriction \gamma_0}(t_0)\le_{\gamma_0}^0\fork[\gamma_0,1]{a\restriction \gamma_0}(r_0)=s_0\restriction \gamma_0.$$
On the other hand, $\fork[\gamma_0,1]{a\restriction \gamma_0}(t_0)\le_{\gamma_0}\fork[\gamma_0,1]{a\restriction \gamma_0}(p_0)$ and 
$$\dom(\tp_{\gamma_0}(\fork[\gamma_0,1]{a\restriction \gamma_0}(t_0)))=\dom(\tp_{\gamma_0}(\fork[\gamma_0,1]{a\restriction \gamma_0}(p_0))),$$
where the last equality follows from Clause~\eqref{type3} of Definition~\ref{type}.

Combining this with Definition~\ref{type}\eqref{newstretch} we indeed get that $z^0_0\le_{\gamma_0} a^0_0$.
\end{proof}

\begin{claim}\label{IndeedWestretch}
For all $i\in[\dom(\tp_{\gamma_0}(a\restriction\gamma_0)),\dom(\tp_{\gamma_0}(a^0_0))]$, 
$$\tp_{\gamma_0}(a^0_0)(i)\leq\mtp_{\gamma_0}(a\restriction\gamma_0).$$
\end{claim}

\begin{proof}
Let $i$ be as above.
 By Definition~\ref{type}\eqref{type3}, $\dom(\tp_{\gamma_0}(\fork[\gamma_0,1]{a\restriction\gamma_0}(p_0))=\dom(\tp_{\gamma_0}(a\restriction\gamma_0))$. So, combining Clauses~\eqref{type2} and \eqref{type4} of Definition~\ref{type},
$$\tp_{\gamma_0}(a^0_0)(i)\leq\mtp_{\gamma_0}(\fork[\gamma_0,1]{a\restriction\gamma_0}(p_0))\leq \mtp_{\gamma_0}(a\restriction\gamma_0),$$
as desired.
\end{proof}

$\br\br$ For every $\tau<\theta$ such that both $a_0^\tau$ and $z^\tau_0$ have already been defined, set
$\varrho^{\tau+1}_0:=\dom(\tp_{\gamma_{\tau+1}}(a\restriction \gamma_{\tau+1}))+\omega+1$,
and then let 
$$a_0^{\tau+1}:=\fork[\gamma_{\tau+1},\gamma_\tau]{a\restriction\gamma_{\tau+1}}(a_0^\tau){}^{\curvearrowright\varrho^{\tau+1}_0},$$
$$z_0^{\tau+1}:=\fork[\gamma_{\tau+1},\gamma_\tau]{a\restriction\gamma_{\tau+1}}(z_0^\tau){}^{\curvearrowright\varrho^{\tau+1}_0}.$$
where ${}^\curvearrowright$ is the corresponding stretch operation of the type $\tp_{\gamma_{\tau+1}}$.\footnote{ \label{FootnoteType2}Note that  $\fork[\gamma_{\tau+1},\gamma_\tau]{a\restriction\gamma_{\tau+1}}(a_0^\tau)=\fork[\gamma_{\tau+1},\beta_{\tau+1}]{a\restriction\gamma_{\tau+1}}(\fork[\beta_{\tau+1},\gamma_{\tau}]{a\restriction\beta_{\tau+1}}(a_0^\tau))$.}

\begin{claim}
For all $\tau'\leq \tau$, $a_0^{\tau+1}\restriction \gamma_{\tau'}=a_0^{\tau'}$ and $z_0^{\tau+1}\restriction \gamma_{\tau'}=z_0^{\tau'}$.
\end{claim}
\begin{proof}
Let $\tau'\leq \tau$. By Clause~\eqref{type4}(a) of Definition~\ref{type} and Lemma~\ref{CviIteration}\eqref{vi-like}, $$a^{\tau+1}_0\restriction \beta_{\tau+1}=\fork[\gamma_{\tau+1},\gamma_\tau]{a\restriction\gamma_{\tau+1}}(a_0^\tau)\restriction \beta_{\tau+1}=\fork[\beta_{\tau+1},\gamma_\tau]{a\restriction\beta_{\tau+1}}(a_0^\tau).$$
Hence, Lemma~\ref{CviIteration}\eqref{CViforking}  yields $a^{\tau+1}_0\restriction \gamma_\tau=a^\tau_0$. Using the induction hypothesis, we get $a^{\tau+1}_0\restriction \gamma_{\tau'}=a^{\tau'}_0$. The argument for $z^{\tau+1}_0$ is the same. 
\end{proof}

\begin{claim}\label{ShowingCompatibility}
$z^{\tau+1}_0\le_{\gamma_{\tau+1}} a^{\tau+1}_0, s_0\restriction \gamma_{\tau+1}$.
\end{claim}
\begin{proof}
Recall that by the induction hypothesis, $z^\tau_0\le_{\gamma_\tau} a^\tau_0, s_0\restriction\gamma_\tau$. Thus, Clause~\eqref{type4} of Definition~\ref{type} and Lemma~\ref{CvIteration} combined yield
$$z^{\tau+1}_0\le_{\gamma_{\tau+1}}^0\fork[\gamma_{\tau+1},\gamma_\tau]{a\restriction \gamma_{\tau+1}}(z^\tau_0)\le_{\gamma_{\tau+1}}^0\fork[\gamma_{\tau+1},\gamma_\tau]{a\restriction \gamma_{\tau+1}}(s_0\restriction \gamma_\tau)=s_0\restriction \gamma_{\tau+1}.$$
Similarly, Lemma~\ref{CvIteration} yields
$$\fork[\gamma_{\tau+1},\gamma_{\tau}]{a\restriction \gamma_{\tau+1}}(z^{\tau}_0)\le_{\gamma_{\tau+1}}\fork[\gamma_{\tau+1},\gamma_\tau]{a\restriction \gamma_{\tau+1}}(a^{\tau}_0).$$
Also,  by Clause~\eqref{type3} of Definition~\ref{type} and the remark made at Footnote~\ref{FootnoteType2}
$$\dom(\tp_{\gamma_{\tau+1}}(\fork[\gamma_{\tau+1},\gamma_\tau]{a\restriction \gamma_{\tau+1}}(z^{\tau}_0)))=\dom(\tp_{\gamma_{\tau+1}}(\fork[\gamma_{\tau+1},\gamma_\tau]{a\restriction \gamma_{\tau+1}}(a^{\tau}_0)).$$
Therefore,  Definition~\ref{type}\eqref{newstretch}  yields $z^{\tau+1}_0\le_{\gamma_{\tau+1}} a^{\tau+1}_0$, as desired.
\end{proof}

Finally, the following can be proved exactly as  in Claim ~\ref{IndeedWestretch}:

\begin{claim}
For all $i\in[\dom(\tp_{\gamma_{\tau+1}}(a\restriction\gamma_{\tau+1})),\dom(\tp_{\gamma_{\tau+1}}(a^{\tau+1}_0))]$, 
$$\tp_{\gamma_{\tau+1}}(a^{\tau+1}_0)(i)\leq\mtp_{\gamma_{\tau+1}}(a\restriction\gamma_{\tau+1}).$$
\end{claim}

$\br\br$ For every $\tau\in\acc(\theta+1)$, let $a_0^\tau:=\bigcup_{\tau'<\tau}a_0^{\tau'}$ and $z_0^\tau:=\bigcup_{\tau'<\tau}z_0^{\tau'}$. 
By the induction hypothesis,  $\langle (a_0^{\tau'},z_0^{\tau'})\mid \tau'\leq \tau\rangle$ is coherent. Clearly,  $z^\tau_0\le_{\gamma_\tau} a_0^\tau$ and 
 arguing as in Claim~\ref{ShowingCompatibility}, we have $z^\tau_0\le_{\gamma_\tau}^0 s_0\restriction \gamma_\tau$.
\medskip

At the end of the recursion, we define $a_0$ and $z_0$ as mentioned before. Note that by our construction $z_0$ witnesses that $a_0$ is a legitimate move for $\pI$, so, in response, $\pII$ plays a condition $b_0$ in $D$
extending $a_0$ and satisfying $b_0\le_\delta^0 s_0$. 
By Clause~\eqref{type6} of Definition~\ref{type}, noting that for all $\gamma\in B_{a}$, $\dom(\tp_\gamma(a_0\restriction\gamma))\neq 0$,
we infer that $B_a\s B_{a_0}$. Altogether, $B_a\s B_{a_0}\s B_{b_0}$. Now, for every $\gamma\in B_a$, 
$$\dom(\tp_\gamma(a\restriction\gamma))+1<\dom(\tp_\gamma(a_0\restriction\gamma)).$$
Also, for all $i\in[\dom(\tp_\gamma(a\restriction\gamma)), \dom(\tp_\gamma(a_0\restriction\gamma))]$, 
$$\tp_\gamma(a_0\restriction\gamma)(i)\leq\mtp_\gamma(a\restriction\gamma).$$

$\br$ Suppose that $0<\xi<\chi$. Recall that $\langle (a_\eta,b_\eta)\mid \eta<\xi\rangle$ is an initial play of the game and that we want to define $a_\xi$. To that effect, let $p_\xi$ the $\xi^{th}$-move of $\pI$ in the game $\Game_{\mathbb P_1}(p, \vec{r})$, provided   $\langle (a_\eta\restriction 1, b_\eta\restriction 1)\mid \eta<\xi\rangle$ gathers the previous ones.  Let $t_\xi$ be such that $t_\xi\le_1 p_\xi, s_\xi$ and set $B_\xi:=\bigcup_{\eta<\xi} B_{b_{\eta}}$. 

\smallskip

If $B_\xi$ is empty then again set $a_\xi:=\fork[\delta,1]{a}(p_\xi)$ and $z_\xi:=\fork[\delta,1]{a}(t_\xi)$ and argue as in the case $\xi=0$. 
Otherwise,  $B_\xi$ is nonempty and we let $\langle \gamma_\tau\mid \tau\le\theta\rangle$ be the increasing enumeration of the closure of $B_\xi$. 
By recursion on $\tau\le\theta$, we define a  coherent sequence $\langle (a_\xi^\tau, z_\xi^\tau)\mid \tau\le\theta\rangle\in\prod_{\tau\le \theta}(P_{\gamma_\tau}\times P_{\gamma_\tau})$,
and then we shall let $a_\xi:=\fork[\delta,\gamma_\theta]{a}(a^\theta_\xi)$ and $z_\xi:=\fork[\delta,\gamma_\theta]{a}(z^\theta_\xi)$. The construction and the subsequent verifications are the same  as in the case $\xi=0$, so we skip them. The only difference now is that, for each $\tau\in \nacc(\theta+1)$, we set $\varrho^\tau_\xi:=(\sup_{\eta<\xi}\dom(\tp_{\gamma_\tau}(b_\eta\restriction \gamma_\tau)))+\omega+1.$

Thereby,  we get a condition $a_\xi$ which is a legitimate move for $\pI$ and, in response, $\pII$ plays a condition $b_\xi$ in $D$
extending $a_\xi$ and satisfying $b_\xi\le_\delta^0 s_\xi$. Once again, $a_\xi\restriction 1=p_\xi$,  $B_{\xi}\s B_{a_\xi}\s B_{b_\xi}$ and  for all $\gamma\in B_\xi$,  
\begin{equation*}\label{propertyDequation}
\tag{$\dagger \dagger$}(\sup_{\eta<\xi}\dom(\tp_\gamma(b_\eta\restriction\gamma)))+1<\dom(\tp_\gamma(a_\xi\restriction\gamma)).
\end{equation*} 
Also, for all $i\in[\dom(\tp_\gamma(a\restriction\gamma)), \sup_{\eta<\xi}\dom(\tp_\gamma(b_\eta\restriction\gamma))]$, 
\begin{equation*}\label{propertyDequation2}
\tag{$\dagger \dagger\dagger$}\tp_\gamma(a_\xi\restriction\gamma)(i)\leq\mtp_\gamma(a\restriction\gamma).\footnote{For details about the verification of \eqref{propertyDequation2}, see Claim~\ref{IndeedWestretch}. }
\end{equation*}

At the end we obtain a sequence $\langle(a_\xi,b_\xi)\mid \xi<\chi\rangle$ which is a play in the game $\Game_{\mathbb{P}_\delta}(a,\vec{s},D)$. By construction, for each $\xi<\chi$,  $a_\xi\restriction 1=p_\xi$, so that $\langle b_\xi\restriction 1\mid \xi<\chi\rangle$ is diagonalizable with respect to $\vec{r}$. Let $p'\le^0_1\pi_{\delta,1}(a)$ be a witness for this latter fact. 

Our next task is to show that $\langle b_\xi\mid \xi<\chi\rangle$ is diagonalizable and that the corresponding witness $b$ fulfills the requirements of the lemma. 

\begin{claim}\label{ClaimPrepropertyD}
The tuple $(a, \vec{r}, p', g, B_\chi)$ meets  the requirements of Lemma~\ref{MixingforLimits}, where $g: W_n(\pi_{\delta,1}(a))\rightarrow\mathbb P_\delta\downarrow a$ is  defined via
$g(r_\xi):=b_\xi$. 
\end{claim}
\begin{proof}
Let us go over the clauses of Lemma~\ref{MixingforLimits}: Clause~(0) holds by  the construction of $\langle B_{b_\xi}\mid \xi<\chi\rangle$,  Clause~\eqref{Mixing1} is obvious and Clause~\eqref{Mixing2} follows from the discussion of the previous paragraph. So, let us address the rest.

For each $\gamma\in B_\chi$,  denote $\iota_\gamma:=\min\{\xi<\chi\mid \gamma\in B_{b_\xi}\}$.

Clause~\eqref{Mixing3}:  Let $\gamma\in B_\chi$ and $\xi<\chi$: 
\begin{itemize}
\item If $\xi<\iota_\gamma$ then $\gamma\notin B_{b_\xi}$ and so $b_\xi\restriction \gamma=\myceil{b_\xi\restriction \beta}{\mathbb{P}_\gamma}$, where $\gamma=\beta+1$. Thus, Lemma~\ref{CvIteration} and Definition~\ref{type}\eqref{type6} yield $\dom(\tp_\gamma(b_\xi\restriction \gamma))=0$.
\item If $\xi=\iota_\gamma$, then $\gamma\in B_{b_\xi}$ and so $b_\xi\restriction \gamma\neq \myceil{b_\xi\restriction \beta}{\mathbb{P}_\gamma}$, where $\gamma=\beta+1$. Again, Lemma~\ref{CvIteration} and Definition~\ref{type}\eqref{type6} yield $\dom(\tp_\gamma(b_\xi\restriction \gamma))\geq 1$.

\item If $\xi>\iota_\gamma$, then $\gamma\in B_{b_{\iota_\gamma}}\s B_\xi$.\footnote{Recall that $B_\xi=\bigcup_{\eta<\xi}B_{b_\eta}$.} Combining \eqref{propertyDequation} above with $b_\xi\le_\delta a_\xi$ and Clause~\eqref{type2} of Definition~\ref{type} we get $$(\sup_{\eta<\xi}\dom(\tp_{\gamma}(b_\eta\restriction \gamma)))+1<\dom(\tp_\gamma(b_\xi\restriction \gamma)).$$
\end{itemize}

Clause~\eqref{Mixing4}:  Let $\gamma\in B_\chi$,  $\iota_\gamma<\xi<\chi$ and $i$ be as in Clause~\eqref{Mixing4} of Lemma~\ref{MixingforLimits}. By definition, $\gamma\in B_{b_{\iota_\gamma}}\s B_\xi$, hence \eqref{propertyDequation2} yields 
$\tp_\gamma(a_\xi\restriction\gamma)(i)\leq \mtp_\gamma(a\restriction\gamma)$. Combining this with Definition~\ref{type}\eqref{type3} and $b_\xi\le_\delta a_\xi$ we arrive at 
$$\tp_\gamma(b_\xi\restriction\gamma)(i)\leq \tp_\gamma(a_\xi\restriction\gamma)(i)\leq \mtp_\gamma(a\restriction\gamma).$$

Clause~\eqref{Mixing5}: Let $\gamma\in B$. For all $\xi$ such that $\iota_\gamma\leq \xi<\chi$, then $\gamma\in B_{b_\xi}$. Since for all such $\xi$'s, $b_\xi$ is a condition in $D\s (\z{\mathbb{P}}_\delta)_{l+n}$ then $\mtp_\gamma(b_\xi\restriction\gamma)=0$ (see Definition~\ref{DefRingForLimits}). Thus, clearly, $\sup_{\iota_\gamma\leq \xi<\chi} \mtp_\gamma(b_\xi\restriction\gamma)<\omega$.
\end{proof}
Combining Claim~\ref{ClaimPrepropertyD} with Lemma~\ref{MixingforLimits} we get a condition $b$ witnessing Clauses~(a)--(c) of the latter. Note that thanks to (a) and (c) we can appeal to Proposition~\ref{TheWitnessOfMixingYieldsDiagona} with  $(\pitchfork_{\delta,1}, \pi_{\delta,1})$, $a$, $\vec{s}$, $\langle b_\xi\mid \xi<\chi\rangle$, $p'$ and $b$ and conclude that $b$ diagonalizes $\langle b_\xi\mid \xi<\chi\rangle$ with respect to $\vec{s}$. 
\end{proof}

\begin{cor}\label{lemmapropertyDforlimit}
For every $\delta\in\acc(\mu^++1)$,
if $(\z{\mathbb{P}}_\delta)_n$ forms a dense subposet of $(\mathbb{P}_\delta)_n$ for every $n<\omega$,
then $(\mathbb{P}_\delta,\ell_\delta)$ has property $\mathcal{D}$.
\end{cor}
\begin{proof} By Lemmas \ref{lemmapropertyDforlimitMoreOver} and \ref{propertyDatdense}.
\end{proof}

The next lemma will be useful in the proof of Lemma~\ref{CivIteration}.

\begin{lemma}\label{Morethanclosure}
Let $\delta\in[2,\mu^+]$. Then, 
for every $n<\omega$ and every directed set $D$ of conditions in $(\z{\mathbb{P}}_\delta)_n$ of size $<\kappa_n$, there is $q\in (\z{P}_\delta)_n$ such that $q$ is a lower bound of $D$ with $B_q=\bigcup_{p\in D} B_p$.
\end{lemma}
\begin{proof}
We argue by induction on $\delta$. The base case $\delta=2$ can be proved similarly to the successor case below. So, we assume by induction  that the statement  holds for all $\gamma\in\delta\setminus2$ and prove it for $\delta$.

Fix an arbitrary directed family $D\s (\z{P}_\delta)_n$ of size $<\kappa_n$. 

\smallskip

$\br$ Suppose that $\delta=\gamma+1$. Then $\bar{D}:=\{\pi_{\delta,\gamma}(p)\mid p\in D\}$ is a directed subset of $(\z{P}_\gamma)_n$ of size $<\kappa_n$, so that the inductive assumption yields a lower bound  $p'\in (\z{P}_\gamma)_n$ for $\bar{D}$ such that $B_{p'}:=\bigcup_{p\in D} B_{\pi_{\delta,\gamma}(p)}$.  Set $\hat{D}:=\{\fork[\delta,\gamma]{p}(p')\mid p\in D\}$,
and note that $|\hat{D}|\leq |D|<\kappa_n$.
By Lemma~\ref{CvIteration}, $({\pitchfork_{\delta,\gamma}},{\pi_{\delta,\gamma}})$ is a forking projection from $(\mathbb{P}_\delta,\ell_\delta)$ to $(\mathbb{P}_\gamma,\ell_\gamma)$.
So, Definition~\ref{forking}\eqref{frk7} together with Remark~\ref{RemarkType} imply that $\hat{D}$ is a directed subset of $(\z{P}_\delta)^{\pi_{\delta,\gamma}}_n$. 

Recalling that $(\z{P}_\delta)^{\pi_{\delta,\gamma}}_n$ is isomorphic to the $\kappa_n$-directed-closed  poset $\z{\mathbb{A}}^\pi_n$ given by \blkref{II}, we may pick a lower bound $q\in(\z{P}_\delta)^{\pi_{\delta,\gamma}}_n$ for $\hat{D}$ such that $\pi_{\delta,\gamma}(q)=p'$. It is clear that $q$ is the desired lower bound. 

\smallskip

$\br$ Suppose that $\delta$ is limit. Let $C:=\cl(\bigcup_{p\in D}B_p)\cup\{1,\delta\}$.
We shall define a sequence $\langle p_\gamma\mid \gamma\in C\rangle\in\prod_{\gamma\in C}(\z{P}_\gamma)_n$
such that, for all $\gamma\in C$, $p_\gamma$ is a lower bound for $\{ \pi_{\delta,\gamma}(p)\mid p\in D\}$ with $B_{p_\gamma}=\bigcup_{p\in D}B_{\pi_{\delta,\gamma}(p)}$. 
The sequence will be $\subseteq$-increasing in the sense that $p_{\gamma'}\restriction\gamma=p_\gamma$ for any pair $\gamma<\gamma'$ of elements of $C$.
Note that for each $\gamma\in C$, Lemma~\ref{ringcoheres} yields $\{ \pi_{\delta,\gamma}(p)\mid p\in D\}\s(\z{P}_\gamma)_n$.
We define the sequence $\langle p_\gamma\mid \gamma\in C\rangle$  by recursion on $\gamma\in C$:

\begin{itemize}
\item For $\gamma=1$, $\{ \pi_{\delta,1}(p)\mid p\in D\}$ is a directed subset of $(\z{P}_1)_n$ of size $<\kappa_n$.
By \blkref{I}, $(\mathbb P_1,\lh_1,c_1)$ is $\Sigma$-Prikry,
and hence we may find a lower bound $p_1\in (\z{P}_1)_n$ for the set under consideration.

\item Suppose $\gamma>1$ is a non-accumulation point of $C\cap\delta$. Set $\gamma:=\beta+1$ and $\alpha:=\sup(C\cap\gamma)$. Clearly, $\alpha\leq \beta$, so that Lemma~\ref{CviIteration}\eqref{CViforking} yields $$\fork[\beta,\alpha]{\pi_{\delta,\beta}(p)}(p_\alpha)=p_\alpha\ast \emptyset_{\beta},$$
for each $p\in D$. Set $q:=p_\alpha\ast \emptyset_{\beta}$ and note that the induction hypothesis on $p_\alpha$ yields $B_q=\bigcup_{p\in D}B_{\pi_{\delta,\beta}(p)}$. Set $$\bar{D}:=\{\fork[\gamma,\beta]{\pi_{\delta,\gamma}(p)}(q)\mid p\in D\}.$$

Let $p\in D$. Then  $\pi_{\delta,\gamma}(p)\in (\z{P}_\gamma)_n$ and by  Lemma~\ref{LiftingAndRings}, $q\in (\z{P}_\beta)_n$. Also, by Lemma~\ref{CvIteration},  $\tp_\gamma$ is a type over $(\pitchfork_{\gamma,\beta},\pi_{\gamma,\beta})$, hence  Remark \ref{RemarkType}  yields  $\fork[\gamma,\beta]{\pi_{\delta,\gamma}(p)}(q)\in (\z{P}_\gamma)_n$. Altogether, $\bar{D}\s  (\z{P}_\gamma)_n$. 

Since $\bar{D}$ is a directed subset of $(\z{\mathbb{P}}_\gamma)_n$ of size $<\kappa_n$, arguing as in the successor case above  we  find $p_\gamma\in (\z{P}_\gamma)_n$  a lower bound for $\bar{D}$ with $\pi_{\gamma,\beta}(p_\gamma)=q$.  
Let us point out that $B_{p_\gamma}=B_q\cup\{\gamma\}$, and thus $B_{p_\gamma}=\bigcup_{p\in D} B_{\pi_{\delta,\gamma}(p)}$

Let $r\in \bar{D}$. Since $p_\gamma\leq_\gamma r$,  $B_{p_\gamma}\supseteq B_r$ and by Lemma~\ref{CviIteration}\eqref{SupportForking}, $B_r=B_{\pi_{\delta,\gamma}(p)}\cup B_q$, so since $\gamma\in C$, 
$$B_{p_\gamma}\supseteq B_r=B_{\pi_{\delta,\beta}(p)}\cup \{\gamma\}\cup B_q=B_q\cup\{\gamma\}.$$ 
On the other hand, $B_{p_\gamma}\s B_q\cup\{\gamma\}$ since $\pi_{\gamma,\beta}(p_\gamma)=q$. Altogether, $B_{p_\gamma}=B_{q}\cup\{\gamma\}$.  
In addition, $p_\gamma$ is  a lower bound for $\{\pi_{\delta,\gamma}(p)\mid p\in D\}$. 
Finally, $\pi_{\gamma,\alpha}(p_\gamma)=\pi_{\beta, \alpha}(q)= p_\alpha$, 
and so the sequence  $\langle p_{\bar{\gamma}}\mid \bar{\gamma}\leq \gamma\rangle$ is $\s$-increasing. 

\item Suppose $\gamma\in\acc(C)$. Define $p_\gamma:=\bigcup_{\beta \in (C\cap\gamma)} p_\beta$.
By regularity of $\mu$, we have $|B_{p_\gamma}|<\mu$, so that $p_\gamma\in P_\gamma$. Also, by the induction hypothesis, $B_{p_\gamma}=\bigcup_{p\in D} B_{\pi_{\delta,\gamma}(p)}$. 

 For all $p\in D$ and all $\beta\in C\cap\gamma$, we have $\pi_{\delta,\beta}(p_\gamma)=p_\beta\le_\beta \pi_{\delta,\beta}(p)$, hence $p_\gamma$ is a bound for $\{ \pi_{\delta,\gamma}(p)\mid p\in D\}$ in $(P_\gamma)_n$. 
 
 We claim that $p_\gamma\in (\z{P}_\gamma)_n$:  Let  $\alpha\in B_{p_\gamma}$ and $\beta\in C\cap \gamma$ be such that $\alpha\in B_{p_\beta}$.  By the induction hypothesis $p_\beta\in (\z{P}_\beta)_n$, hence  Lemma~\ref{ringcoheres} yields $\pi_{\gamma,\alpha}(p_\gamma)=\pi_{\beta,\alpha}(p_\beta)\in (\z{P}_\alpha)_n$. 
Also, by similar reasons, $\pi_{\gamma,1}(p_\gamma)=\pi_{\beta,1}(p_\beta)\in(\z{P}_1)_n$. Altogether, $p_\gamma\in (\z{P}_\gamma)_n$ and clearly, $\langle p_{\bar{\gamma}}\mid \bar{\gamma}\leq \gamma\rangle$ is $\s$-increasing.

\item Suppose $\gamma=\delta$, but $\gamma\not\in\acc(C)$. In this case, let $\bar\gamma:=\sup(C\cap\gamma)$, and then set $p_\gamma:=p_{\bar\gamma}*\emptyset_\gamma$.
As the interval $(\bar\gamma,\gamma]$ is disjoint from $\bigcup_{p\in D}B_p$, for every $p\in D$,
$$p_\gamma=p_{\bar\gamma}\ast \emptyset_\gamma\le_\gamma \pi_{\gamma,\bar\gamma}(p)\ast \emptyset_\gamma=p.$$
Also, by the induction hypothesis, $p_{\bar{\gamma}}\in (\z{P}_{\bar{\gamma}})_n$ and $B_{p_\gamma}=B_{p_{\bar{\gamma}}}=\bigcup_{p\in D} B_{p}$.  Finally, Lemma~\ref{ringcoheres}  yields $p_\gamma\in (\z{P}_\gamma)_n$. Note also that with this choice $\langle p_{\bar{\gamma}}\mid \bar{\gamma}\leq \gamma\rangle$ is $\s$-increasing.
\end{itemize}

Clearly, $p_\delta$ is a lower bound for $D$ in $(\z{\mathbb{P}}_\delta)_n$ with the desired property.
\end{proof}

We are now ready to address Clause~(iv) of Goal~\ref{goals}.

\begin{lemma}\label{CivIteration} 
For all nonzero $\delta\le\mu^+$, $(\mathbb P_\delta,\lh_\delta,c_\delta)$ satisfies all the requirements to be a $\Sigma$-Prikry triple, with the possible exceptions of Clause~\eqref{c6} and the density requirement in Clause~\eqref{c2}.
Additionally, $\emptyset_\delta$ is the greatest condition in $\mathbb{P}_\delta$, $\lh_\delta=\lh_1\circ\pi_{\delta,1}$, and $\emptyset_\delta\forces_{\mathbb{P}_\delta} \check{\mu}=\kappa^+$.

Under the additional hypothesis that for each $\delta\in\acc(\mu^++1)$ and every $n<\omega$, $(\z{\mathbb{P}}_\delta)_n$ is a dense subposet of $(\mathbb{P}_\delta)_n$, we have that for all nonzero $\delta\leq \mu^+$, $(\mathbb P_\delta,\lh_\delta,c_\delta)$ is $\Sigma$-Prikry triple having property $\mathcal{D}$.
\end{lemma}
\begin{proof}
We argue by induction on $\delta$. The base case $\delta=1$ follows from the fact that $\mathbb P_1$ is isomorphic to $\mathbb Q$ given by \blkref{I}.
The successor step $\delta=\delta'+1$ follows from the fact that $\mathbb P_{\delta'+1}$ was obtained by invoking \blkref{II}.

Next, suppose that $\delta\in\acc(\mu^++1)$ is such that the conclusion of the lemma holds below $\delta$.
In particular, the hypotheses of Lemma~\ref{CvIteration} are satisfied,
so that, for all nonzero $\beta\le\gamma\le\delta$,
$({\pitchfork_{\gamma,\beta}},{\pi_{\gamma,\beta}})$ is a forking projection from $(\mathbb{P}_\gamma, \lh_\gamma)$ to $(\mathbb{P}_\beta,\lh_\beta)$. 
We now go over the clauses of Definition~\ref{SigmaPrikry}:

\medskip

(1) The first bullet of Definition~\ref{gradedposet} follows from the fact that $\lh_\delta=\lh_1\circ\pi_{\delta,1}$.
Next, let $p\in P_\delta$ be arbitrary.
Denote $\bar p:=\pi_{\delta,1}(p)$. Since $(\mathbb P_1,\lh_1,c_1)$ is $\Sigma$-Prikry, we may pick $p'\le_1\bar p$ with $\lh_1(p')=\lh_1(\bar p)+1$.
As $({\pitchfork_{\delta,1}},{\pi_{\delta,1}})$ is a forking projection from $(\mathbb{P}_\delta, \lh_\delta)$ to $(\mathbb{P}_1,\lh_1)$,
Fact~\ref{forkingfacts}(2) implies that $\fork[\delta,1]{p}(p')$ is an element of $(P_\delta)^p_1$.

(2) Let $n<\omega$. By Lemma~\ref{Morethanclosure}, the poset $(\z{\mathbb{P}}_\delta)_n$ is $\kappa_n$-directed closed. Moreover, under the extra assumption that $(\z{\mathbb{P}}_\delta)_n$ is a dense subposet of $(\mathbb{P}_\delta)_n$ we have that $(\z{\mathbb{P}}_\delta)_n$ witnesses the statement of Clause~\eqref{c2}.

\medskip

The next claim takes care of Clause~(3)

\begin{claim}\label{finalchaincondition} Suppose $p,p'\in P_{\delta}$ with $c_{\delta}(p)=c_{\delta}(p')$. Then, $(P_{\delta})_0^p\cap (P_{\delta})_0^{p'}$ is nonempty.
\end{claim}
\begin{proof} If $\delta<\mu^+$, then since $({\pitchfork_{\delta,1}},{\pi_{\delta,1}})$ is a forking projection from $(\mathbb{P}_\delta, \lh_\delta,c_\delta)$ to $(\mathbb{P}_1,\lh_1,c_1)$,
we get from Clause~\eqref{frk2} of Definition~\ref{forking} that $c_1(p\restriction1)=c_1(p'\restriction1)$,
and then by Clause~\eqref{c1} of Definition~\ref{SigmaPrikry},
we may pick $r\in (P_1)_0^{p\restriction1}\cap (P_1)_0^{p'\restriction1}$.
Consequently, Clause~\eqref{frk2} of Definition~\ref{forking} entails $\fork[\delta,1]{p}(r)=\fork[\delta,1]{p'}(r)$.
Finally, Fact~\ref{forkingfacts}(2) implies that $\fork[\delta,1]{p}(r)$ is in $(P_\delta)_0^{p}$
and that $\fork[\delta,1]{p'}(r)$ is in $(P_\delta)_0^{p'}$. In particular, $(P_\delta)_0^{p}\cap (P_\delta)_0^{p'}$ is nonempty.

From now on, assume $\delta=\mu^+$.
In particular, for all nonzero $\beta<\gamma<\mu^+$,
$(\mathbb{P}_\gamma, \lh_\gamma,c_\gamma)$ is a $\Sigma$-Prikry triple
admitting a forking projection  to $(\mathbb{P}_{\beta}, \lh_{\beta},c_{\beta})$ as witnessed by $({\pitchfork_{\gamma,\beta}},{\pi_{\gamma,\beta}})$.
To avoid trivialities, assume also that $|\{\one_{\mu^+},p,p'\}|=3$.
For each $q\in\{p,p'\}$, let $C_q:=\cl(B_q)$ and define a function $e_q:C_q\rightarrow H_\mu$ via
$$e_q(\gamma):=(\phi_\gamma[C_q\cap\gamma],c_{\gamma}(q\restriction\gamma)).$$

Write $i$ for the common value of $c_{\mu^+}(p)$ and $c_{\mu^+}(p')$.
It follows that, for every $\gamma\in C_p\cap C_{p'}$, $e_p(\gamma)=e^i(\gamma)=e_{p'}(\gamma)$,
so that $\phi_\gamma[C_p\cap\gamma]=\phi_\gamma[C_{p'}\cap\gamma]$ and hence $C_p\cap\gamma=C_{p'}\cap\gamma$.
Consequently, $R:=C_p\cap C_{p'}$ is an initial segment of $C_p$ and an initial segment of $C_{p'}$.

Let $\zeta:=\max(C_p\cup C_{p'})$, so that $p=(p\restriction\zeta)*\emptyset_{\mu^+}$ and $p'=(p'\restriction\zeta)*\emptyset_{\mu^+}$.
Set $\gamma_0:=\max(\{0\}\cup R)$.
By the above analysis, $C_p\cap(\gamma_0,\zeta]$ and $C_{p'}\cap(\gamma_0,\zeta]$ are two disjoint closed sets.

If $\gamma_0=\zeta$, then $e_p(\zeta)=e_{p'}(\zeta)$,
so that $c_{\zeta}(p\restriction\zeta)=c_{\zeta}(p'\restriction\zeta)$,
and hence $(P_{\zeta})_0^{p\restriction \zeta}\cap (P_{\zeta})_0^{p'\restriction \zeta}$ is nonempty.
Pick $r$ in that intersection. Then $r*\emptyset_{\mu^+}$ is an element of $(P_{\mu^+})_0^p\cap (P_{\mu^+})_0^{p'}$.

Next, suppose that $\gamma_0<\zeta$.
Consequently, there exists a finite increasing sequence $\langle \gamma_{j+1}\mid j\le k\rangle$ of ordinals from $C_p\cup C_{p'}$ such that
$\gamma_{k+1}=\zeta$ and, for all $j\le \kappa$:
\begin{itemize}
\item[(i)] if $\gamma_{j+1}\in C_p$, then $(\gamma_j,\gamma_{j+1}]\cap(C_p\cup C_{p'})\s C_p$;
\item[(ii)] if $\gamma_{j+1}\notin C_p$, then $(\gamma_j,\gamma_{j+1}]\cap(C_p\cup C_{p'})\s C_{p'}$.
\end{itemize}

We now define a sequence $\langle r_j\mid j\le k+1\rangle$ in $\prod_{j=0}^{k+1}\left((P_{\gamma_j})_0^{p\restriction \gamma_j}\cap (P_{\gamma_j})_0^{p'\restriction\gamma_j}\right)$, as follows.

\begin{itemize}
\item  For $j=0$, if $\gamma_0\in C_p\cap C_{p'}$, then $e_p(\gamma_0)=e_{p'}(\gamma_0)$,
so that $c_{\gamma_0}(p\restriction\gamma_0)=c_{\gamma_0}(p'\restriction\gamma_0)$,
and we may indeed pick $r_0\in (P_{\gamma_0})_0^{p\restriction \gamma_0}\cap (P_{\gamma_0})_0^{p'\restriction \gamma_0}$.
If $\gamma_0\notin C_p\cap C_{p'}$, then $\gamma_0=0$, and we simply let $r_0:=\emptyset$.

\item Suppose that $j<k+1$, where $r_j$ has already been defined.
Let $q:=\fork[\gamma_{j+1},\gamma_j]{p\restriction\gamma_{j+1}}(r_j)$ and $q':=\fork[\gamma_{j+1},\gamma_j]{p'\restriction\gamma_{j+1}}(r_j)$.
By Lemma~\ref{CviIteration}\eqref{SupportForking}, $B_q=(B_p\cap\gamma_{j+1})\cup B_{r_j}$ and $B_{q'}=(B_{p'}\cap\gamma_{j+1})\cup B_{r_j}$.
In particular, if $\gamma_{j+1}\in C_p$, then $(\gamma_j,\gamma_{j+1}]\cap(B_{q}\cup B_{q'})\s B_q$,
so that $q'=r_j*\emptyset_{\gamma_{j+1}}$ and $q\le_{\gamma_{j+1}} q'$ by Clauses \eqref{CViforking} and \eqref{Cviiforking} of Lemma~\ref{CviIteration}, respectively.
Likewise, if $\gamma_{j+1}\notin C_p$, then $q=r_j*\emptyset_{\gamma_{j+1}}$, so that $q'\le_{\gamma_{j+1}}q$.
Thus, $\{q,q'\}\cap (P_{\gamma_j})_0^{p\restriction \gamma_j}\cap (P_{\gamma_j})_0^{p'\restriction\gamma_j}$ is nonempty,
and we may let $r_{j+1}$ be an element of that set.
\end{itemize}

Evidently, $r_{k+1}*\emptyset_{\mu^+}$ is an element of $(P_{\mu^+})_0^p\cap (P_{\mu^+})_0^{p'}$.
\end{proof}
\begin{enumerate}
  \setcounter{enumi}{3}
\item Let $p\in P_\delta$, $n,m<\omega$ and $q\in (P_\delta^p)_{n+m}$ be arbitrary.
Recalling that $({\pitchfork_{\delta,1}},{\pi_{\delta,1}})$ is a forking projection from $(\mathbb{P}_\delta, \lh_\delta)$ to $(\mathbb{P}_1,\lh_1)$,
we infer from Clause~\eqref{frk4} of Definition~\ref{forking} that
$\fork[\delta,1]{p}(m(p\restriction 1, q\restriction 1))$ is the greatest element of
$\{r\le^n_\delta p\mid q\le^m_\delta r\}$.

\item Recalling that $(\mathbb P_1,\lh_1,c_1)$ is $\Sigma$-Prikry,
and that $({\pitchfork_{\delta,1}},{\pi_{\delta,1}})$ is a forking projection from $(\mathbb{P}_\delta, \lh_\delta)$ to $(\mathbb{P}_1,\lh_1)$,
we infer from Fact~\ref{forkingfacts}(1) that, for every $p\in P_\delta$, $|W(p)|=|W(p\restriction1)|<\mu$.

\item Let $p',p\in P_\delta$ with $p'\le_\delta p$.
Let $q\in W(p')$ be arbitrary. For all $\gamma<\delta$,
the pair $({\pitchfork_{\delta,\gamma}},{\pi_{\delta,\gamma}})$ is a forking projection from $(\mathbb{P}_\delta, \lh_\delta)$ to $(\mathbb{P}_\gamma,\lh_\gamma)$,
so that by the special case $m=0$ of Clause~\eqref{frk4} of Definition~\ref{forking},
$$w(p,q)=\fork[\delta,\gamma]{p}(w(p\restriction \gamma, q\restriction \gamma)).$$
Now, for all $q'\le_\delta q$, the induction hypothesis implies that, for all $\gamma<\delta$,
$w(p\restriction \gamma, q'\restriction \gamma)\le_\gamma w(p\restriction \gamma, q\restriction \gamma)$.
Together with Clause~\eqref{frk5} of Definition~\ref{forking}, it follows that, for all $\gamma<\delta$,
$$w(p,q')\restriction\gamma=w(p\restriction \gamma, q'\restriction \gamma)\le_\gamma w(p\restriction \gamma, q\restriction \gamma)=w(p,q)\restriction\gamma.$$
So, by the definition of $\le_\delta$, $w(p,q')\le_\delta w(p,q)$, as desired.

\item By our assumptions, $(\pitchfork_{\delta,1},\pi_{\delta,1})$ is a forking projection from $(\mathbb{P}_\delta,\ell_\delta)$ to $(\mathbb{P}_1,\ell_1)$ and  $(\mathbb{P}_1,\ell_1,c_1)$ is $\Sigma$-Prikry. Moreover, under the extra assumption that for each $n<\omega$, $(\z{\mathbb{P}}_\delta)_n$ is a dense subposet of $(\mathbb{P}_\delta)_n$, Corollary~\ref{lemmapropertyDforlimit} yields property $\mathcal{D}$ for $(\mathbb{P}_\delta,\ell_\delta)$.
It thus follows from Lemma~\ref{propertyDyieldsCPP} that $(\mathbb{P}_\delta,\ell_\delta)$ has the $\CPP$.
\end{enumerate}

To complete our proof we shall need the following claim.
\begin{claim}\label{CorMuIteration}
For each $\delta$ with $1\leq \delta\leq \mu^+$, $\one_{\mathbb P_{\delta}}\forces_{\mathbb{P}_{\delta}}\check{\mu}=\kappa^+$.
\end{claim}
\begin{proof}
The case $\delta=1$ is given by \blkref{I}. Towards a contradiction, suppose that $1<\delta\leq \mu^+$ and that $\one_{\mathbb P_{\delta}}\not\forces_{\mathbb P_{\delta}}\check{\mu}=\kappa^+$.
As $\one_{\mathbb P_1}\forces_{\mathbb P_1}\check{\mu}=\kappa^+$ and $\mathbb P_{\delta}$ projects to $\mathbb P_1$,
this means that there exists $p\in P_{\delta}$ such that $p\forces_{\mathbb P_{\delta}}|\mu|\le|\kappa|$.
Since $\mathbb{P}_1$ is isomorphic to the poset $\mathbb Q$ of \blkref{I}, and since $\one_\mathbb{Q}\forces_{\mathbb{Q}}``\kappa\text{ is singular}"$,\footnote{\label{sole}This is the sole part of the whole proof to make use of the fact that the poset given by \blkref{I} forces $\kappa$ to be singular.}
$\one_{\mathbb P_1}\forces_{\mathbb P_1}``\kappa\text{ is singular}"$.
As $\mathbb P_{\delta}$ projects to $\mathbb P_1$,
in fact $p\forces_{\mathbb P_{\delta}}\cf(\mu)<\kappa$.
Thus, Lemma~\ref{l14}(2) yields a condition $p'\le_{\delta} p$ with $|W(p')|\ge\mu$, contradicting Clause~(5) above.
\end{proof}
This completes the proof of Lemma~\ref{CivIteration}.
\end{proof}

\section{An application}\label{ReflectionAfterIteration}
In this section, we present the first application of our iteration scheme.
We will be constructing a model of finite simultaneous reflection at a successor of a singular strong limit cardinal $\kappa$ in the presence of $\neg\sch_{\kappa}$.

\begin{definition}
For cardinals $\theta<\mu=\cf(\mu)$ and stationary subsets $S,\Gamma$ of $\mu$,
$\refl({<}\theta,S,\Gamma)$ stands for the following assertion.
For every collection $\mathcal S$ of stationary subsets of $S$,
with $|\mathcal S|<\theta$ and $\sup(\{\cf(\alpha)\mid \alpha\in\bigcup\mathcal S\})<\mu$,
there exists $\gamma\in\Gamma\cap E^\mu_{>\omega}$ such that, for every $S\in\mathcal S$, $S\cap\gamma$ is stationary in $\gamma$.

We write $\refl({<}\theta,S)$ for $\refl({<}\theta,S,\mu)$.
\end{definition}

A proof of the following folklore fact may be found in \cite[\S5]{partI}.
\begin{fact}\label{factPartI} If $\kappa$ is a singular strong limit cardinal admitting a stationary subset $S\s\kappa^+$ for which $\refl({<}\cf(\kappa)^+,S)$ holds,
then  $2^\kappa=\kappa^+$.
\end{fact}
In particular, if $\kappa$ is a singular strong limit cardinal of countable cofinality for which $\sch_\kappa$ fails,
and $\refl({<}\theta,\kappa^+)$ holds, then $\theta\le\omega$. We shall soon show that $\theta:=\omega$ is indeed feasible.

The following general statement about simultaneous reflection will be useful in our verification later on.

\begin{prop}\label{p3} Suppose that $\mu$ is non-Mahlo cardinal, and $\theta\le\cf(\mu)$.
For stationary subsets $T,\Gamma,R$ of $\mu$, $\refl({<}2,T,\Gamma)+\refl({<}\theta,\Gamma,R)$ entails $\refl({<}\theta,T\cup\Gamma,R)$.
\end{prop}
\begin{proof} Given a collection $\mathcal S$ of stationary subsets of $T\cup\Gamma$,
with $|\mathcal S|<\theta$ and $\sup(\{\cf(\alpha)\mid \alpha\in\bigcup\mathcal S\})<\mu$, we shall first attach to any set $S\in\mathcal S$,
a stationary subset $S'$ of $\Gamma$, as follows.

$\br$ If $S\cap \Gamma$ is stationary, then let $S':=S\cap\Gamma$.

$\br$ If $S\cap \Gamma$ is nonstationary,
then for every (sufficiently thin) club $C\s\mu$, $S\cap C$ is a stationary subset of $T$, and so by $\refl({<}2,T,\Gamma)$,
there exists $\gamma\in\Gamma\cap E^\mu_{>\omega}$ such that $(S\cap C)\cap\gamma$ is stationary in $\gamma$,
and in particular, $\gamma\in C$.
So, the set $\{\gamma\in\Gamma\mid S\cap\gamma\text{ is stationary}\}$ is stationary,
and, as $\mu$ is non-Mahlo, we may pick $S'$ which is a stationary subset of it and all of its points consists of the same cofinality.

Next, as $|\mathcal S|<\cf(\mu)$, we have $\sup(\{\cf(\gamma)\mid \gamma\in S', S\in \mathcal S\})<\mu$,
and so, from $\refl({<}\theta,\Gamma,R)$, we find some $\delta\in R$ such that $S'\cap\delta$ is stationary for all $S\in\mathcal S$.
\begin{claim} Let $S\in\mathcal S$. Then $S\cap\delta$ is stationary in $\delta$.
\end{claim}
\begin{proof} If $S'=S$, then $S\cap\delta=S'\cap\delta$ is stationary in $\delta$, and we are done.
Next, assume $S'\neq S$, and let $c$ be an arbitrary club in $\delta$.
As $S'\cap\delta$ is stationary in $\delta$, we may pick $\gamma\in\acc(c)\cap S'$.
As $\gamma\in S'\s E^\mu_{>\omega}$, $c\cap\gamma$ is a club in $\gamma$, and as $\gamma\in S'$,
$S\cap\gamma$ is stationary, so $S\cap c\cap \gamma\neq\emptyset$. In particular, $S\cap c\neq\emptyset$.
\end{proof}
This completes the proof.
\end{proof}

\subsection{About Building Block II} 
In this subsection, we describe Building Block II that we will be feeding to the iteration scheme of the preceding section.
We were originally planning to use the functor given by \cite[\S6]{partI}, 
but unfortunately we found a gap in the proof of the mixing property \cite[Lemma~6.16]{partI}. 
To mitigate this gap, we shall relax Clause~\eqref{C4ptree} of \cite[Definition~6.2]{partI}
and prove that the outcome is a functor satisfying the weak mixing property (Lemma~\ref{mixinglemma} below). 
Most of the results of \cite[\S6]{partI} remain valid, as will be detailed later. Therefore, reading of this subsection does assume that the reader is comfortable with \cite[\S6]{partI}.
The upshot of this subsection is encapsulated by Corollary~\ref{onestep}.

We commence by describing our setup for this subsection.

\begin{setup}
Suppose that we are given a $\Sigma$-Prikry notion of forcing $(\mathbb P,\ell,c)$ having property $\mathcal D$.
Denote $\mathbb{P}=(P,\le)$ and $\Sigma=\langle \kappa_n\mid n<\omega\rangle$.
Also, define $\kappa$ and $\mu$ as in Definition~\ref{SigmaPrikry},
and assume that $\one_{\mathbb P}\forces_{\mathbb P}``\check\kappa\text{ is singular}"$ and that $\mu^{<\mu}=\mu$. Recall that for each $n<\omega$, we denote by $\z{\mathbb{P}}_n$ a dense $\kappa_n$-directed-closed subposet of $\mathbb{P}_n$. Our universe of sets is denoted by $V$,
and we assume that, for all $n<\omega$,
$V^{\mathbb P_n}\models \refl({<}2,E^\mu_{\omega},E^\mu_{<\kappa_n})$.
Write $\Gamma:=\{\gamma<\mu\mid \omega<\cf^V(\gamma)<\kappa\}$. 
We also assume that we are given a condition $r^\star$ forcing 
that $\dot T$ is a $\mathbb P$-name 
for some subset $T$ of $(E^\mu_\omega)^V$
such that, for all $\gamma\in \Gamma$, $T\cap\gamma$ is nonstationary in $\gamma$.
\end{setup}

For each $n<\omega$, denote $\dot{T}_n:=\{(\check\alpha,p)\mid (\alpha,p)\in E^\mu_\omega\times P_n\ \&\ p\Vdash_{\mathbb P}\check\alpha\in\dot T\}$.
\begin{lemma}\label{cor9}
For every $q\le r^\star$,
$q\Vdash_{\mathbb{P}_{\lh(q)}}``\dot{T}_{\lh(q)}\text{ is nonstationary}"$.
\end{lemma}
\begin{proof} The proof is almost the same as that of \cite[Lemma 6.1]{partI},
so we settle here for a sketch.

Suppose not. Set $n:=\lh(r^\star)$ and pick $p\le^0 r^\star$ that $\mathbb P_n$-forces $\dot{T}_{n}$ is stationary.
As $V^{\mathbb P_n}\models \refl({<}2,E^\mu_{\omega},E^\mu_{<\kappa_n})$,
we may fix $p'\le^0 p$ and $\gamma\in E^\mu_{<\kappa_n}$ of uncountable cofinality such that $p'\Vdash_{\mathbb{P}_n}``\dot{T}_n\cap\gamma\text{ is stationary}"$.
As $\mathbb P_n$ has a $\kappa_n$-directed-closed dense subset, $\gamma\in\Gamma$,
and there exists a ground model stationary subset $B$ of $\gamma$ such that
$$r\Vdash_{\mathbb{P}_n}``\dot{T}_n\cap\gamma\text{ contains the stationary set }\check B".$$

By definition of the name $\dot T_n$, $r\Vdash_{\mathbb{P}}\check B\s\dot{T}\cap\gamma$.
Finally, as $\otp(B)<\kappa$, we infer from Lemma~\ref{l14}\eqref{C1l14} that $B$ remains stationary in any forcing extension by $\mathbb P$.
So,  $r\Vdash_{\mathbb{P}}``\dot{T}\cap\gamma\text{ is stationary}"$,
contradicting the fact that $r\le p'\le p\le r^\star$ and $\gamma\in\Gamma$.
\end{proof}
Set $I:=\omega\setminus \ell(r^\star)$. By the preceding lemma,
for each $n<\omega$, we may pick a $\mathbb P_n$-name $\dot{C}_n$ for a club subset of $\mu$ such that,
for all $q\le r^\star$ with $n:=\lh(q)$ in $I$, $q\Vdash_{\mathbb P_n}\dot T_n\cap\dot C_n=\emptyset$.
Let $R$ be the binary relation
$$R:=\{(\alpha,q)\in\mu\times P\mid q\le r^\star\ \&\ \forall r\le q[\lh(r)\in I\rightarrow r\forces_{\mathbb{P}_{\lh(r)}}\check\alpha\in \dot{C}_{\lh(r)}]\}.$$
\begin{remark}\label{additionalfeatureofT+}  The relation $R$  is downwards closed, i.e.,
for all $(\alpha,q)\in R$ and $q'\le q$, $(\alpha,q')\in R$, as well.
\end{remark}

We shall present a $\Sigma$-Prikry notion of forcing for killing the stationarity of the following set: 
$$\dot{T}^+:=\{(\check{\alpha},p)\mid (\alpha,p)\in (E^\mu_\omega)^V\times P \;\&\; p\forces_{\mathbb{P}_{\ell(p)}} \check{\alpha}\notin \dot{C}_{\ell(p)}\}.$$
By the next lemma, killing the stationarity of $\dot{T}^+$ would also kill the stationarity of $\dot T$,
which is our primary goal.
\begin{lemma}\label{lemma46}
\begin{enumerate}
\item  $r^\star\forces_{\mathbb{P}} \dot{T}\s \dot{T}^+$;
\item For every $(\alpha,q)\in R$, $q\forces_\mathbb{P}\check\alpha\notin\dot{T}^+$.
\end{enumerate}
\end{lemma}
\begin{proof} (1) Let $q\le r^\star$ and an ordinal $\alpha$ be such that $q\forces_{\mathbb{P}}\check{\alpha}\in\dot{T}$. 
Note that $\alpha\in (E^\mu_\omega)^V$.
Put $n:=\ell(q)$. 
By the definition, $(\check{\alpha},q)\in \dot{T}_{n}$ and so, in particular,  $q\forces_{\mathbb{P}_{n}}\check{\alpha}\in \dot{T}_{n}$. 
Since $q\leq r^\star$, it follows that $q\forces_{\mathbb{P}_{n}}\check{\alpha}\notin \dot{C}_{n}$, and thus $(\check{\alpha},q)\in \dot{T}^+$. 
In particular, $q\forces_{\mathbb{P}}\check{\alpha}\in\dot{T}^+$.

(2) Suppose that $(\alpha,q)\in R$. Towards a contradiction, suppose that there is $q'\le q$ such that $q'\forces_{\mathbb{P}} \check{\alpha}\in \dot{T}^+$. 
By further extending $q'$, we may assume that $(\check{\alpha},q')\in\dot{T}^+$. Therefore, $q'\forces_{\mathbb{P}_{\ell(q')}}\check{\alpha}\notin \dot{C}_{\ell(q')}$. 
However $q'\leq q$, contradicting the fact that $R$ is downwards closed.
\end{proof}

The next definition is motivated by the upcoming simple lemma.
\begin{definition}\label{tauns} Set $\tau_n:=\{(\check{\alpha},p)\in \dot{T}^+\mid \alpha\in (E^\mu_\omega)^V\ \&\ p\in P_n\}$
for every $n<\omega$.
\end{definition}
\begin{lemma}\label{ThekeylemmaaboutT+}
Let $n<\omega$ and $p\in P_n$. The following hold:
\begin{enumerate}
\item $\tau_n\s \dot{T}^+_n$;
\item $p\forces_{\mathbb{P}_n}\tau_n=(\check{E}^\mu_\omega\setminus \dot{C}_n)$.
\end{enumerate}
\end{lemma}
\begin{proof}
(1) Given $(\check{\alpha},p)\in \tau_n$, we have that $(\check{\alpha},p)\in \dot{T}^+$ and $p\in P_n$. Hence, $p\forces_{\mathbb{P}}\check{\alpha}\in \dot{T}^+$ and $p\in P_n$, which yields $(\check{\alpha},p)\in \dot{T}^+_n$, as desired.

(2) We begin by proving the left-to-right inclusion. Let $q\leq_{\mathbb{P}_n} p$ and $\alpha\in (E^\mu_\omega)^V$ be such that $q\forces_{\mathbb{P}_n} \check{\alpha}\in\tau_n$. By possibly $\mathbb{P}_n$-extending $q$ we may further assume that $(\check{\alpha},q)\in\tau_n$. By the definition of $\tau_n$, $(\check{\alpha},q)\in \dot{T}^+$ and $q\in P_n$. 
By the definition of $\dot{T}^+$,  $q\forces_{\mathbb{P}_n}\check{\alpha}\notin \dot{C}_n$. 
Altogether, $q\forces_{\mathbb{P}_n}\check{\alpha}\in (\check{E}^\mu_\omega\setminus \dot{C}_n)$.
For the other inclusion, suppose that $q\leq_{\mathbb{P}_n} p$ and $\alpha\in (E^\mu_\omega)^V$ are such that $q\forces_{\mathbb{P}_n}\check{\alpha}\notin \dot{C}_n$.  
By the definition of $\dot{T}^+$ this yields $(\check{\alpha},q)\in \dot{T}^+$, and hence $(\check{\alpha},q)\in\tau_n$. Consequently, $q\forces_{\mathbb{P}_n}\check{\alpha}\in \tau_n$.
\end{proof}
The above lemma will be crucial in our verification of density of the poset $(\z{\mathbb{P}}_\delta)_n$ in $(\mathbb{P}_\delta)_n$ at limit stages $\delta$ (see Lemma~\ref{ringisdense}).

\begin{definition}[relaxed form of {\cite[Definition~6.2]{partI}}]\label{labeled-p-tree} Suppose $p\in P$.
A \emph{labeled $p$-tree} is a function $S:W(p)\rightarrow[\mu]^{<\mu}$ such that for all $q\in W(p)$:
\begin{enumerate}
\item\label{C1ptree} $S(q)$ is a closed bounded subset of $\mu$;
\item\label{C2ptree} $S(q')\supseteq S(q)$ whenever $q'\le q$;
\item\label{C3ptree} $q\Vdash_{\mathbb P} S(q)\cap\dot{T}^+=\emptyset$;
\item\label{d162}\label{C4ptree} there is a natural number $m$ such that
for any pair $q'\le q$ of elements of $W(p)$, if $S(q')\neq\emptyset$ and $\lh(q)\ge\lh(p)+m$, then $(\max(S(q')),q)\in R$.
The least such $m$ is denoted by $m(S)$.
\end{enumerate}
\end{definition}
\begin{remark}\label{incompatiblewithrstar}
By Clause~\eqref{C4ptree} and the Definition of $R$, for any pair $q'\le q$ of elements of $W(p)$,
if $\lh(q)\ge\lh(p)+m(S)$ and $q$ is incompatible with $r^\star$, then $S(q')=\emptyset$.
\end{remark}

\begin{definition}[{\cite[Definition~6.3]{partI}}]\label{strategy}
For $p\in P$, we say that $\vec S=\langle S_i\mid i\leq\alpha\rangle$ is a \emph{$p$-strategy} iff all of the following hold:
\begin{enumerate}
\item\label{C1pstrategy} $\alpha<\mu$;
\item\label{i3}
\label{C2pstrategy} $S_i$ is a labeled $p$-tree for all $i\leq\alpha$;
\item\label{C3pstrategy} for every $i<\alpha$ and $q\in W(p)$, $S_{i}(q)\sqsubseteq S_{i+1}(q)$;
\item\label{C4pstrategy} for every $i<\alpha$ and a pair $q'\le q$ in $W(p)$,  $(S_{i+1}(q)\setminus S_i(q))\sqsubseteq (S_{i+1}(q')\setminus S_i(q'))$;
\item\label{C5pstrategy} for every  limit $i\leq\alpha$ and $q\in W(p)$, $S_i(q)$ is the ordinal closure of $\bigcup_{j<i}S_j(q)$.
In particular, $S_0(q)=\emptyset$ for all $q\in W(p)$.
\end{enumerate}
\end{definition}

Now, we are ready to describe our functor.

\begin{definition}[{\cite[Definition~6.4]{partI}}]\label{d20}
Let $\mathbb{A}(\mathbb{P}, \dot{T})$ be the notion of forcing $\mathbb{A}:=(A,\unlhd)$, where:
\begin{enumerate}
\item
\label{C1d20} $(p,\vec S)\in A$ iff $p\in P$, and $\vec S$ is either the empty sequence, or a $p$-strategy;
\item
\label{C2d20} $(p', \vec{S'})\unlhd(p, \vec S)$ iff:
\begin{enumerate}
\item
\label{C2ad20} $p'\le p$;
\item
\label{C2bd20} $\dom(\vec{S'})\geq \dom(\vec S)$;
\item
\label{C2cd20} $S'_i(q)=S_i(w(p,q))$ for all $i\in \dom(\vec S)$ and $q\in W(p')$.
\end{enumerate}
\end{enumerate}

For all $p\in P$, denote $\myceil{p}{\mathbb A}:=(p,\emptyset)$.
\end{definition}

\begin{definition}[{\cite[Definitions 6.10 and 6.11]{partI}}]\label{d45}\label{DefCA}\hfill
\begin{itemize}
\item Define $c_{\mathbb A}:A\rightarrow H_\mu$ by letting, for all $(p,\vec S)\in A$,
$$c_{\mathbb A}(p,\vec S):=(c(p),\{ ( i,c(q),S_i(q))\mid i\in\dom(\vec S), q\in W(p)\}).$$
\item Define $\pi:A\rightarrow P$ by stipulating $\pi(p,\vec S):=p$ and $\ell_\mathbb{A}:=\ell\circ\pi$.
\item Given $a=(p,\vec S)$ in $A$, define $\fork{a}:\cone{p}\rightarrow A$ by letting
for each $p'\le p$, $\fork{a}(p'):=(p',\vec{S'})$, where $\vec{S'}$ is the sequence $\langle S_i':W(p')\rightarrow[\mu]^{<\mu}\mid i<\dom(\vec{S})\rangle$ satisfying:
\begin{equation}\label{pitchfork}
\tag{*}S'_i(q):=S_i(w(p,q))\text{  for all }i\in\dom(\vec{S'})\text{ and }q\in W(p').
\end{equation}
\end{itemize}
\end{definition}

Even after relaxing Clause~\eqref{C4ptree} of \cite[Definition~6.2]{partI} to that of Definition~\ref{labeled-p-tree}, the following remains valid,
with essentially the same proofs.

\begin{fact}[{\cite[Corollary 4.13,  Lemma 6.6, Theorem 6.8]{partI}}]\hfill\label{FactPartI}
\begin{enumerate}
\item $\one\forces_\mathbb{A}\check{\mu}=\check{\kappa}^+$; 
\item For every $\nu\geq \mu$, if $\mathbb{P}$ is a subset of $H_\nu$, then so is $\mathbb{A}$;
\item $\myceil{r^\star}{\mathbb{A}}\forces_\mathbb{A}``\dot{T}^+$ is nonstationary''.\footnote{Here, Claim~\ref{claim244} below plays the role of \cite[Lemma~6.7]{partI}. Also, note that this is trivial when $\dot{T}^+$ is a $\mathbb{P}$-name for a nonstationary subset of $\mu$ in $V$. }
\end{enumerate}
\end{fact}
\begin{remark}\label{RemarkKillingT}
By Lemma~\ref{lemma46} and Fact~\ref{FactPartI}(3), $\myceil{r^\star}{\mathbb{A}}\forces_\mathbb{A}``\dot{T}$ is nonstationary''.
\end{remark}

\begin{lemma}\label{forkingindeed}
 $(\pitchfork,\pi)$ is a forking projection from $(\mathbb{A},\ell_\mathbb{A},c_\mathbb{A})$ to $(\mathbb{P},\ell,c)$.
\end{lemma}
\begin{proof} The proof of \cite[Lemma~6.13]{partI} goes through,
so we only focus on Clause~\eqref{frk0} of Definition~\ref{forking}.
Let $a\in A$ and $p'\le\pi(a)$;
we shall show that $\fork{a}(p')\in A$ and $\fork{a}(p')\unlhd a$.

Write $a$ as $(p,\vec{S})$. If $\vec S=\emptyset$, then $\fork{a}(p')=\myceil{p'}{\mathbb A}$, and we are done.

Next, suppose that $\dom(\vec S)=\alpha+1$. Let $(p',\vec{S'}):=\fork{a}(p')$.
Let $i\leq\alpha$ and we shall verify that $S'_i$ is a $p'$-labeled tree
with $m(S_i')\le m(S_i)$. We go over the clauses of Definition~\ref{labeled-p-tree}.
To this end, let $q'\le q$ be arbitrary pair of elements of $W(p')$.
\begin{itemize}
\item[(2)] By Definition~\ref{SigmaPrikry}(\ref{itsaprojection}),
we have $w(p,q')\le w(p,q)$, so that $S_i'(q')=S_i(w(p,q'))\supseteq S_i(w(p,q))=S_i'(q)$.
\item[(3)] As $q\le w(p,q)$, $w(p,q)\Vdash_{\mathbb P} S_i(w(p,q))\cap\dot{T}=\emptyset$, so that, since $S_i'(q)=S_i(w(p,q))$,
we clearly have $q\Vdash_{\mathbb P} S_i'(q)\cap\dot{T}=\emptyset$.
\item[(4)]  To avoid trivialities,
Suppose that $S'_i(q')\neq\emptyset$ and $\lh(q)\ge m(S_i)$.
Write $\gamma:=\max(S_i'(q'))$.
 As $\lh(w(p,q))=\lh(q) \ge m(S_i)$ and $\gamma=\max(S_i(w(p,q')))$, we infer that $(\gamma,w(p,q))\in R$.
In addition, $q\le w(p,q)$, so by the definition of $R$ it follows that $(\gamma,q)\in R$.
Recalling that $\max(S_i'(q))=\gamma$, we are done.\footnote{Following the terminology of Definition~\ref{labeled-p-tree}\eqref{C4ptree} note that here we have showed that $m(S'_i)\leq m(S_i)$. This will become important soon, whenever we introduce the type map associate to Sharon's functor (see Lemma~\ref{typeindeed}).}
\end{itemize}
To prove that $(p',\vec{S}')$ is a condition in $A$ it now remains to argue that $\vec{S}'$ fulfills the requirements described in Clauses~\eqref{C3pstrategy} and \eqref{C5pstrategy} of Definition~\ref{strategy} but this already follows from the definition of  $\vec{S}'$ and the fact that $\vec{S}$ is a $p$-strategy. Finally $\fork{a}(p')=(p',\vec{S}')\unlhd (p,\vec{S})=a$ by the very choice of $p'$ and by Definition~\ref{d45}.
\end{proof}

We now introduce a type $\tp$ over $(\pitchfork,\pi)$ witnessing the weak mixing property.

\begin{definition}\label{defntypes} Define a map $\tp:A\rightarrow{}^{<\mu}\omega$, as follows.

Given $a=(p,\vec S)$ in  $A$, write $\vec S$ as $\langle S_i\mid i<\beta\rangle$, and then let
$$\tp(a):=\langle m(S_i)\mid i<\beta\rangle.$$
\end{definition}
We shall soon verify that $\tp$ is a type, but will use the $\mtp$ notation of Definition~\ref{type}
from the outset. In particular, we will have $\z{\mathbb A}=(\z A,\unlhd)$, with $\z{A}:=\{ a\in A\mid \pi(a)\in\z{P}_{\ell(\pi(a))}\ \&\ \mtp(a)=0\}$.
Note that the supercollection $\{ a\in A\mid \mtp(a)=0\}$ coincides with the set $A$ from \cite[Definition~6.4]{partI}.
In particular, the proof of \cite[Lemma~6.15]{partI} goes through, yielding the following crucial consequence of each of the $\dot C_n$'s being a $\mathbb P_n$-name for a closed subset of $\mu$:

\begin{fact}\label{C1ASigmaPrirky} For all $n<\omega$, $\z{\mathbb A}_n^\pi$ is $\mu$-directed-closed.\qed
\end{fact}

\begin{lemma}\label{typeindeed} The map $\tp$ is a type over $(\pitchfork,\pi)$.
\end{lemma}
\begin{proof}
We go over the clauses of  Definition~\ref{type}:
	
\begin{enumerate}
\item[\eqref{type1}] This follows from the mere definition of $\tp$.
\item[\eqref{type2}] Write $b=(p',\vec S')$ and $a=(p,\vec S)$.
By Definitions \ref{d20} and \ref{defntypes}, $\dom(\tp(b))=\dom(\vec S')\ge\dom(\vec S)=\dom(\tp(a))$. Fix $i\in\dom(\tp(a))$ and let us show that $\tp(b)(i)\leq \tp(a)(i)$, i.e., that $m(S'_i)\leq m(S_i)$. 

Let $q'\le q$ be a pair of elements in $W(p')$ with $S'_i(q')\neq \emptyset$ and $\ell(q)\geq\ell(p')+ m(S_i)$. By Definition~\ref{d20}\eqref{C2cd20},  $S'_i(q')=S_i(w(p,q'))$, hence it follows that $w(p,q')\le w(p,q)$ is a pair of elements in $W(p)$ with $S_i(w(p,q'))\neq \emptyset$. Set $\gamma:=\max(S_i(w(p,q')))$. By Definition~\ref{labeled-p-tree}\eqref{C4ptree}, $(\gamma,w(p,q))\in R$ hence the definition of $R$ yields $(\gamma,q)\in R$. Noting that $\gamma=\max(S'_i(q'))$ it finally follows that  $m(S'_i)\leq m(S_i)$.

\item[\eqref{type3}] This follows from Definition~\ref{d45}\eqref{pitchfork}.
\item[\eqref{type6}] Let $a\in A$. If $a=\myceil{\pi(a)}{\mathbb{A}}$ then  $a=(\pi(a),\emptyset)$, and so $\tp(\myceil{\pi(a)}{\mathbb{A}})$ is  the empty sequence. Conversely, if $\tp(a)$ is the empty sequence then Definition~\ref{defntypes} implies that $a$ takes the form $(\pi(a),\emptyset)$, hence $a=\myceil{\pi(a)}{\mathbb{A}}$.

\item[\eqref{type4}] 	Write $a$ as $(p,\langle S_i\mid i<\dom(\tp(a))\rangle)$ and let $\alpha\in\mu\setminus \dom(\tp(a))$.  There are two cases to consider:

$\br$ 	If $\dom(\tp(a))=0$, then let $a{}^{\curvearrowright\alpha}:=(p,\langle T_i\mid i\le\alpha\rangle)$,
where $T_i:W(p)\rightarrow\{\emptyset\}$ is constant for every $i\le\alpha$.

$\br$ 	Otherwise, say $\dom(\tp(a))=\beta+1$, let $a{}^{\curvearrowright\alpha}:=(p,\langle T_i\mid i\le\alpha\rangle)$,
where $T_i:=S_{\min\{i,\beta\}}$ for every $i\le\alpha$.	

It is routine to check that $a{}^{\curvearrowright\alpha}$ is as desired.

\item[\eqref{newstretch}] Write $b=(p',\vec{S}')$ and $a=(p,\vec{S})$ and set $\gamma:=\dom(\tp(b))$. 
If 	$\gamma=0$ then $b{}^{\curvearrowright \alpha}\unlhd a{}^{\curvearrowright\alpha}$ follows simply from $p'\le p$. Otherwise,  $\gamma$ takes the form $\beta+1$ and  the above clause yields $b{}^{\curvearrowright \alpha}=(p',\vec{T}')$, where $\vec{T}':=\langle T'_i\mid i\leq \alpha\rangle$ and $T'_i:=S'_{\min\{i,\beta\}}$. Similarly, $a{}^{\curvearrowright \alpha}=(p,\vec{T})$, where $\vec{T}:=\langle T_i\mid i\leq \alpha\rangle$ and $T_i:=S_{\min\{i,\beta\}}$. Using that $b\unlhd a$, Definition~\ref{d20} yields $b{}^{\curvearrowright\alpha}\unlhd a{}^{\curvearrowright\alpha}$, as wanted.

\item[\eqref{type5}] Let $a=(p,\vec{S})\in A$. To avoid trivialities, let us assume that $\vec{S}\neq \emptyset$.

$\br$ Suppose $p$ is incompatible with $r^\star$. 
Then, by Remark~\ref{incompatiblewithrstar}, for all $i<\dom(\tp(a))$ and all $q\in W(p)$, $S_i(q)=\emptyset$.
Therefore, $\mtp(a)=0$.
Using Definition~\ref{SigmaPrikry}\eqref{c2}  find $p'\le^0 p\in \z{P}$ and set $b:=\fork{a}(p')$.  Combining  Clauses~\eqref{type2} and \eqref{type3} above with the fact that $\mtp(a)=0$ it easily follows that $\mtp(b)=0$. Also,  $\pi(b)=p'\in \z{P}_{\ell(p)}$. Thus,   $b\in\z{A}_{\ell(p)}\downarrow a$, as wanted.

$\br$ Suppose $p\le r^\star$. The following claim will give us the desired condition.
\end{enumerate}	

\begin{claim}\label{claim244} 
Let $\epsilon<\mu$.
There exist $\alpha>\epsilon$ and $q\le^0\pi(a)$ such that $(\alpha,q)\in R$.
Furthermore, there exist $\alpha>\epsilon$ and
$b=(q,\vec T)\unlhd^0 a$ such that $b\in\z{\mathbb A}$, $\dom(\vec T)=\alpha+1$, and
 for all $r\in W(q)$, $\max(T_\alpha(r))=\alpha$. 
\end{claim}
\begin{proof}
Since $(\mathbb P,\lh,c)$ is $\Sigma$-Prikry, we infer from Definition~\ref{SigmaPrikry}(\ref{csize}) that $|W(p)|<\mu$.
Thus, by possibly extending $\epsilon$, we may assume that $S_i(q)\s\epsilon$, for all $q\in W(p)$ and $i\in\dom(\tp(a))$.
By Clause~\eqref{type4}, we may also assume that $\dom(\tp(a))$ is a successor ordinal, say, it is $\beta+1$.

As $p\le r^\star$, by the very same proof of \cite[Claim~5.6.2(1)]{partI} and using Clause~\eqref{c2} of Definition~\ref{SigmaPrikry}, we may fix $(\alpha,q)\in R$ with $\alpha>\beta+\epsilon$, $q\le^0 p$ and $q\in \z{P}_{\ell(p)}$.
Define $\vec T=\langle T_i:W(q)\rightarrow[\mu]^{<\mu}\mid i\leq\alpha\rangle$ by letting for all $r\in W(q)$ and $i\in\dom(\vec T)$:
$$T_i(r):=\begin{cases}
S_i(w(p,r)),&\text{if }i\leq\beta;\\
S_{\beta}(w(p,r))\cup\{\alpha\},&\text{otherwise}.
\end{cases}$$
It is easy to see that $T_i$ is a labeled $q$-tree for each $i\leq \alpha$. By Definitions \ref{strategy},  \ref{d20} and \ref{d45},
we also have that $b=(q,\vec T)$ is a condition in $\mathbb A$ with $b\unlhd^0a$ and $\pi(b)=q\in\z{P}_{\ell(p)}$.
As $(\alpha,q)\in R$, then $(\alpha,r)\in R$ for all $r\le q$, hence $\mtp(b)=0$. Therefore,  $b$ is a condition in $\z{\mathbb A}$ with the desired properties. 
\end{proof}

This completes the proof.
\end{proof}

\begin{lemma}[Weak Mixing Property]\label{mixinglemma}
For all $a\in A$, $n<\omega$, $\vec r$, and $p'\le^0 \pi(a)$,
and for every function $g:W_n(\pi(a))\rightarrow \mathbb{A}\downarrow a$, if there exists an ordinal $\iota$ such that all of the following hold:
\begin{enumerate}
\item $\vec r=\langle r_\xi \mid \xi<\chi\rangle$ is a good enumeration of $W_n(\pi(a))$;
\item  $\langle \pi(g(r_\xi))\mid \xi<\chi\rangle$ is diagonalizable with respect to $\vec r$, as witnessed by $p'$;
\item for every $\xi<\chi$:
\begin{itemize}
\item if $\xi<\iota$, then $\dom(\tp(g(r_\xi))=0$;
\item if $\xi=\iota$, then $\dom(\tp(g(r_\xi))\ge 1$;
\item if $\xi>\iota$, then $\dom(\tp(g(r_\xi))>(\sup_{\eta<\xi}\dom(\tp(g(r_\eta)))+1$;
\end{itemize}
\item for all $\xi\in (\iota,\chi)$ and $i\in[\dom(\tp(a)),\sup_{\eta<\xi}\dom(\tp(g(r_\eta)))]$,
$$\tp(g(r_\xi))(i)\leq\mtp(a),$$
\item
$\sup_{\iota\leq \xi<\chi} \mtp(g(r_\xi))<\omega$,
\end{enumerate}
then there exists $b\unlhd^0 a$ with $\pi(b)=p'$ and $\mtp(b)\leq n+\sup_{\iota\leq \xi<\chi} \mtp(g(r_\xi))$, such that for all  $q'\in W_n(p')$, $$\fork{b}(q')\unlhd^0 g(w(\pi(a),q')).$$
\end{lemma}
\begin{proof} Let $a:=(p,\vec{S})$. 
For each $\xi<\chi$, set $(p_{\xi}, \vec{S}^{\xi}):=g(r_\xi)$.

\begin{claim}\label{iotachiOK}
If $\iota\ge\chi$ then there is $b\in A$ as in the lemma.
\end{claim}
\begin{proof}
If $\iota\ge\chi$ then Clause~\eqref{Mixing3} yields $\dom(\tp(g(r_\xi))=0$ for all $\xi<\chi$. Hence, Clause~\eqref{type6} of Definition~\ref{type} yields $g(r_\xi)=\myceil{p_\xi}{\mathbb{A}}$ for all $\xi<\chi$. In particular also $a=\myceil{p}{\mathbb{A}}$. Set $b:=\myceil{p'}{\mathbb{A}}$, where $p'$ is given by  Clause~\eqref{Mixing2}.

Clearly, $\pi(b)=p'$ and $b\unlhd^0 a$. 
Let $q'\in W_n(p')$. By Clause~\eqref{Mixing2} above, $q'\le^0 p_\xi$, where $\xi$ is such that $r_\xi=w(p, q')$.  
Finally, Definition~\ref{forking}\eqref{frk6} yields $\fork{b}(q')=\myceil{q'}{\mathbb{A}}\unlhd^0\myceil{p_\xi}{\mathbb{A}}=g(r_\xi)$, as desired.
\end{proof}

Hereafter let us assume that $\iota<\chi$. For each $\xi\in[\iota,\chi)$,  Clause~\eqref{Mixing3} and Definition~\ref{defntypes} together imply that $\dom(\vec{S}^\xi)=\alpha_\xi+1$ for some $\alpha_\xi<\mu$. 
Moreover, Clause~\eqref{Mixing3} yields $\sup_{\iota\leq \eta<\xi} \alpha_\eta<\alpha_\xi$ for all $\xi\in (\iota,\chi)$. 
Likewise,  the same clause  implies that $g(r_\xi)=\myceil{p_\xi}{\mathbb{A}}$, hence $\vec{S}^\xi=\emptyset$, for all $\xi<\iota$.

Let $\langle s_\tau\mid \tau<\theta\rangle$ be a good enumeration  $W_n(p')$.
By Fact~\ref{forkingfacts}, $\theta<\mu$.
For each $\tau<\theta$, set $r_{\xi_\tau}:=w(p,s_\tau)$.
By Clause~\eqref{Mixing1} above,  for each $\tau<\theta$, $$s_\tau\le^0 \pi(g({w(p,s_\tau)}))=\pi(g(r_{\xi_\tau}))=p_{\xi_\tau}.$$

Set $\alpha':=\sup_{\iota\leq \xi<\chi}\alpha_{\xi}$ and $\alpha:=\sup(\dom(\vec{S}))$.\footnote{Note that $a$ might be $\myceil{p}{\mathbb{A}}$, so we are allowing $\alpha=0$.} 
By regularity of $\mu$ and Clause~\eqref{Mixing3} above it follows that $\alpha< \alpha'<\mu$. 
Our goal is to define a sequence $\vec{T}=\langle T_i:W(p')\rightarrow[\mu]^{<\mu}\mid i\leq\alpha'\rangle$ for which $b:=(p',\vec{T})$ is a condition satisfying the conclusion of the lemma.

As  $\langle s_\tau\mid \tau<\theta\rangle$ is a good enumeration of the $n^{th}$-level of the $p'$-tree $W(p')$,
Fact~\ref{lemma7} entails that, for each $q\in W({p'})$, there is a unique ordinal $\tau_q<\theta$, such that $q$ is comparable with $s_{\tau_q}$.
It thus follows from Fact~\ref{lemma7}(3) that, for all $q\in W(p')$, $\ell(q)-\ell(p')\geq n$ iff $q\in W(s_{\tau_q})$.
Moreover, for each $q\in W_{\geq n}(p')$, $q\le s_{\tau_q}\le^0 p_{{\xi_{\tau_q}}}$, hence $w(p_{_{\xi_{\tau_q}}},q)$ is well-defined.

Now, for all $i\leq\alpha'$ and $q\in W({p'})$, let:
$$T_i(q):=
\begin{cases}
S^{\xi_{\tau_q}}_{\min\{i,\alpha_{\xi_{\tau_q}}\}}(w(p_{\xi_{\tau_q}},q)),& \text{if }q\in {W(s_{\tau_q})}\;\& \;\iota\leq \xi_{\tau_q};\\
S_{\min\{i,\alpha\}}(w(p,q)),& \text{if }q\notin {W(s_{\tau_q})}\;\&\; \alpha>0;\\
\emptyset, & \text{otherwise.}
\end{cases}$$
\begin{claim}\label{claim5211}
Let $i\leq\alpha'$. Then $T_i$ is a labeled $p'$-tree.
\end{claim}
\begin{proof}
Fix $q\in W(p')$ and let us go over the Clauses of Definition~\ref{labeled-p-tree}. The verifications of \eqref{C1ptree}--\eqref{C3ptree} are similar to that of \cite[Claim 6.16.1]{partI}, so we just provide details for the new Clause~\eqref{C4ptree}.

For each $i< \alpha'$, set
$$\xi(i):=\min\{\xi\in[\iota,\chi)\mid i\leq \alpha_\xi\}.$$

\begin{subclaim}
If $i< \alpha'$, then
$$m(T_i)\leq n+\max\{\mtp(a),\sup\nolimits_{\iota\leq \eta<\xi(i)} \mtp(g(r_\eta)),\tp(g(r_{\xi(i)})(i)\}.$$
\end{subclaim}

\begin{proof}
Let $q'\le q$ be  in $W(p')$ with $q\in W_k(p')$, where
$$k\geq n+\max\{\mtp(a),\sup\nolimits_{\iota\leq \eta<\xi(i)} \mtp(g(r_\eta)),\tp(g(r_{\xi(i)})(i)\}.$$
Suppose that  $T_i(q')\neq \emptyset$. Denote $\tau:=\tau_{q'}$ and $\delta:=\max(T_i(q'))$. 
Since $\ell(q)\geq \ell(p')+n$, note that  $q,q'\in W(s_\tau)$. Also, $\iota\leq \xi_\tau$, as otherwise $T_i(q')=\emptyset$. Therefore, we fall into the first option of the casuistic getting
 $$T_i(q')=S^{{\xi_\tau}}_{\min\{i,\alpha_{\xi_\tau}\}}(w(p_{\xi_\tau},q')).$$

$\br$  Assume that $\xi_\tau<\xi(i)$. Then,  $\alpha_{\xi_\tau}<i$ and so  $$T_i(q')=S^{\xi_\tau}_{\alpha_{\xi_\tau}}(w(p_{\xi_\tau},q')).$$
We have that $w(p_{\xi_\tau},q')\le w(p_{\xi_\tau},q)$ is a pair in $W_{k-n}(p_{\xi_\tau})$ and that the set $S^{\xi_\tau}_{\alpha_{\xi_\tau}}(w(p_{\xi_\tau},q'))$ is non-empty. Also, $k-n\geq \mtp(g(r_{\xi_\tau}))=m(S^{\xi_\tau}_{\alpha_{\xi_\tau}})$. So, by Clause~\eqref{C4ptree} for $S^{\xi_\tau}_{\alpha_{\xi_\tau}}$, we have that $(\delta, w(p_{\xi_\tau},q))\in R$, and thus $(\delta,q)\in R$.

\smallskip

$\br$  Assume that $\xi(i)\leq  \xi_\tau$. Then
 $i\leq \alpha_{\xi(i)}\leq \alpha_{\xi_\tau}$, and thus $$T_i(q')=S^{\xi_\tau}_i(w(p_{\xi_\tau},q')).$$
 If $\dom(\tp(a))\leq i\leq \sup_{\iota\leq \eta<\xi(i)}\alpha_\eta$,  by Clause~\eqref{Mixing4} above,  $$\tp(g(r_{\xi_\tau}))(i)\leq \mtp(a).$$
 Otherwise, if $\sup_{\iota\leq \eta<\xi(i)}\alpha_\eta<i\leq \alpha_{\xi(i)}$, again by Clause~\eqref{Mixing4} above
 $$\tp(g(r_{\xi_\tau}))(i)\leq \max\{\mtp(a),\tp(g(r_{\xi(i)})(i)\}.$$
 In either case,  $w(p_{\xi_\tau},q)\in W_{k-n}(p_{\xi_\tau})$ and  $k-n\geq \tp(g(r_{\xi_\tau}))(i)= m(S^{\xi_\tau}_i)$. So by Clause~\eqref{C4ptree} of $S^{\xi_\tau}_i$  we get that $(\delta, w(p_{\xi_\tau},q))\in R$, hence $(\delta,q)\in R$.
\end{proof}

\begin{subclaim}
$m(T_{\alpha'})\leq n+\sup_{\iota\leq \xi<\chi} \mtp(g(r_\xi))$.
\end{subclaim}
\begin{proof}
Let $q'\le q$ be  in $W(p')$ with $q\in W_k(p')$ and
$k\geq  n+\sup_{\iota\leq \xi<\chi} \mtp(g(r_\xi))$,
and suppose that  $T_{\alpha'}(q')\neq \emptyset$.
 Denote $\tau:=\tau_{q'}$ and $\delta:=\max(T_{\alpha'}(q'))$.

Since $k\geq n$, $q,q'\in W(s_{\tau})$. Also, $\iota\leq \xi_\tau$, as otherwise $T_{\alpha'}(q')=\emptyset$.   
Hence, $T_{\alpha'}(q')=S^{\xi_{\tau}}_{\alpha_{\xi_\tau}}(w(p_{\xi_{\tau}},q')).$
Then $w(p_{\xi_{\tau}},q')\leq w(p_{\xi_{\tau}},q)$ is a pair in $W_{k-n}(p_{\xi_\tau})$ with $k-n\geq \mtp(g(r_{\xi_\tau}))= m(S^{\xi_\tau}_{\alpha_{\xi_\tau}})$. 
So, by Definition~\ref{labeled-p-tree}\eqref{C4ptree} regarded with respect to $S^{\xi_\tau}_{\alpha_{\xi_\tau}}$, it follows that  $(\delta,w(p_{\xi_{\tau}},q))\in R$. Thus, $(\delta,q)\in R$, as wanted.
\end{proof}

The combination of the above subclaims yield Clause~\eqref{C4ptree} for $T_i$. \qedhere
\end{proof}

\begin{claim} The sequence $\vec{T}=\langle T_i:W(p')\rightarrow[\mu]^{<\mu}\mid i\leq\alpha'\rangle$ is a $p'$-strategy.
\end{claim}
\begin{proof} We need to go over the clauses of Definition~\ref{strategy}.
However, Clause~\eqref{C1pstrategy} is trivial, Clause~\eqref{C2pstrategy} is established in the preceding claim,
and Clauses \eqref{C3pstrategy} and \eqref{C5pstrategy} follow from the corresponding features of $\vec S$ and the $\vec S^{r^\tau}$'s. Finally, Clause~\eqref{C4pstrategy} can be proved similarly to \cite[Claim 6.16.2]{partI}, noting that if $\alpha>0$ then $\iota=0$. 
\end{proof}

Thus, we have established that $b:=(p',\vec{T})$ is a legitimate condition in $\mathbb{A}$, such that $\mtp(b)\leq n+\sup_{\xi<\chi} \mtp(g(r_\xi))$.

\medskip

The next series of claims take care of the rest of the lemma:

\begin{claim} Let $\tau<\theta$. For each $q\in W_n(s_\tau)$, $w(p',q)=w(s_\tau,q)=q$.
\end{claim}
\begin{proof}  The first equality can be proved exactly as in \cite[Claim 6.16.4]{partI}.
For the second, notice that $q$ and $w(s_\tau, q)$ are conditions in $W(s_\tau)$ with the same length. Hence, Fact~\ref{lemma7}(2) yields $q=w(s_\tau,q)$, as wanted.
\end{proof}

\begin{claim}\label{claim5123} $\pi(b)=p'$ and $b\unlhd^0 a$.
\end{claim}
\begin{proof}
The proof of this can be found in \cite[Claim 6.16.3]{partI}.
\end{proof}

\begin{claim}
For each $\tau<\theta$, $\fork{b}(s_\tau)\unlhd^0 g(r_{\xi_\tau})$.\footnote{Recall that $\langle s_\tau\mid \tau<\theta\rangle$ was a good enumeration of $W_n(p')$. }
\end{claim}
\begin{proof} Let $\tau<\theta$ and $\vec{T}^\tau$ be denote the $s_\tau$-strategy such that $\fork{b}(s_\tau)=(s_\tau,\vec{T}_\tau)$. 
By Corollary \ref{forkingindeed},  we have that  $\pi(\fork{b}(s_\tau))=s_\tau\le^0 p_{\xi_\tau}$. 
 
If $\xi_\tau<\iota$, then
 $\fork{b}(s_\tau)\unlhd^0 \myceil{p_{\xi_\tau}}{\mathbb{A}}=g(r_{\xi_\tau})$, and we are done.
 
 So, let us assume that $\iota\leq \xi_\tau$. Let $i\leq \alpha_{\xi_\tau}$ and $q\in W(s_\tau)$. By Definition~\ref{d45}\eqref{pitchfork}, $T^\tau_i(q)=T_i(w(p',q))$ and by one of the preceding claims, $w(p',q)=w(s_\tau,q)=q$, hence $T^\tau_i(q)=T_i(q)$. Also $r_{\xi_{\tau_q}}=w(p,s_{\tau_q})=w(p,s_\tau)=r_{\xi_\tau}$, where the second last equality follows from  $q\in W(s_\tau)$. Therefore, $$T^\tau_i(q)=S^{\xi_\tau}_{\min\{i,\alpha_{\xi_\tau}\}}(w(p_{\xi_\tau},q))=S^{\xi_\tau}_{i}(w(p_{\xi_\tau},q)).$$ Altogether, $\fork{b}(s_\tau)\unlhd^0 g(r_{\xi_\tau})$, as wanted.
\end{proof}
The above claims yield the proof of the lemma.
\end{proof}

Combining Lemmas \ref{forkingindeed} and \ref{mixinglemma} we arrive at:

\begin{cor}\label{forkingwithmixingindeed}
$(\pitchfork,\pi)$ is a forking projection from $(\mathbb{A},\ell_\mathbb{A},c_\mathbb{A})$ to $(\mathbb{P},\ell,c)$ having the weak  mixing property.\qed
\end{cor}

Now we take advantage of the preceding corollary to establish that $(\mathbb{A},\ell_\mathbb{A},c_\mathbb{A})$ is  $\Sigma$-Prikry and that $(\mathbb{A},\ell_\mathbb{A})$ has property $\mathcal{D}$. On this respect, note that the latter statement follows combining Corollary \ref{forkingwithmixingindeed},  Lemma~\ref{MixingLiftsPropertyD} and property $\mathcal{D}$ of $(\mathbb{P},\ell)$ (Setup~\ref{ReflectionAfterIteration}). For the former let us go over the clauses of Definition~\ref{SigmaPrikry}: Clauses \eqref{c4},\eqref{c1},\eqref{c5},\eqref{csize} and \eqref{itsaprojection} follow from lemmas 4.5,  4.7,  4.8 and 4.9 of \cite{partI}, respectively. Clause~\eqref{c6} follows combining  property $\mathcal{D}$ of $(\mathbb{P},\ell)$ with Corollary \ref{forkingwithmixingindeed} and Corollary \ref{PropertyDplusMixingYieldsCPP}. Also, by \cite[Corollary 4.13]{partI}, $\one_\mathbb{A}\forces_\mathbb{A} \check{\mu}=\check{\kappa}^+$. 
Finally, note that Clause~\eqref{c2} follows from Lemma~\ref{forkinganddirectedclosure} together with  Corollary~\ref{forkingwithmixingindeed} and Fact~\ref{C1ASigmaPrirky}.

Altogether, we arrive at the main result of this section:

\begin{cor}\label{onestep} Suppose:
\begin{itemize}
\item[(i)] $(\mathbb P,\ell,c)$ is a $\Sigma$-Prikry notion of forcing such that the pair $(\mathbb{P},\ell)$ has property $\mathcal{D}$;
\item[(ii)] $\one_{\mathbb P}\forces_{\mathbb{P}}\check\mu=\kappa^+$;
\item[(iii)] $\mathbb P=(P,\le)$ is a subset of $H_{\mu^+}$;
\item[(iv)] $r^\star\in P$ forces that $\dot T$ is a $\mathbb P$-name for some subset $T$ of $(E^\mu_\omega)^V$
such that, for all $\gamma<\mu$ with $\omega<\cf^V(\gamma)<\kappa$, $T\cap\gamma$ is nonstationary in $\gamma$.

\end{itemize}

Then, there exists a $\Sigma$-Prikry triple $(\mathbb A,\lh_{\mathbb A},c_{\mathbb A})$ such that $(\mathbb{A},\ell_{\mathbb{A}})$ has property $\mathcal{D}$ and for which the following are true:
\begin{enumerate}
\item $(\mathbb A,\lh_{\mathbb A},c_{\mathbb A})$ admits a forking projection $({\pitchfork},{\pi})$ to $(\mathbb P,\lh,c)$ that has the weak mixing property;
\item for each $n<\omega$, $\z{\mathbb{A}}^\pi_n$ is $\mu$-directed-closed;
\item $\one_{\mathbb A}\forces_{\mathbb A}\check\mu=\kappa^+$;
\item $\mathbb A=(A,\unlhd)$ is a subset of $H_{\mu^+}$;
\item $\myceil{r^\star}{\mathbb A}$ forces that $\dot T$ is nonstationary.
\end{enumerate}
\end{cor}
\begin{proof}
Item (1) and the assertion that $(\mathbb A,\lh_{\mathbb A},c_{\mathbb A})$ is $\Sigma$-Prikry and that $(\mathbb{A},\ell_{\mathbb{A}})$ has property $\mathcal{D}$   follow from our previous arguments.
Item (2) follows from Fact \ref{C1ASigmaPrirky} and items (3)--(5) already appeared in Fact \ref{FactPartI} (See also Remark~\ref{RemarkKillingT}).
\end{proof}

\subsection{Connecting the dots}\label{connecting}

\noindent For the rest of this section, we make the following assumptions:
\begin{itemize}
\item $\Sigma=\langle \kappa_n\mid n<\omega\rangle$ is an increasing sequence of Laver-indestructible supercompact cardinals;
\item $\kappa:=\sup_{n<\omega}\kappa_n$, $\mu:=\kappa^+$ and $\lambda:=\kappa^{++}$;
\item $2^\kappa=\kappa^+$ and $2^\mu=\mu^+$;
\item $\Gamma:=\{\gamma<\mu\mid \omega<\cf^V(\gamma)<\kappa\}$.
\end{itemize}

Under these assumptions, \cite[Corollary~5.11]{partI}
reads as follows:

\begin{fact}\label{c27}  If $(\mathbb P,\lh,c)$ is a $\Sigma$-Prikry notion of forcing
such that  $\one_{\mathbb P}\forces_{\mathbb{P}}\check\mu=\kappa^+$,
then  $V^{\mathbb P}\models \refl({<}\omega,\Gamma)$.
\end{fact}

We now want to appeal to the iteration scheme of the previous section. For this, we need to introduce our three building blocks of choice.

\blk{I} Let $\mathbb{Q}$ be  the Extender Based Prikry Forcing (EBPF)
for blowing up $2^\kappa$ to $\kappa^{++}$.
By results in \cite[Ch.10, \S2.5]{Pov}, this notion of forcing
can be regarded as a $\Sigma$-Prikry triple
$(\mathbb{Q},\ell,c)$ for which $(\mathbb{Q},\ell)$ has property $\mathcal{D}$,
$\mathbb Q$ is a subset of $H_{\mu^+}$, and $\one_{\mathbb Q}\forces_{\mathbb Q}\check\mu=\kappa^+$.
Furthermore, for each $n<\omega$, $\mathbb{Q}_n$ is $\kappa_n$-directed-closed, so we set $\z{\mathbb{Q}}_n:=\mathbb{Q}_n$.
Finally, as $\kappa$ is singular, $\one_{\mathbb Q}\forces_{\mathbb{Q}}``\kappa\text{ is singular}"$.

\blk{II} For every $\Sigma$-Prikry triple $(\mathbb P,\lh_{\mathbb P},c_{\mathbb P})$ having property $\mathcal{D}$
such that
$\mathbb P=\left(P,\le\right)$ is a subset of $H_{\mu^+}$ and
$\one_{\mathbb P}\forces_{\mathbb P}\check\mu=\kappa^+$,
every $r^\star\in P$, and every $\mathbb P$-name $z\in H_{\mu^+}$,
we obtain a corresponding $\Sigma$-Prikry triple $(\mathbb A,\lh_{\mathbb A},c_{\mathbb A})$ as follows:
\begin{itemize}
\item[$\br$] If $r^\star\in P$ forces that $z$ is a $\mathbb P$-name for a stationary subset of $(E^\mu_\omega)^V$
that does not reflect in $\Gamma$, 
then we first let $\sigma$ be a $\mathbb{P}$-\emph{nice name} for a subset of $(E^\mu_\omega)^V$ such that $\one_{\mathbb P}\forces_{\mathbb{P}}\sigma=z$. 
Setting $\dot T:=\{(\check{\alpha},p)\in\sigma\mid p\text{ is compatible with }r^\star\}$, we then get that $r^\star\forces_{\mathbb P}\dot T=z$.
Furthermore, $\one_{\mathbb P}$ forces that $\dot T$ is a $\mathbb P$-name for some subset $T$ of $(E^\mu_\omega)^V$
such that, for all $\gamma\in\Gamma$, $T\cap\gamma$ is nonstationary in $\gamma$.
We then obtain $(\mathbb A,\lh_{\mathbb A},c_{\mathbb A})$ by appealing to Corollary~\ref{onestep} with  
the $\Sigma$-Prikry triple $(\mathbb P,\lh_{\mathbb P},c_{\mathbb P})$,
the condition $\one_{\mathbb P}$
and the name $\dot T$.
Consequently, $\myceil{\one_{\mathbb P}}{\mathbb A}$ forces that $\dot T$ is nonstationary,
so that $\myceil{r^\star}{\mathbb A}$ forces that $z$ is nonstationary.

\item[$\br$] Otherwise, 
we invoke Corollary~\ref{onestep} 
with the $\Sigma$-Prikry triple $(\mathbb P,\lh_{\mathbb P},c_{\mathbb P})$,
the condition $\one_{\mathbb P}$ and the name $\dot T:=\emptyset$.
\end{itemize}
In either case, we get:
\begin{enumerate}[label=(\alph*)]
\item $(\mathbb A,\lh_{\mathbb A},c_{\mathbb A})$ admits a forking projection $(\pitchfork,\pi)$ to $(\mathbb P,\lh_{\mathbb P},c_{\mathbb P})$ that has the weak mixing property; 
\item  for each $n<\omega$, $\z{\mathbb{A}}^\pi_n$  is $\kappa_n$-directed-closed;\footnote{Recall Footnote~\ref{FootnoteBBII} on page~\pageref{FootnoteBBII}.}
\item  $\one_{\mathbb A}\forces_{\mathbb A}\check\mu=\kappa^+$;
\item $\mathbb A=(A,\unlhd)$ is a subset of $H_{\mu^+}$;
\item\label{C8onestep} if $r^\star$ forces that $z$ is a $\mathbb P$-name for a stationary subset of $(E^\mu_\omega)^V$
that does not reflect in $\Gamma$, then $\myceil{r^\star}{\mathbb A}$ forces that $z$ is nonstationary.
\end{enumerate}

\blk{III} As $2^\mu=\mu^+$, we fix a surjection $\psi:\mu^+\rightarrow H_{\mu^+}$ such that the preimage of any singleton is cofinal in $\mu^+$.

\bigskip

We would like now to appeal to the iteration scheme of Section~\ref{Iteration} with these building blocks.
However, Lemma~\ref{CivIteration} partially bears on the extra assumption that for all $\delta\in\acc(\mu^++1)$ and $n<\omega$, $(\z{\mathbb{P}}_\delta)_n$ is dense in $(\mathbb{P}_\delta)_n$. Our next task will be checking that the iteration defined using the previous building blocks has this feature. Once we are done  we will prove Theorem \ref{TheoremReflection}, which yields the very first application of our iteration scheme.

\begin{definition}\label{def424} For every nonzero $\beta<\mu$, 
as $(\mathbb P_{\beta+1},\lh_{\beta+1},c_{\beta+1})$ is obtained from $(\mathbb P_{\beta},\lh_{\beta},c_{\beta})$ by invoking the above-mentioned \blkref{II},
we shall denote by $\langle\dot{C}^\beta_n\mid n<\omega\rangle$, and $R_\beta$ the corresponding objects appearing 
before Definition~\ref{labeled-p-tree} and involved in defining  $\mathbb P_{\beta+1}$.
In particular, $R_\beta\s \mu\times P_\beta$. 
We shall denote by $(\dot{T}^+)^\beta$ the $\mathbb{P}_\beta$-name defined after Remark~\ref{additionalfeatureofT+}, that is,
$$(\dot{T}^+)^\beta:=\{(\check{\alpha},p)\mid (\alpha,p)\in (E^\mu_\omega)^{V}\times P_\beta\;\&\; (p\forces_{(\mathbb{P}_\beta)_{\ell(p)}}\check{\alpha}\notin \dot{C}^\beta_{\ell(p)})\}.$$
As in Definition~\ref{tauns}, for each $n<\omega$, we put $\tau^\beta_n:=\{(\check{\alpha},p)\in(\dot{T}^+)^\beta\mid \alpha\in (E^\mu_\omega)^V\ \&\ p\in  (P_{\beta})_n\}$.
\end{definition}
\begin{remark}\label{RemarkAdditionalFeature}
By Lemma~\ref{ThekeylemmaaboutT+}, the definition of the $\mathbb{P}_\beta$-name  $(\dot{T}^+)^{\beta}$ entails that for every $p\in P_\beta$, setting $n:=\ell(p)$, 
$p\forces_{\mathbb{P}_n}\tau^{\beta}_{n}=(\check{E}^\mu_\omega\setminus \dot{C}^\beta_n)$.
This will be crucially used in the verification of Clause~\eqref{C4pstrategy} of Definition~\ref{labeled-p-tree},
towards the end of the proof of Claim~\ref{claim4243}.
\end{remark}

\begin{definition} Let $\delta\in[2,\mu^+]$ and $a\in P_\delta$. 
\begin{itemize}
\item For every nonzero $\beta<\delta$, if $\beta+1\in B_a$, 
then $(a\restriction\beta+1)=(a\restriction \beta){}^\smallfrown\langle\vec{S}\rangle$ for some nonempty sequence $\vec S$,
so we denote this sequence by $\vec S^{a,\beta}=\langle S_i^{a,\beta}\mid i\le \alpha^{a,\beta}\rangle$;
\item For every nonzero $\beta<\delta$ such that $\beta+1\in B_a$, 
and every $q\in W(a\restriction\beta)$, we let $\sigma^{a,\beta}(q):=\max(\{0\}\cup S^{a,\beta}_{\alpha^{a,\beta}}(q))$;
\item Let $\sigma^a:=\sup\{\sigma^{a,\beta}(q)\mid \beta+1\in B_a\setminus2\ \&\ q\in W(a\restriction\beta)\}$. 
\end{itemize}
\end{definition}

\begin{lemma}\label{Remarkonsigmas}
Let $\delta\in[2,\mu^+]$, $a\in P_\delta$ and $\beta+1\in B_a\setminus2$.  
\begin{enumerate}
\item\label{Remarkonsigmas1} $\sigma^{a,\beta}:W(a\restriction\beta)\rightarrow\mu$ is order-reversing;
In particular, for all $q\in W(a\restriction\beta)$, $\sigma^{a,\beta}(a\restriction\beta)\le \sigma^{a,\beta}(q)\le\sigma^a$;
\item\label{Remarkonsigmas2}   For all $r\le_\beta a\restriction\beta$ and $s\in W(r)$, 
letting $b:=\fork[\beta+1,\beta]{a\restriction(\beta+1)}(r)$,
$$\sigma^{b,\beta}(s)=\sigma^{a,\beta}(w(a\restriction\beta,s));$$
\item\label{Remarkonsigmas3}  For all $b\le_{\beta+1} a\restriction(\beta+1)$ and $q\in W(b\restriction \beta)$, 
$$\sigma^{b,\beta}(q)\ge \sigma^{a,\beta}(w({a\restriction\beta},q)).$$
\end{enumerate}
\end{lemma}
\begin{proof} (1) By Definition~\ref{labeled-p-tree}\eqref{C2ptree}.

(2) By Definition~\ref{d45}\eqref{pitchfork}.

(3) By Definition~\ref{strategy}\eqref{C3pstrategy} and Definition~\ref{d45}.
\end{proof}

\begin{lemma}\label{ringisdense}
For all $\delta\in[2,\mu^+]$, $n<\omega$, and $\epsilon<\mu$,
$$D^\epsilon_{\delta,n}:=\{b\in (\z{P}_\delta)_n\mid \forall\beta+1\in B_{b}\setminus2\,[\epsilon<\sigma^{b,\beta}(b\restriction\beta)]\}$$
is dense in $(\mathbb{P}_\delta)_{n}$. 

In particular, for all $\delta\in[2,\mu^+]$ and $n<\omega$,  $(\z{\mathbb{P}}_\delta)_n$ is dense in $(\mathbb{P}_\delta)_n$.
\end{lemma}
\begin{proof}
By induction on $\delta$. 
Suppose  that we are given $\delta\in[2,\mu^+]$ such that for all $\gamma\in[2,\delta)$, $n<\omega$ and $\epsilon<\mu$,
$D^\epsilon_{\gamma,n}$ is dense in $(\mathbb{P}_\gamma)_{n}$. 

\medskip
 
\underline{Case~1:} Suppose that  $\delta=\beta+1$ is a successor ordinal.
Let $a\in P_\delta$ and $\epsilon<\mu$ be arbitrary. Denote $n:=\lh_\delta(a)$.
Appealing to Claim~\ref{claim244}, find $a'\le^0_\beta(a\restriction\beta)$ and $\alpha>\max\{\epsilon, \sigma^a\}$ with $(\alpha,a')\in R_{\beta}$. 

If $\beta=1$, then set $b':=a'$; otherwise, 
appeal to the inductive hypothesis to pick $b'\in D_{\beta,n}^{\epsilon}$ extending $a'$. 
In either case, $(\alpha,b')\in R_\beta$.  

If $\beta+1\notin B_a$, then we are done by setting $b:=b'*\emptyset_\delta$, so suppose 
that $\beta+1\in B_a$. In particular, $a=(a\restriction\beta){}^\smallfrown\langle\langle S_i^{a,\beta}\mid i\le \alpha^{a,\beta}\rangle\rangle$.
Now, let $b:=b'{}^\smallfrown\langle\langle S_i\mid i\leq \alpha^{a,\beta}+1\rangle\rangle$, where for all $i\leq \alpha^{a,\beta}+1$ and $q\in W(b')$, $S_i(q)$ is defined as follows: 
$$S_i(q):=\begin{cases}
S_i^{a,\beta}(w(a\restriction\beta, q)),& \text{if $i\le\alpha^{a,\beta}$;}\\
S_i^{a,\beta}(w(a\restriction\beta, q))\cup\{\alpha\}, & \text{otherwise.}
\end{cases}
$$

As $(\alpha,b\restriction\beta)=(\alpha,b')\in R_\beta$, we infer that $b\in P_\delta$ and $\mtp_{\beta+1}(b)=0$. 
Thus, since $b'\in \z{P}_\beta$,  it follows that $b\in\z{P}_\delta$. Finally, since $B_b=B_{b'}\cup \{\beta+1\}$, our choice of $b'$ implies that $b$ is an element of $D_{\delta,n}^\epsilon$
extending $a$.
\medskip

\underline{Case~2:}  Suppose that $\cf(\delta)>\kappa$.
Let $a\in P_\delta$ and $\epsilon<\mu$ be arbitrary. Denote $n:=\lh_\delta(a)$.
Then $B_a$ is bounded in $\delta$.
Fix $\gamma<\delta$ such that $a=(a\restriction\gamma)\ast \emptyset_\delta$. 
By the inductive hypothesis, we find $a'\in D_{\gamma,n}^\epsilon$ extending $a\restriction\gamma$.
Set $b:=a'\ast \emptyset_\delta$, so that $B_b=B_{a'}$.
Then $b\in D_{\delta,n}^\epsilon$ extends $a$, as desired.

\medskip

\underline{Case~3:}  Suppose that $1<\cf(\delta)\le\kappa$.
As $\kappa$ is the limit of the strictly increasing sequence $\langle \kappa_n\mid n<\omega\rangle$ (recall the opening of Subsection~\ref{connecting}),
we may let $m<\omega$ be the least such that $\cf(\delta)<\kappa_m$. 

\begin{claim}\label{towardsdagger}
For all $n\ge m$ and $\epsilon<\mu$, $D_{\delta,n}^\epsilon$ is dense in $(\mathbb P_\delta)_n$.
\end{claim}
\begin{proof} Let $\epsilon<\mu$ and let $a\in P_\delta$ be such that $n:=\lh_\delta(a)$ is $\ge m$.
By the proof of Case~2, we may assume that $B_a$ is unbounded in $\delta$. 
Let $\langle \gamma_\tau\mid \tau<\cf(\delta)\rangle$ be the increasing enumeration of a small cofinal subset of $B_a\setminus2$,
and set $\gamma_{\cf(\delta)}:=\delta$.
We now construct a sequence of conditions $\langle b_\tau\mid \tau\le \cf(\delta)\rangle\in\prod_{\tau\le\cf(\delta)}D^\epsilon_{\gamma_\tau,n}$ in such a way that, 
for all $\eta<\tau\le\cf(\delta)$, $b_\eta\le^0_{\gamma_\eta} a\restriction \gamma_\eta$ and $b_\eta= b_\tau\restriction \gamma_\eta$.
The construction is by recursion on $\tau\le\cf(\delta)$, as follows:

$\br$ For $\tau=0$, use the induction hypothesis to find $b_0\le^0_{\gamma_0} a\restriction\gamma_0$ in $D^\epsilon_{\gamma_0,n}$. 

$\br$ For every $\tau<\cf(\delta)$ such that $b_\tau$ has already been defined,
use the induction hypothesis to find $b_{\tau+1}\le^0_{\gamma_{\tau+1}} \fork[\gamma_{\tau+1},\gamma_\tau]{a\restriction \gamma_{\tau+1}}(b_{\tau})$ in $D^\epsilon_{\gamma_{\tau+1},n}$.

$\br$ For every $\tau\in\acc(\cf(\delta)+1)$ such that $\langle b_\eta\mid \eta<\tau\rangle$ has already been defined,
we get from the induction hypothesis together with Clause~\eqref{type6} of Definition~\ref{type}
that $\langle b_\eta*\emptyset_{\gamma_\tau}\mid \eta<\tau\rangle$ is a $\le^0_{\gamma_\tau}$-decreasing sequence in $(\z{\mathbb{P}}_{\gamma_\tau})_{n}$. 
Thus, by Lemma~\ref{Morethanclosure}, we may find a lower bound $b_\tau$ in $(\z{\mathbb{P}}_{\gamma_\tau})_{n}$ such that $B_{b_\tau}=\bigcup_{\eta<\tau} B_{b_\eta}$.
Consequently, $b_\tau\in D^\epsilon_{\gamma_\tau,n}$.

At the end of the above process, we have obtained $b_{\cf(\delta)}$ which is an element of $D^\epsilon_{\delta,n}$ extending $a$,
as desired.
\end{proof}

For each $n<\omega$, let us say that $\dagger_{\delta,n}$ holds iff, for all $\epsilon<\mu$, $D_{\delta,n}^\epsilon$ is dense in $(\mathbb P_\delta)_n$.
By Claim~\ref{towardsdagger}, $\dagger_{\delta,n}$ holds  for all $n\ge m$.
In particular, if  $m=0$, then we are done. 
To address the general case, 
we now assume that we are given $n<\omega$ such that $\dagger_{\delta,n+1}$ holds, and we shall prove that $\dagger_{\delta,n}$ holds, as well.

Let $\epsilon<\mu$, and let $a\in P_\delta$ with $\lh_\delta(a)=n$;
we need to find a condition $b\in D_{\delta,n}^\epsilon$ extending $a$. As a first step, we prove the following claim.
\begin{claim} There exists a $\le^0_\delta$-decreasing sequence of conditions $\langle a_j\mid j<\omega\rangle$ and an increasing sequence of ordinals $\langle\epsilon_j\mid j<\omega\rangle$  such that, for all $j<\omega$ and $\beta+1\in B_{a_j}\setminus2$, 
the following three hold:
\begin{enumerate}[label=(\Roman*)]
\item \label{largerthandelta} For every $q\in W_{1}(a_{j+1}\restriction\beta)$,  $\sigma^{a_j}\leq \epsilon_j<\sigma^{a_{j+1},\beta}(q)$;
\item \label{tryingtofallintothering} For every $q\in W_{\geq 1}(a_{j+1})$, $q\restriction\beta\forces_{(\mathbb{P}_\beta)_{\ell_\delta(q)}}\dot{C}^\beta_{\ell_\delta(q)}\cap (\check{\epsilon}_j,\check{\epsilon}_{j+1})\neq\emptyset$;
\item\label{point} $a_j\le^0_\delta a$ and $\epsilon_j>\epsilon$.
\end{enumerate}
\end{claim}
\begin{proof}
The construction is by recursion on $j<\omega$.
We start by setting $(a_0,\epsilon_0):=(a,\max\{\sigma^a,\epsilon+1\})$. This will take care of Clause~\ref{point}.
Next, suppose that $j<\omega$ and that the pair $(a_j,\epsilon_j)$ has already been successfully defined.
Since $\dagger_{\delta,n+1}$ holds,  $D_{\delta,n+1}^{\epsilon_j}$
is a set of conditions in $(\z{\mathbb{P}}_\delta)_{n+1}$
which is dense in $(\mathbb{P}_\delta)_{n+1}$. 
Let $\vec{s}=\langle s_\xi\mid \xi<\chi\rangle$ be a good enumeration of $W_1(a_j)$.
By Lemma~\ref{lemmapropertyDforlimitMoreOver}, 
we may now use the winning strategy of $\pI$ in playing the game $\Game_{\mathbb P_\delta}(a_j,\vec{s},D_{\delta,n+1}^{\epsilon_j})$,
thus obtaining a sequence $\langle b_\xi\mid \xi<\chi\rangle$ of conditions in $D_{\delta,n+1}^{\epsilon_j}$ along with a condition $a_{j+1}\le^0_\delta a_j$ such that $a_{j+1}$ diagonalizes $\langle b_\xi\mid \xi<\chi\rangle$ with respect to $\vec{s}$. 
Set $\epsilon_{j+1}:=\sup_{\xi<\chi}\sigma^{b_\xi}$ and note that $\epsilon_{j+1}<\mu$.
We now verify Clauses \ref{largerthandelta} and \ref{tryingtofallintothering}. For this, fix an arbitrary $\beta$ with $\beta+1\in B_{a_j}\setminus 2$.
\begin{enumerate}[label=(\Roman*)]
\item Let $q\in W_{1}(a_{j+1}\restriction\beta)$.
As $c:=\fork[\delta,\beta]{a_{j+1}}(q)$ is in $W_1(a_{j+1})$, we may let $\xi<\chi$ be the unique ordinal to satisfy $c\le^0_\delta b_\xi$. 
Since $b_\xi\le_\delta a_j$, $\beta+1\in B_{a_j}\s B_{b_\xi}$. Also, $c\restriction(\beta+1)\le^0_{\beta+1}b_\xi\restriction(\beta+1)$ and $c\restriction\beta=q$, so $w(b_\xi\restriction \beta, c\restriction\beta)=w(b_\xi\restriction\beta, q)=b_\xi\restriction\beta$. Thus, Clauses~\eqref{Remarkonsigmas2} and \eqref{Remarkonsigmas3} of Lemma~\ref{Remarkonsigmas} together yield
$$\epsilon_j<\sigma^{b_\xi,\beta}(b_\xi\restriction\beta)\leq \sigma^{c,\beta}(q)=\sigma^{a_{j+1},\beta}(w(a_{j+1}\restriction\beta, q)).$$
Since $q\in W(a_{j+1}\restriction\beta)$, the latter is equal to $\sigma^{a_{j+1},\beta}(q)$.
Finally, since $b_0\le_{\delta} a_j$, Lemma~\ref{Remarkonsigmas}\eqref{Remarkonsigmas3} yields $\sigma^{a_j}\leq \sigma^{b_0}\leq \epsilon_j$.

\item Let $q\in W_{\geq 1}(a_{j+1})$.
Appealing to Fact~\ref{lemma7},  let $\bar{q}$ be the unique member of $W_1(a_{j+1})$ such that $q\le_{\delta} \bar{q}$. 
Now, since $a_{j+1}$ diagonalizes $\langle b_\xi\mid \xi<\chi \rangle$,  we may let $\xi<\chi$ be such that  $\bar{q}\le^0_\delta  b_\xi$. 
As $b_\xi\le_\delta a_j$, $\beta+1\in B_{b_\xi}\setminus 2$.
Also, since $b_\xi\in D^{\epsilon_j}_{\delta,n+1}\s \z{P}_{\delta}$, Lemma~\ref{ringcoheres} yields $(b_\xi\restriction\beta+1)\in\z{P}_{\beta+1}$. 
In particular,  $\mtp_{\beta+1}(b_\xi\restriction\beta+1)=0$. Equivalently,\footnote{See Definition~\ref{labeled-p-tree}\eqref{C4ptree} and Definition~\ref{defntypes}.} 
$$(\sigma^{b_\xi, \beta}(q'),b_\xi\restriction\beta)\in R_{\beta}\;\; \text{for all }q'\in W(b_\xi\restriction\beta).$$ 
Since $q\restriction\beta \le_\beta\bar{q}\restriction\beta\le^0_{\beta} b_\xi\restriction\beta$, $(\sigma^{b_\xi, \beta}(b_\xi\restriction\beta), q\restriction\beta)\in R_\beta$. 
Therefore,   
$$q\restriction\beta \forces_{(\mathbb{P}_\beta)_{\ell_\delta(q)}} \sigma^{b_\xi, \beta}(b_\xi\restriction\beta)\in \dot{C}^\beta_{\ell_\delta(q)}.$$  
Finally, as $\epsilon_j<\sigma^{b_\xi, \beta}(b_\xi\restriction\beta)\leq \sup_{\xi<\chi}\sigma^{b_\xi}<\epsilon_{j+1}$, the conclusion follows.\qedhere
\end{enumerate}
\end{proof}

Let $\langle(a_j,\epsilon_j)\mid j<\omega\rangle$ be given by the preceding claim.
Set  $\epsilon_\omega:=\sup_{j<\omega}\epsilon_j$.
Our second step is to construct a condition $b\in\z{P}_{\delta}$ such that $b\le^0_\delta a_j$ for all $j<\omega$,
and such that $\sigma^{b,\beta}(b\restriction\beta)=\epsilon_\omega$ for all $\beta$ with $\beta+1\in B_b\setminus2$.
By Clause~\ref{point}, $b$ will satisfy $b\in D^{\epsilon}_{\delta,n}$ and $b\le^0_\delta a$.

Here we go.
Let $\langle \gamma_\tau\mid \tau<\theta\rangle$ be  the increasing enumeration    of $\bigcup_{j<\omega}B_{a_j}\setminus 2$.
For each $\tau<\theta$, denote by $\beta_\tau$ the predecessor of $\gamma_\tau$.
The sought condition will be obtained as the limit  $b:=(\bigcup_{\tau<\theta} b_\tau)*\emptyset_\delta$ of a sequence of conditions  $\langle b_\tau\mid \tau<\theta\rangle$  such that, for every $\tau<\theta$, all of the following will hold:
\begin{enumerate}[label=(\alph*)]
\item\label{Clause1Induction} $b_\tau\in(\z{P}_{\gamma_\tau})$ and $B_{b_\tau}\setminus 2=\{\gamma_\varrho\mid \varrho\leq \tau\}$; 
\item\label{Clause4Induction} $\sigma^{b_\tau,\beta}(b_\tau\restriction\beta)=\epsilon_\omega$ for all $\beta+1\in B_{b_\tau}$;
\item\label{Clause2Induction} $b_\tau\le_{\gamma_\tau}^0 a_j\restriction\gamma_\tau$ for all $j<\omega$;
\item\label{Clause3Induction} $b_\tau\restriction\gamma_\varrho=b_\varrho$ for all $\varrho\leq \tau$.
\end{enumerate}

We now turn to the recursive construction of the sequence $\langle b_\tau\mid \tau<\theta\rangle$,
starting with the case $\tau:=0$.
Recall that by our \blkref{I}, $(\mathbb{P}_1)_n=(\z{\mathbb{P}}_1)_n$.
So $\langle a_j\restriction 1\mid n<\omega\rangle$ is a decreasing sequence of conditions  in $(\z{\mathbb{P}}_1)_n$,
and hence  Definition~\ref{SigmaPrikry}\eqref{c2} yields a $\le^0_1$-lower bound $p\in (P_1)_n$ for it.
For each $j<\omega$, set $$c_j:=\fork[\gamma_0,\beta_0]{a_j\restriction\gamma_0}(\myceil{p}{\mathbb{P}_{\beta_0}}).$$ 
Since $\myceil{p}{\mathbb{P}_{\beta_0}}\le^0_{\beta_0} \myceil{a_j\restriction 1}{\mathbb{P}_{\beta_0}}$, it follows that $c_j\le^0_{\gamma_0}(a_j\restriction\gamma_0)$ for each $j<\omega$.   
By Clauses~\eqref{C5LemmaOfForking} and \eqref{Cviiforking} of Lemma~\ref{CviIteration}, the sequence $\langle c_j\mid j<\omega\rangle$ is $\le_{\gamma_0}^{\pi_{\gamma_0,\beta_0}}$-decreasing,
hence --- as in the proof of \cite[Lemma~6.15]{partI} --- it is order-isomorphic to $(\omega, \ni)$.

Let $j_0<\omega$ denote the least index such that $\gamma_0\in B_{a_{j}}$ for every $j\geq j_0$. For each $j\geq j_0$, put $\vec{R}^j:=\vec{S}^{c_j,\beta_0}$ so that $c_j=\myceil{p}{\mathbb{P}_{\beta_0}}{}^\smallfrown\langle\vec{R}^j\rangle$. By  Definition~\ref{d45}\eqref{pitchfork},
$\dom(\vec{R}^j)=\dom(\vec{S}^{a_j,\beta_0})=\alpha^{a_j,\beta_0}+1$, and
\begin{equation}\label{correspondenceSj}
\tag{$\star$}R^{j}_i(q)=S^{a_j,\beta_0}_i(w(a_j\restriction \beta_0,q))\;\text{for all }i\leq \alpha^{a_j,\beta_0}\text{ and }q\in W(\myceil{p}{\mathbb{P}_{\beta_0}}).
\end{equation}

Set $\alpha_j:=\alpha^{a_j,\beta_0}$ for each $j\geq j_0$. 
We now define $b_0:=\myceil{p}{\mathbb{P}_{\beta_0}}{}^\smallfrown\langle\vec{S}\rangle$, where $\vec{S}$ is the sequence $\langle S_i\mid i\leq \alpha+1\rangle$ 
with $\alpha:=\sup_{j_0\leq j<\omega}\alpha_j$, defined according to the following casuistic:

$\br$ For $i<\alpha$, $S_i(q)$ is defined as the unique member of $$\{R^{j}_i(q)\mid j\geq j_0,\, \alpha_j\ge i\}.$$

$\br$ For $i=\alpha$, we distinguish several cases:

$\br\br$ If $\alpha=\alpha_j$ for some $j\geq j_0$, then $S_\alpha(q):= R^j_{\alpha_j}(q)$ for the least such $j$;

$\br\br$ If $S_i(q)=\emptyset$ for all $i<\alpha$, then we continue and let $S_{\alpha}(q):=\emptyset$;

$\br\br$ Otherwise, set  $S_\alpha(q):=\bigcup_{i<\alpha} S_i(q)\cup \{\varepsilon_q\}$, where
$$\varepsilon_q:=\sup\{\max(S_{i}(q))\mid i<\alpha,\, S_i(q)\neq\emptyset\}.$$

A moment's reflection makes it clear that 
\begin{equation}\label{blackbox}
\tag{$\boxtimes$}\varepsilon_q=\sup\{\max(R^j_{\alpha_j}(q))\mid j\geq j_0,\, R^j_{\alpha_j}(q)\neq\emptyset\}.
\end{equation}

$\br$ For $i=\alpha+1$, let $S_{\alpha+1}(q):=S_\alpha(q)\cup\{\epsilon_\omega\}$.

\begin{claim}\label{claim4243} $\vec{S}$ is a $\myceil{p}{\mathbb{P}_{\beta_0}}$-strategy,
and hence $b_0\in P_{\gamma_0}$.
In addition, $\mtp_{\gamma_0}(b_0)=m(S_{\alpha+1})=0$.
\end{claim}
\begin{proof}
Since for each $j<\omega$, $\vec{R}^{j}$ is a $\myceil{p}{\mathbb{P}_{\beta_0}}$-strategy,
we just need to verify that $S_\alpha$ and $S_{\alpha+1}$ are both labeled $\myceil{p}{\mathbb{P}_{\beta_0}}$-trees. Actually, since $\langle S_i(q)\mid i<\alpha\rangle$ is a weakly $\sq$-increasing sequence of closed sets it is enough to check Clauses~\eqref{C3ptree} and \eqref{C4ptree} of Definition~\ref{labeled-p-tree}.  We commence with checking this for $S_\alpha$. The proof for $S_{\alpha+1}$ will be straightforward once we are done with that for $S_\alpha$. 

\medskip

\underline{Clause~\eqref{C3ptree} for $S_\alpha$:} Let $q\in W(\myceil{p}{\mathbb{P}_{\beta_0}})$ and to avoid trivialities assume that $S_\alpha(q)\neq \emptyset$. For each $j\geq j_0$, Clause~\eqref{C3ptree} for  $\vec{R}^{j}$ yields  $$q\forces_{\mathbb{P}_{\gamma_0}}R^j_{\alpha_j}(q)\cap {(\dot{T}^+)}^{\beta_0}=\emptyset,$$
hence it is enough to address the case where $S_\alpha(q)=(\bigcup_{i<\alpha} S_i(q))\cup\{\varepsilon_q\}$. To establish the clause it will be enough to show that  $q\forces_{\mathbb{P}_{\gamma_0}}\varepsilon_q\notin (\dot{T}^+)^{\beta_0}$.   

Set $\vec{\sigma}:=\langle \sigma_j\mid j\geq j_0\rangle$, where $\sigma_j:=\sigma^{c_{j}, \beta_0}(\myceil{p}{\mathbb{P}_{\beta_0}})$ for each $j\geq j_0$. 
\begin{subclaim}
If $q=\myceil{p}{\mathbb{P}_{\beta_0}}$ and $\vec{\sigma}$ is eventually constant,
then $\varepsilon_q=\max(\rng(\vec{\sigma}))$. Otherwise, $\varepsilon_q=\epsilon_\omega$.
\end{subclaim}
\begin{proof} $\br$ Suppose that $q=\myceil{p}{\mathbb{P}_{\beta_0}}$ and that $\vec{\sigma}$ is eventually constant.
By the definition of $\vec{S}$, for each $j\geq j_0$,  
$$\max(S_{\alpha_j}(\myceil{p}{\mathbb{P}_{\beta_0}}))=\max(R^j_{\alpha_j}(\myceil{p}{\mathbb{P}_{\beta_0}}))=\sigma_j.$$
Also, since $\vec{\sigma}$ is eventually constant there is $j^\star\geq j_0$ such that 
$$\max(R^j_{\alpha_j}(\myceil{p}{\mathbb{P}_{\beta_0}}))=\sigma_{j^\star}\;\;\text{for all $j\geq j^\star$}.$$
So, for each $j\geq j^\star$,  Clause~\eqref{C3pstrategy} of Definition~\ref{strategy} for $\vec{R}^{j+1}$ yields
$$S_{\alpha_j}(\myceil{p}{\mathbb{P}_{\beta_0}})=S_{\alpha_{j+1}}(\myceil{p}{\mathbb{P}_{\beta_0}}).$$ 
Thus, 
$S_\alpha(\myceil{p}{\mathbb{P}_{\beta_0}})=S_{\alpha_{j^\star}}(\myceil{p}{\mathbb{P}_{\beta_0}})$. 
Therefore,   $\varepsilon_{q}=\sigma_{j^\star}=\max(\rng(\vec{\sigma}))$.

\medskip

$\br$ Suppose that $q\in W_{\geq 1}(\myceil{p}{\mathbb{P}_{\beta_0}})$. 
Using Fact~\ref{factPartI},  let $\bar{q}\in W_{ 1}(\myceil{p}{\mathbb{P}_{\beta_0}})$ be such that $q\le_{\beta_0}\bar{q}$.  
Combining \eqref{correspondenceSj} with Clause~\ref{largerthandelta}, we have 
$$\max(R^j_{\alpha_j}(q))\leq \sigma^{a_{j}}< \epsilon_{j+1}\;\;\text{for all }j\geq j_0.$$
 
On the other hand, by Clause~\ref{largerthandelta} and Lemma~\ref{Remarkonsigmas}\eqref{Remarkonsigmas1},
 $$\epsilon_j< \sigma^{a_{j+1},\beta_0}(w(a_{j+1}\restriction\beta_0, \bar{q}))\leq \sigma^{a_{j+1},\beta_0}(w(a_{j+1}\restriction\beta_0, q)).$$
 Also, recalling how $c_{j+1}$ was defined, Lemma~\ref{Remarkonsigmas}\eqref{Remarkonsigmas2} yields
 $$\sigma^{a_{j+1},\beta_0}(w(a_{j+1}\restriction\beta_0, q))=\sigma^{c_{j+1},\beta_0}(q)=\max(R^{j+1}_{\alpha_{j+1}}(q)).$$
Thus,  $\epsilon_{j}< \max(R^{j+1}_{\alpha_{j+1}}(q))$. Combining the above we infer that $\varepsilon_q=\epsilon_\omega$. 

\medskip

$\br$  Suppose that $q=\myceil{p}{\mathbb{P}_{\beta_0}}$, but $\vec{\sigma}$ is not eventually constant.
Fix a nonzero $j\geq j_0$  such that $\sigma_j<\sigma_{j+1}$. Our first task is to prove that 
\begin{equation}\label{starstar}
\tag{$\star\star$}\sigma^{a_{j}, \beta_0}(w(a_{j}\restriction\beta_0, r))<\sigma_{j+1}\text{ for all }r\in W_1(a_{j+1}\restriction\beta_0).
\end{equation}

Note that by equation~\eqref{correspondenceSj} above, 
$$\max(R^{j}_{\alpha_j}(r))=\max(S^{a_j, \beta_0}_{\alpha_j}(w(a_j\restriction\beta_0,r)))=\sigma^{a_{j},\beta_0}(w(a_j\restriction\beta_0, r)),$$
hence all we need to prove is that $\sigma_{j+1}>\max(R^j_{\alpha_j}(r))$. Once we establish  \eqref{starstar}, we will be able to show that $\varepsilon_q=\epsilon_\omega$.

So, fix $r\in W_1(a_{j+1}\restriction\beta_0)$ and let us look at the sequence $$\langle \max(R^{j+1}_i(q))\mid \alpha_j\leq i\leq \alpha_{j+1}\rangle.$$
By Definition~\ref{strategy}\eqref{C3ptree} for $\vec{R}^{j+1}$, the above sequence is weakly increasing. Also, note that the first value of this sequence is $\sigma_j$ and the last one is $\sigma_{j+1}$.

Let $i^\star\leq \alpha_{j+1}$ be the first index such that $\max(R^{j+1}_i(q))=\sigma_{j+1}$ for all $i\in[ i^\star, \alpha^{a_{j+1},\beta_0}]$.  Since
by assumption $\sigma_j<\sigma_{j+1}$ note that  $\alpha_j<i^\star$.

$\br\br$ If $i^\star$ takes the form $k+1$, then Definition~\ref{strategy}\eqref{C4pstrategy} for $\vec{R}^{j+1}$ 
yields
$$R^{j+1}_{i^\star}(q)\setminus R^{j+1}_{k}(q)\sq R^{j+1}_{i^\star}(r)\setminus R^{j+1}_{k}(r).$$
By minimality of $i^\star$, $\sigma_{j+1}=\max(R^{j+1}_{i^\star}(q))>\max(R^{j+1}_k(q))$. In particular, $\sigma_{j+1}$ is a member of the left-hand-side of the above expression and thus $\sigma_{j+1}>\max(R^{j+1}_k(r))$. Since $\alpha_j<i^\star$, then $\alpha_j\leq k$ and so $$\sigma_{j+1}>\max(R^{j+1}_{k}(r))\geq \max(R^{j+1}_{\alpha_j}(r))=\max(R^{j}_{\alpha_j}(r)).$$

$\br\br$ If $i^\star$ is a limit then Clause~\eqref{C5pstrategy} of Definition~\ref{strategy} yields 
$$R^{j+1}_{i^\star}(q)=\bigcup_{k<i^\star} R^{j+1}_{k}(q)\cup\{\sigma_{j+1}\}.$$
By minimality of $i^\star$, unboundedly many $k\in (\alpha_j, i^\star)$ must satisfy
$$\max(R^{j+1}_k(q))<\max(R^{j+1}_{k+1}(q))<\sigma_{j+1}.$$
Thus, again by  Clauses~\eqref{C3pstrategy} and \eqref{C4pstrategy}  of Definition~\ref{strategy} for $\vec{R}^{j+1}$, $$\max(R^{j}_{\alpha_j}(r))=\max(R^{j+1}_{\alpha_j}(r))\leq \max(R^{j+1}_{k}(r))<\max(R^{j+1}_{k+1}(q))< \sigma_{j+1}.$$

The above discussion yields \eqref{starstar}. 
Now, combining \eqref{starstar} with  Clause~\ref{largerthandelta} for $a_j$ we get that
$$\epsilon_{j-1}<\sigma^{a_{j},\beta_0}(w(a_{j}\restriction\beta_0,r))<\sigma_{j+1}.$$

On the other hand, equation \eqref{correspondenceSj} yields $$R^{j+1}_{\alpha_{j+1}}(q)=S^{a_{j+1},\beta_0}_{\alpha_{j+1}}(w(a_{j+1}\restriction\beta_0,q)),$$ hence  Clause~\ref{largerthandelta} 
implies that $\sigma_{j+1}\leq \sigma^{a_{j+1}}\leq \epsilon_{j+1}$. 
Recalling the expression displayed in \eqref{blackbox},
we altogether infer that $\epsilon_\omega=\sup_{j\geq j_0}\sigma_j=\varepsilon_q$.
\end{proof}

We are now in conditions to show that $q\forces_{\mathbb{P}_{\gamma_0}}\varepsilon_q\notin (\dot{T}^+)^{\beta_0}$. 

Suppose first that $q=\myceil{p}{\mathbb{P}_0}$ and that $\vec{\sigma}$ is eventually constant. Let $j\geq j_0$ be such that $\max(\rng(\vec{\sigma}))=\sigma_j$.  By the preceding Subclaim, $\varepsilon_q=\sigma_j$. 
Appealing to Clause~\eqref{C2ptree} of Definition~\ref{labeled-p-tree} for $R^j$ we have $\sigma_j\in R^j_{\alpha_j}(\myceil{p}{\mathbb{P}_{\beta_0}})$, hence $\varepsilon_q\in  R^j_{\alpha_j}(\myceil{p}{\mathbb{P}_{\beta_0}})$. Finally, Clause~\eqref{C3ptree} of Definition~\ref{labeled-p-tree} for $R^j_{\alpha_j}$ yields $q\forces_{\mathbb{P}_{\gamma_0}} R^j_{\alpha_j}(q)\cap (\dot{T}^+)^{\beta_0}=\emptyset$ and thus 
$q\forces_{\mathbb{P}_{\gamma_0}}\varepsilon_q\notin (\dot{T}^+)^{\beta_0}.$

\medskip

Now, suppose we are in the other case, so that $\varepsilon_q=\epsilon_\omega$. In this case, instead of proving $q\forces_{\mathbb{P}_{\gamma_0}}\varepsilon_q\notin(\dot{T}^+)^{\beta_0}$ we will moreover prove that $(\epsilon_\omega, \myceil{p}{\mathbb{P}_{\beta_0}})\in R_{\beta_0}$. This  will become handy later on when verifying  Clause~ \eqref{C4ptree} for $S_\alpha$ and Clauses~\eqref{C3ptree} and \eqref{C4ptree} for $S_{\alpha+1}$.

By Remark~\ref{additionalfeatureofT+}, $(\epsilon_\omega, \myceil{p}{\mathbb{P}_{\beta_0}})\in R_{\beta_0}$ amounts to asserting that $r\forces_{{(\mathbb{P}_{\beta_0})}_{\ell_{\beta_0}(r)}}\epsilon_\omega\in \dot{C}^{\beta_0}_{\ell_{\beta_0}(r)}$  for all $r\le_{\beta_0} \myceil{p}{\mathbb{P}_{\beta_0}}$.\footnote{Here note that we are (crucially) using that in \blkref{II} we appeal to Corollary~\ref{onestep} with respect to the condition $\one_{\mathbb{P}_{\beta_0}}$.} 
To this end, let $r\le_{\beta_0}\myceil{p}{\mathbb{P}_{\beta_0}}$.

\begin{itemize}
\item[$\br$]	 If $r\in (P_{\beta_0})^{\myceil{p}{\mathbb{P}_{\beta_0}}}_{\geq 1}$,
then,  for each $j\geq j_0$,  $\fork[\delta,\beta_0]{a_{j+1}}(r)\in (P_\delta)_{\geq 1}^{a_{j+1}}$ and so there is some $c_{j+1}\in  W_{\geq 1}(a_{j+1})$ such that $\fork[\delta,\beta_0]{a_{j+1}}(r)\le^0_\delta c_{j+1}$. By Lemma~\ref{CviIteration}\eqref{C5LemmaOfForking}, $r\leq^0_{\beta_0} c_{j+1}\restriction\beta_0$.  Observe that $\gamma_0\in B_{a_{j}}$, hence Clause \ref{tryingtofallintothering} for $a_{j+1}$ yields
$$c_{j+1}\restriction\beta_0\forces_{(\mathbb{P}_{\beta_0})_{\ell_{\beta_0}(r)}}\dot{C}^{\beta_0}_{\ell_{\beta_0}(r)}\cap (\check{\epsilon}_{j},\check{\epsilon}_{j+1})\neq \emptyset.$$ 
Consequently, for each $j\geq j_0$,  $r\forces_{(\mathbb{P}_{\beta_0})_{\ell_{\beta_0}(r)}}\dot{C}^{\beta_0}_{\ell_{\beta_0}(r)}\cap (\check{\epsilon}_{j},\check{\epsilon}_{j+1})\neq \emptyset.$ 

Finally, since  $r\forces_{(\mathbb{P}_{\beta_0})_{\ell_{\beta_0}(r)}}``\dot{C}^{\beta_0}_{\ell_{\beta_0}(r)}\text{ is a club}$'', one concludes that
$$r\forces_{(\mathbb{P}_{\beta_0})_{\ell_{\beta_0}(r)}}\check{\epsilon}_\omega\in \dot{C}^{\beta_0}_{\ell_{\beta_0}(r)}.$$

\item[$\br$] 	If   $r\in (P_{\beta_0})^{\myceil{p}{\mathbb{P}_{\beta_0}}}_{0}$, then we first claim that $\myceil{p}{\mathbb{P}_{\beta_0}}\forces_{\mathbb{P}_{\beta_0}}\check{\epsilon}_\omega\notin {(\dot{T}^+)}^{\beta_0}$.
Indeed, if this is not the case, then we may pick $s\in (P_{\beta_0})^{\myceil{p}{\mathbb{P}_{\beta_0}}}_{\geq 1}$ such that $s\forces_{\mathbb{P}_{\beta_0}}\check{\epsilon}_\omega\in {(\dot{T}^+)}^{\beta_0}$. By $\leq_{\beta_0}$-extending $s$ we may assume that $(\check{\epsilon}_\omega,s)\in (\dot{T}^+)^{\beta_0}$, so, by Definition~\ref{def424}, $(\check{\epsilon}_\omega,s)\in \tau^{\beta_0}_{\ell_{\beta_0}(s)}$.
In particular, $s\forces_{(\mathbb{P}_{\beta_0})_{\ell_{\beta_0}(s)}}\check{\epsilon}_\omega\in \tau^{\beta_0}_{\ell_{\beta_0}(s)}$.
Then, by Remark~\ref{RemarkAdditionalFeature}, $s\forces_{(\mathbb{P}_{\beta_0})_{\ell_{\beta_0}(s)}}\check{\epsilon}_\omega\notin {\dot{C}}^{\beta_0}_{\ell_{\beta_0}(s)}$,
contradicting the fact that $s\in (P_{\beta_0})^{\myceil{p}{\mathbb{P}_{\beta_0}}}_{\geq 1}$ and the analysis of the previous case,
replacing $r$ by $s$.

So, $ \myceil{p}{\mathbb{P}_{\beta_0}}\forces_{\mathbb{P}_{\beta_0}}\check{\epsilon}_\omega\notin {(\dot{T}^+)}^{\beta_0}$. In particular, 
$$\myceil{p}{\mathbb{P}_{\beta_0}}\forces_{(\mathbb{P}_{\beta_0})_{n}}\check{\epsilon}_\omega\notin {(\dot{T}^+)}^{\beta_0}_{n}.$$ 
Then, by Lemma~\ref{ThekeylemmaaboutT+}, $ \myceil{p}{\mathbb{P}_{\beta_0}}\forces_{(\mathbb{P}_{\beta_0})_{n}}\check{\epsilon}_\omega\notin {\tau}^{\beta_0}_{n}$ and 
$$\myceil{p}{\mathbb{P}_{\beta_0}}\forces_{(\mathbb{P}_{\beta_0})_{n}}\check{\epsilon}_\omega\notin (\check{E}^\mu_\omega\setminus \dot{C}^{\beta_0}_n).$$
Since $\epsilon_\omega\in (E^\mu_\omega)^V$, we must conclude that $\myceil{p}{\mathbb{P}_{\beta_0}}\forces_{(\mathbb{P}_{\beta_0})_{n}}\check{\epsilon}_\omega\in \dot{C}^{\beta_0}_n$. 
Consequently, the same is $(\mathbb{P}_{\beta_0})_{n}$-forced by $r$, as it extends $ \myceil{p}{\mathbb{P}_{\beta_0}}$.
\end{itemize}

Thus, we have proved that  
$(\varepsilon_q , \myceil{p}{\mathbb{P}_{\beta_0}})\in R_{\beta_0}$ for all $q\in W(\myceil{p}{\mathbb{P}_{\beta_0}})$. 

\smallskip

Note that at this point we have managed to establish that $q\forces_{\mathbb{P}}\varepsilon_q\notin (\dot{T}^+)^{\beta_0}$ for all $q\in W(\myceil{p}{\mathbb{P}_{\beta_0}})$. This completes the proof of Clause~\eqref{C3ptree} for $S_\alpha$.

\medskip

\underline{Clause~\eqref{C4ptree} for $S_{\alpha}$:}  Let $q'\le_{\mathbb{P}_{\beta_0}} q$ be a pair of members in  $ W(\myceil{p}{\mathbb{P}_{\beta_0}})$ with $\ell(q)\geq \ell(\myceil{p}{\mathbb{P}_{\beta_0}})+1$. Then $q'\neq\myceil{p}{\mathbb{P}_{\beta_0}}$ and so  the previous discussion yields $\varepsilon_q=\epsilon_\omega$ and $(\epsilon_\omega, \myceil{p}{\mathbb{P}_{\beta_0}})\in R_{\beta_0}$, hence $(\epsilon_\omega, q)\in R_{\beta_0}$, as well.

\medskip

\underline{Clauses~\eqref{C3ptree} $\&$ \eqref{C4ptree} of $S_{\alpha+1}$:}  By the preceding discussion it is clear that  $\max(S_{\alpha+1}(q))=\epsilon_\omega$. 
Also,  we have proved that $(\epsilon_\omega,\myceil{p}{\mathbb{P}_{\beta_0}})\in R_{\beta_0}$. Arguing similarly as above one can use this to prove Clauses~\eqref{C3ptree} and \eqref{C4ptree} for $S_{\alpha+1}$. 
Moreover, regarding Clause~\eqref{C4ptree}, $(\epsilon_\omega,\myceil{p}{\mathbb{P}_{\beta_0}})\in R_{\beta_0}$  yields $m(S_{\alpha+1})=0$.
\end{proof}

We are left with proving that $b_0$ witnesses Clauses~\ref{Clause1Induction}--\ref{Clause3Induction}.
To complete the proof of Clause~\ref{Clause1Induction} we still need to argue that $b_0\in \z{P}_{\gamma_0}$.
By the above claim, $b_0\in P_{\gamma_0}$ and $\mtp_{\gamma_0}(b_0)=m(S_{\alpha+1})=0$. Also, since  $p\in\z{P}_1$ and $\myceil{p}{\mathbb{P}_{\beta_0}}=p\ast\emptyset_{\beta_0}$, we may appeal to Lemma~\ref{LiftingAndRings} and infer that $b_0\restriction\beta_0=\myceil{p}{\mathbb{P}_{\beta_0}}\in\z{P}_{\beta_0}$.  So, recalling Definition~\ref{DefRingForLimits}, we conclude that $b_0\in\z{P}_{\gamma_0}$.
For Clause~\ref{Clause4Induction},  note that $B_{b_0}\setminus 2=\{\gamma_0\}$ and 
$$\sigma^{b_0,\beta_0}(b_0\restriction\beta_0)=\sigma^{b_0,\beta_0}(\myceil{p}{\mathbb{P}_{\beta_0}})=\epsilon_\omega.$$
Note that  $b_0$ is a condition in $\mathbb{P}_{\gamma_0}$ and also,  by construction, $b_0\le_{\gamma_0}^0 c_j\le^0_{\gamma_0} a_j\restriction\gamma_0$ for all $j<\omega$. Thus, Clause~\ref{Clause2Induction} holds.
Finally, Clause~\ref{Clause3Induction} is trivially true at this stage.	
Altogether, the constructed condition $b_0$ is as wanted. 

Next, suppose that $\tau<\theta$ and that  $\langle b_\eta\mid \eta<\tau\rangle$ has been already successfully defined. 
Put $b^*_\tau:=(\bigcup_{\eta<\tau} b_\eta)\ast\emptyset_{\beta_\tau}$.
\begin{claim}\label{b*inthering}
$b^*_\tau\in\z{P}_{\beta_\tau}$ and $b^*_\tau\le^0_{\beta_\tau} a_j\restriction\beta_\tau$ for all $j<\omega$. 
\end{claim}
\begin{proof}
Note that once we establish the former assertion the latter will follow automatically from Clauses~\ref{Clause2Induction} and \ref{Clause3Induction} of our induction hypothesis. 

$\br$ If $\tau$ takes the form $\bar{\eta}+1$ then Clauses~\ref{Clause1Induction} and \ref{Clause3Induction} of the induction hypothesis yield $\bigcup_{\eta<\tau}b_{\eta}=b_{\bar{\eta}}\in \z{P}_{\gamma_{\bar{\eta}}}$. Thus, Lemma~\ref{LiftingAndRings} yields $b^*_\tau\in \z{P}_{\beta_\tau}$. 

$\br$ Otherwise, set $\bar{\gamma}=\sup_{\eta<\tau}\gamma_\eta$ and note that $\bar{\gamma}\leq \beta_\tau$.  By Clause~\ref{Clause3Induction} of the induction hypothesis, $\bar{b}:=\bigcup_{\eta<\tau}b_{\eta}$ is a condition in $\mathbb{P}_{\bar{\gamma}}$. If we show that $\bar{b}\in \z{P}_{\bar{\gamma}}$ then Lemma~\ref{LiftingAndRings} will imply that $b^*_\tau\in \z{P}_{\beta_\tau}$ and we will be done. 

Let $\beta\in B_{\bar{b}}\cup\{1\}$ and $\eta<\tau$ be such that $\beta\in B_{b_\eta}$. By  Clause~\ref{Clause3Induction} of the induction hypothesis, $\bar{b}\restriction\beta=b_{\tau}\restriction\beta$, hence combining Clause~\ref{Clause1Induction} for $b_\tau$ with Lemma~\ref{ringcoheres} we get $\bar{b}\restriction\beta \in \z{P}_\beta$. 
Since $\bar{\gamma}$ is limit then  $\bar{b}\in\z{P}_{\bar{\gamma}}$. 
\end{proof}

For each $j<\omega$, set 
$c_j^\tau:=\fork[\gamma_\tau,\beta_\tau]{a_j\restriction\gamma_\tau}(b^*_\tau).$
By the preceding claim,  $c^\tau_j\in P_{\gamma_\tau}$ and $c^\tau_j\le^0_{\gamma_\tau} a_j\restriction\gamma_\tau$. 
Furthermore, $\langle c^\tau_j\mid j<\omega\rangle$ is $\le^{\pi_{\gamma_\tau,\beta_\tau}}_{\gamma_\tau}$-decreasing, so, as before, it is order-isomorphic to $(\omega,\ni)$.  
Let $j_\tau$ denote the least index such that $\gamma_\tau\in B_{a_j}$ for all $j\geq j_\tau$. For each $j\geq j_\tau$, put $\vec{R}^{j,\tau}:=\vec{S}^{c_j,\beta_\tau}$ so that $c^\tau_j=b^*_\tau{}^\smallfrown \langle\vec{R}^{j,\tau}\rangle$. 

Set $\alpha_j:=\alpha^{a_j,\beta_\tau}$ for each $j\geq j_\tau$.
We now define $b_\tau:=b^*_\tau{}^\smallfrown\langle\vec{S}\rangle$, where $\vec{S}$ is the sequence $\langle S_i\mid i\leq \alpha+1\rangle$
with $\alpha:=\sup_{j_\tau\leq j<\omega}\alpha_j$,  defined as in the case $\tau=0$ above just replacing $\vec{R}^j$ by $\vec{R}^{j,\tau}$.

We claim that $b_\tau$   witnesses Clauses~\ref{Clause1Induction}--\ref{Clause3Induction}: Arguing as in Claim~\ref{claim4243}, we get that $b_\tau\in P_{\gamma_\tau}$ and  $\mtp_{\gamma_\tau}(b_\tau)=m(S_{\alpha+1})=0$. Additionally, by Claim~\ref{b*inthering}, $b_\tau\restriction\beta_\tau=b^*_\tau\in\z{P}_{\beta_\tau}$, hence $b_\tau\in\z{P}_{\gamma_\tau}$. Finally, 
$$B_{b_\tau}\setminus 2=(B_{b^*_\tau}\setminus 2)\cup\{\gamma_\tau\}=\{\gamma_\eta\mid \eta\leq \tau\},$$ 
where the second equality comes from Clause~\ref{Clause1Induction} of the induction hypothesis.  

For the verification of Clause~\ref{Clause4Induction}, let us fix $\eta\leq \tau$. If  $\eta=\tau$ then one argues as in the case where $\tau=0$ that $\sigma^{b_\tau,\beta_\tau}(b_\tau\restriction\beta_\tau)=\sigma^{b_\tau,\beta_\tau}(b^*_\tau)=\epsilon_\omega$.  Otherwise, if $\eta<\tau$, Clauses~\ref{Clause4Induction} and \ref{Clause3Induction} of the  induction hypothesis yield
$$\sigma^{b_\tau,\beta_\eta}(b_\tau\restriction\beta_\eta)=\sigma^{b_\eta,\beta_\eta}(b_\eta\restriction\beta_\eta)=\epsilon_\omega.$$

Clause~\ref{Clause2Induction} follows   noting that $b_\tau\le_{\gamma_\tau}^0 c_j\le_{\gamma_\tau}^0 a_j\restriction\gamma_\tau$ for all $j<\omega$. Finally,  Clause~\ref{Clause3Induction} for $b_\tau$ is a consequence of $b_\tau\restriction\beta_\tau=b^*_\tau$ and Clause~\ref{Clause3Induction} of the induction hypothesis.

\medskip

Having constructed the sequence $\langle b_\tau\mid \tau<\theta\rangle$, as promised, we let $b:=(\bigcup_{\tau<\theta} b_\tau)\ast \emptyset_\delta$. 
By Clause~\ref{Clause3Induction}, $b\in P_\delta$, so, by Clauses~\ref{Clause1Induction} and \ref{Clause2Induction},
$b\le^0_\delta a_j$ for all $j<\omega$. 
Finally, we verify that $b\in D^{\epsilon}_{\delta,n}$. To this end, let $\beta$ with $\beta+1\in B_b\setminus 2$. Pick $\tau<\theta$ such that $\beta=\beta_\tau$. Appealing to Clause~\ref{Clause3Induction} for $b_{\tau+1}$,  $b\restriction \beta= b_{\tau+1}\restriction\beta_\tau= b_{\tau}\restriction\beta_\tau$, hence $$\sigma^{b,\beta}(b\restriction\beta)=\sigma^{b_\tau,\beta_\tau}(b_\tau\restriction\beta_\tau)=\epsilon_\omega>\epsilon.$$
Also, by Clause~\ref{Clause1Induction} for $b_\tau$,  $b\restriction\beta\in\z{P}_{\beta}$. Likewise,  by Lemma~\ref{ringcoheres}, $b\restriction 1\in\z{P}_1$. 
\end{proof}

Thanks to Lemma~\ref{ringisdense} we can now appeal to the iteration scheme of Section~\ref{Iteration} with respect to the building blocks of this section
and obtain, in return, a $\Sigma$-Prikry triple $(\mathbb P_{\mu^+},\lh_{\mu^+},c_{\mu^+})$.

\begin{theorem}\label{TheoremReflection} In $V^{\mathbb{P}_{\mu^+}}$ all of the following hold true:
\begin{enumerate}
\item Any cardinal in $V$ remains a cardinal and retains its cofinality;
\item $\kappa$ is a singular strong limit of countable cofinality;
\item $2^\kappa=\kappa^{++}$;
\item $\refl({<}\omega, \kappa^+)$.
\end{enumerate}
\end{theorem}
\begin{proof}
(1) By Fact~\ref{l14}(1), no cardinal $\le\kappa$ changes its cofinality;
by Fact~\ref{l14}(3), $\kappa^{+}$ is not collapsed,
and by Definition~\ref{SigmaPrikry}\eqref{c1}, no cardinal $>\kappa^+$ changes its cofinality.

(2) In $V$, $\kappa$ is a singular strong limit of countable cofinality,
and so by Fact~\ref{l14}(1), this remains valid in $V^{\mathbb{P}_{\mu^+}}$.

(3) In $V$, we have that $2^\kappa=\kappa^+$.
In addition, by Remark~\ref{remark43}(1), $\mathbb{P}_{\mu^+}$ is isomorphic to a subset of $H_{\mu^+}$,
so that, from $|H_{\mu^+}|=\kappa^{++}$,
we infer that $V^{\mathbb{P}_{\mu^+}}\models 2^\kappa\le\kappa^{++}$.
Finally, as $\mathbb P_{\mu^+}$ projects to $\mathbb P_1$ which is isomorphic to $\mathbb Q$,
we get that $V^{\mathbb{P}_{\mu^+}}\models 2^\kappa\ge\kappa^{++}$.
Altogether, $V^{\mathbb{P}_{\mu^+}}\models 2^\kappa=\kappa^{++}$.

(4) As $\kappa^+=\mu$ and $\kappa$ is singular,
$\refl({<}\omega, \kappa^+)$ is equivalent to $\refl({<}\omega, E^\mu_{<\kappa})$.
By Fact~\ref{c27}, we already know that $V^{\mathbb P_{\mu^+}}\models \refl({<}\omega,\Gamma)$.
So, by Proposition~\ref{p3}, it suffices to verify that $\refl({<}2,(E^\mu_\omega)^V,\Gamma)$ holds in $V^{\mathbb P_{\mu^+}}$.

Let $G$ be $\mathbb P_{\mu^+}$-generic over $V$ and hereafter work within $V[G]$.
Towards a contradiction, suppose  that there exists a subset $T$ of $(E^\mu_\omega)^V$ that does not reflect in $\Gamma$. 
Fix $r^*\in G$ and a $\mathbb P_{\mu^+}$-name $\tau$ such that $\tau_G$ is equal to such a $T$
and such that $r^*$ forces $\tau$ to be a stationary subset of $(E^\mu_\omega)^V$
that does not reflect in $\Gamma$.
Furthermore, we may require that $\tau$ be a \emph{nice name}, i.e., each element of $\tau$ is a pair $(\check \xi,p)$ where $(\xi,p)\in (E^\mu_\omega)^V\times P_{\mu^+}$,
and, for all $\xi\in (E^\mu_\omega)^V$, the set $\{ p\mid (\check \xi,p)\in \tau\}$ is an antichain. 

As $\mathbb P_{\mu^+}$ satisfies Clause~ \eqref{c1} of Definition~\ref{SigmaPrikry}, $\mathbb P_{\mu^+}$ has the $\mu^+$-cc.
Consequently, there exists a large enough $\beta<\mu^+$ such that $$B_{r^*}\cup\bigcup\{ B_p\mid (\xi,p)\in\tau\}\s\beta.$$
Let $r:=r^*\restriction\beta$ and set $$\sigma:=\{(\xi,p\restriction\beta)\mid (\xi,p)\in\tau\}.$$
From the choice of \blkref{III}, we may find a large enough $\delta<\mu^+$ with $\delta>\beta$ such that $\psi(\delta)=(\beta,r,\sigma)$.
As $\beta<\delta$, $r\in P_\beta$ and $\sigma$ is a $\mathbb P_\beta$-name,
the definition of our iteration at step $\delta+1$ involves appealing to \blkref{II}
with $(\mathbb P_\delta,\lh_\delta,c_\delta)$,
$r^\star:=r*\emptyset_\delta$ and $z:=i^{\delta}_\beta(\sigma)$.
For any ordinal $\eta<\mu^+$, denote $G_\eta:=\pi_{\mu^+,\eta}[G]$.
By the choice of $\beta$, and as $\delta>\beta$, we have
$$\tau=\{ (\xi,p*\emptyset_{\mu^+})\mid (\xi,p)\in\sigma\}=\{ (\xi,p*\emptyset_{\mu^+})\mid (\xi,p)\in z\},$$
so that, in $V[G]$,
$$T=\tau_{G}=\sigma_{G_\beta}=z_{G_\delta}.$$
In addition, $r^*=r^\star*\emptyset_{\mu^+}$.

Finally, as $r^*$ forces $\tau$ is a stationary subset of $(E^\mu_\omega)^V$
that does not reflect in $\Gamma$,
$r^\star$ forces that $z$ is a stationary subset of $(E^\mu_\omega)^V$ that does not reflect in $\Gamma$.
So, since $\pi_{\mu^+,\delta+1}(r^*)=r^\star*\emptyset_{\delta+1}=\myceil{r^\star}{\mathbb P_{\delta+1}}$ is in $G_{\delta+1}$,
Clause~\ref{C8onestep} of \blkref{II} entails that, in $V[G_{\delta+1}]$,
there exists a club in $\mu$ which is disjoint from $T$. In particular, $T$ is nonstationary in $V[G]$,
contradicting its very choice.
\end{proof}

Thus, we arrive at the following strengthening of the theorem announced by Sharon in \cite{AS}. We remind the reader that, by Fact~\ref{factPartI}, the extent of reflection obtained is optimal.
\begin{cor}
Suppose that $\langle\kappa_n\mid n<\omega\rangle$ is an increasing sequence of supercompact cardinals,
converging to a cardinal $\kappa$. Then there exists a forcing extension where the following properties hold:
\begin{enumerate}
\item $\kappa$ is a singular strong limit cardinal of countable cofinality;
\item $2^\kappa=\kappa^{++}$, hence $\sch_\kappa$ fails;
\item $\refl({<}\omega,\kappa^+)$ holds.
\end{enumerate}
\end{cor} 
\begin{proof}
Let $\mathbb{L}$ be the inverse limit of the iteration $\langle \mathbb{L}_n; \dot{\mathbb{Q}}_n\mid n<\omega\rangle$, where $\mathbb{L}_0$ is the trivial forcing and for positive integer $n$,
if $\one_{\mathbb{L}_n}\forces_{\mathbb{L}_n}``\kappa_{n-1}\text{ is supercompact}"$,
then $\one_{\mathbb{L}_n}\forces_{\mathbb{L}_n}``\dot{\mathbb{Q}}_n\text{ is a Laver preparation}$ $\text{for } \kappa_n\text{ above }\kappa_{n-1}".$
After forcing with $\mathbb{L}$, each $\kappa_n$ remains supercompact and, moreover, becomes indestructible under $\kappa_n$-directed-closed forcing.
Also, the cardinals and cofinalities of interest are preserved.

Working in $V^\mathbb{L}$,  set
$\mu:=\kappa^+$, $\lambda:=\kappa^{++}$ and $\mathbb{C}:=\mathrm{Add}(\lambda,1)$.
Finally, work in $W:=V^{\mathbb{L}\ast \dot{\mathbb{C}}}$.
Since $\kappa$ is singular strong limit of cofinality $\omega<\kappa_0$ and $\kappa_0$ is supercompact,
$2^\kappa=\kappa^+$. Also, thanks to the forcing $\mathbb C$, $2^\mu=\mu^+$.
Altogether, in $W$, all the following hold:
\begin{itemize}
\item $\langle\kappa_n\mid n<\omega\rangle$ is an increasing sequence of Laver-Indestructible supercompact cardinals;
\item $\kappa:=\sup_{n<\omega}\kappa_n$, $\mu:=\kappa^+$ and $\lambda:=\kappa^{++}$;
\item $2^\kappa=\kappa^+$ and $2^\mu=\mu^+$.
\end{itemize}
Now, appeal to Theorem~\ref{TheoremReflection}.
\end{proof}


\begin{thebibliography}{PRS19}

\bibitem[Git96]{Git-short}
Moti Gitik.
\newblock Blowing up the power of a singular cardinal.
\newblock {\em Ann. Pure Appl. Logic}, 80(1):17--33, 1996.

\bibitem[Git10]{Gitik-handbook}
Moti Gitik.
\newblock Prikry-type forcings.
\newblock In {\em Handbook of set theory. {V}ols. 1, 2, 3}, pages 1351--1447.
  Springer, Dordrecht, 2010.

\bibitem[GM94]{Git-Mag}
Moti Gitik and Menachem Magidor.
\newblock Extender based forcings.
\newblock {\em J. Symbolic Logic}, 59(2):445--460, 1994.

\bibitem[Pov20]{Pov}
Alejandro Poveda.
\newblock {\em Contributions to the theory of {L}arge {C}ardinals through the method of {F}orcing}.
\newblock 2020.
\newblock Thesis (Ph.D.)--Universitat de Barcelona.

\bibitem[PRS21]{partI}
Alejandro Poveda, Assaf Rinot, and Dima Sinapova.
\newblock Sigma-{P}rikry forcing {I}: The {A}xioms.
\newblock {\em Canad. J. Math.}, 73(5): 1205--1238, 2021.

\bibitem[Ros18]{roslanowski}
Andrzej Roslanowski.
\newblock Explicit example of collapsing $\kappa^+$ in iteration of
  $\kappa$-proper forcings.
\newblock {\em arXiv preprint arXiv:1808.01636}, 2018.

\bibitem[Sha05]{AS}
Assaf Sharon.
\newblock {\em Weak squares, scales, stationary reflection and the failure of
  {S}{C}{H}}.
\newblock 2005.
\newblock Thesis (Ph.D.)--Tel Aviv University.

\bibitem[Tod10]{MR2603812}
Stevo Todorcevic.
\newblock {\em Introduction to {R}amsey spaces}, volume 174 of {\em Annals of
  Mathematics Studies}.
\newblock Princeton University Press, Princeton, NJ, 2010.

\end{thebibliography}
\end{document}